\documentclass[Ar4paper, reqno]{amsart}

\title[{Duality of subregular $\mathcal{W}$-algebras and principal $\mathcal W$-superalgebras}]{Duality of subregular $\mathcal{W}$-algebras and principal $\mathcal W$-superalgebras}

\author{Thomas Creutzig}
\address{Department of Mathematical and Statistical Sciences, University of Alberta, 632 CAB, Edmonton, Alberta, Canada T6G 2G1}
\email{creutzig@ualberta.ca}
\author{Naoki Genra}
\address{Department of Mathematical and Statistical Sciences, University of Alberta, 632 CAB, Edmonton, Alberta, Canada T6G 2G1}
\email{genra@ualberta.ca}
\author{Shigenori Nakatsuka}
\address{Graduate School of Mathematical Sciences, The University of Tokyo, 3-8-1 Komaba, Tokyo, Japan 153-8914}
\email{nakatuka@ms.u-tokyo.ac.jp}
\address{Kavli Institute for the Physics and Mathematics of the Universe (WPI), The University of Tokyo Institutes for Advanced Study, The University of Tokyo, Kashiwa, Chiba 277-8583, Japan}
\email{shigenori.nakatsuka@ipmu.jp}

\usepackage{latexsym}
\usepackage{amsmath}
\usepackage{amsfonts}
\usepackage{amssymb}
\usepackage{amsxtra}
\usepackage{mathrsfs}
\usepackage{color}
\usepackage[all]{xy}
\usepackage{etoolbox, textcomp}
\usepackage[dvipdfmx,colorlinks=true,citecolor=red,linkcolor=blue]{hyperref}

\newtheorem{definition}{Definition}[section]
\newtheorem{proposition}[definition]{Proposition}
\newtheorem{theorem}[definition]{Theorem}
\newtheorem{corollary}[definition]{Corollary}
\newtheorem{lemma}[definition]{Lemma}

\newtheorem{conjecture}[definition]{Conjecture}
\newtheorem{assumption}[definition]{Assumption}

\numberwithin{equation}{section}

\newcommand{\Z}{\mathbb{Z}}
\newcommand{\g}{\mathfrak{g}}
\newcommand{\C}{\mathbb{C}}
\newcommand{\End}{\operatorname{End}}
\newcommand{\Hom}{\operatorname{Hom}}
\newcommand{\Span}{\operatorname{Span}}
\newcommand{\Ker}{\operatorname{Ker}}
\newcommand{\conf}{\operatorname{\Delta}}
\newcommand{\wt}{\operatorname{wt}}
\newcommand{\voa}{vertex operator algebra}

\begin{document}
\begin{flushright}
IPMU-20-0062
\end{flushright}
\maketitle

\begin{abstract} 
We prove Feigin-Frenkel type dualities between subregular $\mathcal W$-algebras of type $A, B$ and principal $\mathcal W$-superalgebras of type $\mathfrak{sl}(1|n), \mathfrak{osp}(2|2n)$. The type $A$ case proves a conjecture of Feigin and Semikhatov. 

Let $(\mathfrak{g}_1,\mathfrak{g}_2) = (\mathfrak{sl}_{n+1},\mathfrak{sl}(1|n+1))$ or $(\mathfrak{so}_{2n+1}, \mathfrak{osp}(2|2n))$ and let $r$ be the lacity of $\mathfrak{g}_1$. Let $k$ be a complex number and $\ell$ defined by
$r(k+h^\vee_1)(\ell+h^\vee_2)=1$ with $h^\vee_i$ the dual Coxeter numbers of the $\mathfrak g_i$.
Our first main result is that the Heisenberg cosets $\mathcal C^k(\mathfrak g_1)$ and  $\mathcal C^\ell(\mathfrak g_2)$ of these $\mathcal W$-algebras at these dual levels are isomorphic, i.e. $\mathcal C^k(\mathfrak g_1) \simeq \mathcal C^\ell(\mathfrak g_2)$ for generic $k$. We determine the generic levels and furthermore establish analogous results for the cosets of the simple quotients of the $\mathcal W$-algebras. 

Our second result is a novel Kazama-Suzuki type coset construction: We show that a diagonal Heisenberg coset of the subregular $\mathcal W$-algebra at level $k$ times the lattice vertex superalgebra $V_{\mathbb Z}$
is  the principal $\mathcal W$-superalgebra at the dual level $\ell$. Conversely  a diagonal Heisenberg coset of the principal $\mathcal W$-superalgebra at level $\ell$ times the lattice vertex superalgebra $V_{\sqrt{-1}\mathbb Z}$
is the subregular $\mathcal W$-algebra at the dual level $k$. Again this is proven for the universal $\mathcal W$-algebras as well as for the simple quotients.

We show that a consequence of the Kazama-Suzuki type construction is that the simple principal $\mathcal W$-superalgebra and its Heisenberg coset at level $\ell$ are rational and/or $C_2$-cofinite if the same is true for the simple subregular $\mathcal W$-algebra at dual level $\ell$. This gives many new $C_2$-cofiniteness and rationality results. 
\end{abstract}

\markboth{Duality of subregular $\mathcal{W}$-algebras and principal $\mathcal W$-superalgebras}{Duality of subregular $\mathcal{W}$-algebras and principal $\mathcal W$-superalgebras}

\section{Introduction}

Let $\g$ be a simple Lie superalgebra, $k$ a complex number and $f$ a nilpotent element in $\g$ that belongs to an $\mathfrak{sl}_2$-triple. Then one associates to the universal affine vertex operator superalgebra $V^k(\g)$ at level $k$ via quantum Hamiltonian reduction the universal $\mathcal W$-superalgebra $\mathcal W^k(\g, f)$ \cite{FF3, KRW}. Its unique simple quotient is denoted by $\mathcal W_k(\g, f)$. If $f$ is principal nilpotent, then one often omits mentioning $f$ and we do so as well. Affine vertex superalgebras and their $\mathcal W$-superalgebras are most important families of vertex superalgebras due to their essential role in various aspects of representation theory, geometry and physics.
This ranges from the geometric and quantum geometric Langlands correspondence \cite{F2, AFO} and topological invariants of 4-manifolds \cite{FeGu} to meaningful invariants of three and four dimensional supersymmetric quantum field theories \cite{AGT, BMR,CG,FrGa, GR} or symmetries of six dimensional conformal field theories \cite{BRvR}. The famous Feigin-Frenkel duality \cite{FF1} of principal $\mathcal W$-algebras asserts that for all non-critical levels $k$ in $\mathbb C$ one has (see \cite[Theorem 5.6]{ACL2} for a short proof)
\[
\mathcal W^k(\g) \simeq \mathcal W^\ell({}^L\g),
\]
where ${}^L\g$ is the Langlands dual Lie algebra of $\g$ and the dual level $\ell$ is defined by
\begin{equation} \label{eq:duallevels}
r (k+h^\vee)(\ell+ {}^Lh^\vee) =1
\end{equation}
with $r$ the lacity of $\g$ and $h^\vee, {}^Lh^\vee$ the dual Coxeter numbers of $\g$ and ${}^L\g$ respectively.
$\mathcal W$-algebras corresponding to non-principal nilpotent elements usually have a non-trivial weight one subspace, i.e., they have an affine \voa{} as a subalgebra. In this instance we do not expect an immediate isomorphism of $\mathcal W$-algebras, but different $\mathcal W$-algebras might have coset subalgebras that are isomorphic. For example, the coset by the affine subalgebra of the large $N=4$ superconformal algebra, i.e., the minimal $\mathcal W$-superalgebra of $\mathfrak{d}(2, 1;\alpha)$, coincides with a diagonal coset of the tensor product of two affine vertex algebras of type $\mathfrak{sl}_2$ \cite{CFL}.
We are interested in such isomorphisms if $\g$ is either of type $A$ or $B$ and in this case we still expect that the levels of the involved $\mathcal W$-algebras satisfy a relation of the form \eqref{eq:duallevels}. 
However, we do not expect that this is an isomorphism between cosets of $\mathcal W$-algebras of Lie algebras that are Langlands dual to each other, but the role of the dual Lie algebra is replaced by some Lie superalgebra. The aim of this paper is to show that these expectations are true for very interesing $\mathcal W$-algebras. Our work is motivated by an impressive work of Boris Feigin and Alexei Semikhatov \cite{FS}.

\subsection{Feigin-Semikhatov $\mathcal W^{(2)}_n$-algebras}

Let $\g_1 =\mathfrak{sl}_n$ and $\g_2=\mathfrak{sl}(n|1)$. Then the subregular $\mathcal W$-algebra of $\g_1$ is strongly generated by even fields in conformal weights $1, 2, \dots, n-1$ together with two more fields of conformal weight $\frac{n}{2}$. The principal $\mathcal W$-superalgebra of $\g_2$ is strongly generated by even fields in conformal weight $1, 2, \dots, n$ together with a pair of odd fields of conformal weight $\frac{n+1}{2}$. In \cite{FS}, Feigin and Semikhatov constructed $\mathcal W^{(2)}_n$-algebras as subalgebras of the joint kernel of a set of screening charges acting on some free field algebras. These screening charges were associated to simple positive roots of $\g_2$. The subalgebra of the kernel of $\g_2$ screenings is interpreted as a $\mathcal W$-superalgebra of type $\g_2$ while the $\mathcal W^{(2)}_n$-algebra is generated under operator product by two even fields of conformal weight $\frac{n}{2}$. Next, also a vertex algebra generated by two weight $\frac{n}{2}$ fields is constructed inside the tensor product of the affine vertex superalgebra of $\g_2$ and a lattice vertex operator algebra. Many impressive operator product computations then suggest that the algebras appearing in these different constructions are isomorphic. These computations however do not prove any isomorphism and should be interpreted as the conjecture:
\begin{conjecture}\label{conjectureA} \textup{(Feigin and Semikhatov \cite{FS})}
Let $\g_1 =\mathfrak{sl}_n$, $\g_2=\mathfrak{sl}(n|1)$ and let $f_1$ be subregular nilpotent in $\g_1$ and $f_2$ be principal nilpotent in $\g_2$. 
Let $\pi_{H_i}$ be the Heisenberg subalgebra of $\mathcal{W}^{k_i} (\mathfrak{g}_i, f_i )$.
Then for generic $k_1$
\begin{enumerate}
\item $$\mathrm{Com}\left( \pi_{H_1}, \mathcal{W}^{k_1} (\mathfrak{g}_1, f_1) \right) \simeq \mathrm{Com}\left( \pi_{H_2}, \mathcal{W}^{k_2} (\mathfrak{g}_2) \right),$$
where $(k_1+n)(k_2+n-1)=1$.
\item $$\mathrm{Com}\left( \pi_{H_1}, \mathcal{W}^{k_1} (\mathfrak{g}_1, f_1)\right) \simeq \mathrm{Com}\left( V^{k_3}(\mathfrak{gl}_n), V^{k_3}(\g_2) \right),$$
where $\frac{1}{k_1+n} +\frac{1}{ k_3+n}=1$.
\end{enumerate}
\end{conjecture}
Note that for $n=2$, $\mathcal{W}^{k_1} (\mathfrak{g}_1, f_1)=V^{k_1}(\mathfrak{sl}_2)$ and $ \mathcal{W}^{k_2} (\mathfrak{g}_2)$ is the $N=2$ super Virasoro algebra. In this case the first part of the Conjecture is essentially the well-known Kazama-Suzuki coset realization of the $N=2$ super Virasoro algebra \cite{DPYZ} proven in \cite[Corollary 8.8]{CL3}. The second part of the Conjecture for $n=2$ follows from a nice relation between $V^k(\mathfrak{sl}(2|1))$ and $L_1(\mathfrak{d}(2, 1; -(k+1)))$ \cite{BFST, CG}. The second part of the Conjecture for $n=3$ is \cite[Theorem 6.2]{ACL1}. The proof used that in this case \cite{FS} explicitely computed all the necessary operator product algebras for the inclusion of $\mathrm{Com}\left( \pi_{H_1}, \mathcal{W}^{k_1} (\mathfrak{g}_1, f_1 )\right)$ in $\mathrm{Com}\left( V^{k_3}(\mathfrak{gl}_n), V^{k_3}(\g_2) \right)$. The surjectivity then followed from the computation of strong generators that can be done in an orbifold limit \cite{CL3} and that limit is determined in \cite[Theorem 4.3]{CL1}. We remark that the cases $n=2$ and $n=4$ can be handled similarly, but for larger $n$ not enough is known about operator product algebras and so there is need for a different proof strategy. 

In this work, we prove the first part of this Conjecture. The development of screening realizations of affine vertex superalgebras is work in progress. We believe that this will then allow us to also treat the second part of the Conjecture.
We note that a very different proof of the Conjecture is given in \cite{CL4}.

We actually not only prove the first part of the Conjecture, but also notice that a similar statement holds in the case of type $B$, namely Theorem \ref{thm:coset} and Corollaries \ref{cor:coset} and \ref{isom: simple quotients}.

Let $(\mathfrak{g}_1,\mathfrak{g}_2) = (\mathfrak{sl}_{n+1},\mathfrak{sl}(1|n+1))$ or $(\mathfrak{so}_{2n+1}, \mathfrak{osp}(2|2n))$. Denote by $r$ the lacity of $\mathfrak{g}_1$, (which is equal to that of $\mathfrak{g}_2$), and by $h^\vee_i$ the dual Coxeter number of $\mathfrak{g}_i$, namely,
\begin{align*}
(r,h^\vee_1,h^\vee_2)=
\begin{cases}
(1,n+1,n),& \mathrm{if}\ (\mathfrak{g}_1,\mathfrak{g}_2)=(\mathfrak{sl}_{n+1}, \mathfrak{sl}(1|n+1)),\\
(2,2n-1,n),& \mathrm{if}\ (\mathfrak{g}_1,\mathfrak{g}_2)=(\mathfrak{so}_{2n+1}, \mathfrak{osp}(2|2n)).
\end{cases}
\end{align*}
For $k_1\in \C\backslash \{-h^\vee_1\}$, define $k_2 \in \C$ by the formula
\begin{align}\label{introeq: ell}
r(k_1+h^\vee_1)(k_2+h^\vee_2)=1.
\end{align}
Set the rational numbers $x_i$, ($i=1,2$), by 
\begin{align}\label{eq: second critical level}
&(x_1,x_2)=
\begin{cases}
\left(\frac{1}{n}-n,-\frac{n^2}{n+1}\right),&\mathrm{if}\ (\mathfrak{g}_1,\mathfrak{g}_2)=(\mathfrak{sl}_{n+1}, \mathfrak{sl}(1|n+1)),\\
\left(2-2n,\frac{1}{2}-n\right),&\mathrm{if}\ (\mathfrak{g}_1,\mathfrak{g}_2)=(\mathfrak{so}_{2n+1}, \mathfrak{osp}(2|2n)),
\end{cases}
\end{align}
and $S_i= \{-h_i^\vee, x_i\}$. 
\begin{theorem}
Let $(\mathfrak{g}_1,\mathfrak{g}_2)=(\mathfrak{sl}_{n+1}, \mathfrak{sl}(1|n+1))$, or $(\mathfrak{so}_{2n+1}, \mathfrak{osp}(2|2n))$, and $(k_1,k_2)$ satisfy \eqref{introeq: ell}. Then 
 for $k_1 \in \mathbb C\setminus S_1$ and $k_2 \in \mathbb C\setminus S_2$,
$$\mathrm{Com}\left( \pi_{H_1}, \mathcal{W}^{k_1} (\mathfrak{g}_1, f_\mathrm{sub}) \right) \simeq \mathrm{Com}\left( \pi_{H_2}, \mathcal{W}^{k_2} (\mathfrak{g}_2) \right)$$
and also for the simple quotients
$$\mathrm{Com}\left( \pi_{H_1}, \mathcal{W}_{k_1} (\mathfrak{g}_1, f_\mathrm{sub}) \right) \simeq \mathrm{Com}\left( \pi_{H_2}, \mathcal{W}_{k_2} (\mathfrak{g}_2) \right).$$
\end{theorem}
There is also an analogue of the second part of Conjecture \ref{conjectureA} and that will be explained and proven in \cite{CL5}.

Our result can be improved to a new variant of Kazama-Suzuki cosets.

\subsection{Kazama-Suzuki cosets}

Kazama-Suzuki cosets appeared in the 1980's as building blocks of sigma models in string theory \cite{KaSu}. The idea is to consider $V^k(\mathfrak{sl}_{n+1})$ and tensor it with $n$ pairs of free fermions so that they carry the standard representation of $\mathfrak{gl}_n$ and its conjugate and so especially an action of $L_1(\mathfrak{gl}_n)$. The coset by the diagonal  $V^{k+1}(\mathfrak{gl}_{n})$-action then automatically has odd fields in conformal weight $\frac{3}{2}$ and actually gives rise to an extension of the $N=2$ super Virasoro algebra.
They are conjecturally isomorphic to principal $\mathcal W$-superalgebras of $\mathfrak{sl}(n+1|n)$ and the case $n=1$ is the just mentioned relation between the $N=2$ super Virasoro algebra and $V^k(\mathfrak{sl}_2)$. The case of $n=2$ is proven in \cite{GL} and strong rationality of the Kazama-Suzuki cosets in \cite[Corollary 14.1]{ACL2}. The idea of Kazama-Suzuki can be generalized in the following way. Consider some vertex operator algebra $\mathcal A^k$ with non-trivial action of $V^{k}(\mathfrak{gl}_{n})$, so that the tensor product of $\mathcal A^k$ with $n$-pairs of free fermions has a diagonal $V^{k+1}(\mathfrak{gl}_{n})$-action. Then the commutant by this diagonal action is our new variant of Kazama-Suzuki coset. For a recent related work, see \cite{Sato2}. We are interested in the case of $n=1$, i.e., $V^{k}(\mathfrak{gl}_{1})$ is nothing but a rank one Heisenberg vertex algebra. In this case, there is a remarkable observation due to Boris Feigin, Alexei Semikhatov and Ilya Tipunin that one can also somehow invert this coset construction \cite{FST}. This has been put to efficient use in studying the representation theory of the $N=2$ super Virasoro algebra and its relation to the one of the simple affine vertex algebra of $\mathfrak{sl}_2$, $L_k(\mathfrak{sl}_2)$ \cite{Ad2, Ad3, Sato1, KoSa, CLRW}. Also the relation between the $\beta\gamma$-system vertex algebra and $V^k(\mathfrak{gl}(1|1))$ can be viewed as the very first instance of this phenomenon and has also been used efficiently for new insights \cite{CR1, CR3, AP}.
Our second main theorem is the following Kazama-Suzuki type coset theorem and its inverse, see Theorem \ref{thm:KS} and Corollaries \ref{cor:coset} and \ref{isom: simple quotients}.
Set $K_i=\{-h_i^\vee\}$. Then we have the following.
\begin{theorem}  \label{cor:coset}
Let $(\mathfrak{g}_1,\mathfrak{g}_2)=(\mathfrak{sl}_{n+1}, \mathfrak{sl}(1|n+1))$, or $(\mathfrak{so}_{2n+1}, \mathfrak{osp}(2|2n))$, and $(k_1,k_2)$ satisfy \eqref{introeq: ell}, $k_1 \in \C\setminus K_1$ and $k_2 \in \mathbb C\setminus K_2$.
Then the Kazama-Suzuki type coset isomorphisms
\begin{equation}
\begin{split}
\mathcal{W}^{k_2} (\mathfrak{g}_2) &\simeq \mathrm{Com}\left( \pi_{\widetilde{H}_1},  \mathcal{W}^{k_1} (\mathfrak{g}_1, f_\mathrm{sub}) \otimes V_\Z \right),\\
\mathcal{W}_{k_2} (\mathfrak{g}_2) &\simeq \mathrm{Com}\left( \pi_{\widetilde{H}_1},  \mathcal{W}_{k_1} (\mathfrak{g}_1, f_\mathrm{sub}) \otimes V_\Z \right)
\end{split}
\end{equation}
and their inverses 
\begin{equation}
\begin{split}
\mathcal{W}^{k_1} (\mathfrak{g}_1, f_\mathrm{sub}) &\simeq \mathrm{Com}\left( \pi_{\widetilde{H}_2}, \mathcal{W}^{k_2} (\mathfrak{g}_2) \otimes V_{\Z\sqrt{-1}} \right),\\
\mathcal{W}_{k_1} (\mathfrak{g}_1, f_\mathrm{sub}) &\simeq \mathrm{Com}\left( \pi_{\widetilde{H}_2}, \mathcal{W}_{k_2} (\mathfrak{g}_2) \otimes V_{\Z\sqrt{-1}} \right)
\end{split}
\end{equation}
hold.
\end{theorem}
We expect that these isomorphisms will turn out to be very efficient in exploring the representation theory of the principal $\mathcal W$-superalgebras in terms of the one of the subregular $\mathcal W$-algebras. As a first important step, we deduce $C_2$-cofiniteness and rationality results. 

\subsection{Rationality and $C_2$-cofiniteness}

A well-known result of Tomoyuki Arakawa is that principal $\mathcal W$-algebras at non-degenerate admissible levels are $C_2$-cofinite \cite{Ar3} and rational \cite{Ar4}. The $C_2$-cofininteness of certain non-principal $\mathcal W$-algebras including subregular $\mathcal W$-algebras of type $A$ and $B$ at certain admissible levels has also been proven in \cite{Ar3}. They are also conjectured to be rational and this has been proven for type $A$ in \cite{AvE} (including earlier results \cite{Ar2, CL2} for subregular $\mathcal{W}$-algebras of type $\mathfrak{sl}_3$ and $\mathfrak{sl}_4$).

In contrast to this, the rationality and $C_2$-cofiniteness of principal $\mathcal W$-superalgebras and $\mathcal W$-superalgebras in general is open. We only found \cite{Ad3, Ad1} that used a coset construction to prove the rationality of the $N=1$ and $N=2$ super Virasoro algebras. Note that the $N=1$ super Virasoro algebra is the principal $\mathcal W$-superalgebra of $\mathfrak{osp}(1|2)$. 

We explain in Sections \ref{sec:C2} and \ref{sec:rational} how one can deduce $C_2$-cofininiteness results in the Kazama-Suzuki coset construction. This amounts to essentially showing that certain lattice vertex operator algebras appear as cosets. The key observation for this is our Lemma \ref{lemma:C2}.
The $C_2$-cofiniteness results follow from \cite[Lemma 4.6]{CKLR} (which is based on \cite{Mi}) modulo a lattice vertex operator algebra assumption. But the latter holds in our cases due to Lemma \ref{lemma:C2}. 
Rationality of Heisenberg cosets is \cite[Theorem 4.12]{CKLR}, again modulo a lattice vertex operator algebra assumption, but that assumption holds due to  some Jacobi form argument \cite{Mason}. Then Corollary \ref{cor:C2cofiniteness} and \ref{cor:rational} assert the following:
\begin{corollary} ${}$
\begin{enumerate}
\item
Let $\ell=-n +\frac{n}{u}$  with  $u\in \mathbb Z_{\geq n}$ and $(u, n)=1$. Then $\mathcal{W}_{\ell} (\mathfrak{sl}(1|n+1))$  and  $\mathrm{Com}\left( \pi_{H_2}, \mathcal{W}_{\ell} (\mathfrak{sl}(1|n+1)) \right)$ are rational and $C_2$-cofinite.
\item
Let $\ell=-n +\frac{2n-1}{2u}$  with  $u\in \mathbb Z_{\geq 2n}$ and $(u, 2n-1)=1$. Then $\mathcal{W}_{\ell} (\mathfrak{osp}(2|2n))$  and $\mathrm{Com}\left( \pi_{H_2}, \mathcal{W}_{\ell} (\mathfrak{osp}(2|2n)) \right)$ are $C_2$-cofinite.
\item
Let $\ell=-n +\frac{n}{u}$  with  $u\in \mathbb Z_{\geq 2n+1}$ and $(u, 2n)=1$. Then $\mathcal{W}_{\ell} (\mathfrak{osp}(2|2n))$ and $\mathrm{Com}\left( \pi_{H_2}, \mathcal{W}_{\ell} (\mathfrak{osp}(2|2n)) \right)$ are $C_2$-cofinite.
\end{enumerate}
\end{corollary}
Note that rationality of Heisenberg cosets for the cases $\mathfrak{sl}_3$ and $\mathfrak{sl}_4$ has been treated in \cite{ACL1, CL2}.

\subsection{$\mathcal W_{r_1, r_2, r_3}$-algebras}

A famous result of Schiffmann and Vasserot asserts that the principal $\mathcal W$-algebra of $\mathfrak{gl}_r$ acts on the equivariant cohomology of the moduli space of rank $r$ torsion free coherent sheaves on $\mathbb{CP}^2$ with some framing at $\infty$ \cite{SV}. This result was recently generalized to the moduli space of spiked instantons of Nekrasov \cite{RSYZ} and the resulting $\mathcal W_{r_1, r_2, r_3}$-algebra is characterized as the intersection of certain screening operators on a Heisenberg vertex algebra. The $\mathcal W_{r_1, r_2, r_3}$-algebra is defined in \cite{BFM} and physics \cite{PR} conjectures that it 
 is isomorphic to the $Y_{r_1, r_2, r_3}$-algebra of \cite{GR}. These $Y_{r_1, r_2, r_3}$-algebras are defined to be cosets of $\mathcal W$-algebras and superalgebras (times a Heinsenberg vertex algebra of rank one) and enjoy a triality of isomorphisms if one of the labels is zero \cite{CL4}. If two of the labels are zero, then the $Y_{r_1, r_2, r_3}$-algebra  and $\mathcal W_{r_1, r_2, r_3}$-algebra are a principal $\mathcal W$-algebra of $\mathfrak{gl}_r$ (where $r$ is the non-zero label). If one of the labels is zero, another one is one and the remaining one is $r$, then the resulting $Y_{r_1, r_2, r_3}$-algebra is the Heisenberg coset of the subregular $\mathcal W$-algebra of $\mathfrak{sl}_r$ (again times a rank one Heisenberg vertex algebra) and the screening realization of \cite{BFM} of the $\mathcal W_{r_1, r_2, r_3}$-algebra is precisely the one that we derive in this paper. This means that as a byproduct of our work, we also prove in this instance the conjecture that $Y_{r_1, r_2, r_3}$-algebras and  $\mathcal W_{r_1, r_2, r_3}$-algebras coincide.

\subsection{Applications}

There are a few possible applications. Firstly and as mentioned before, the Kazama-Suzuki coset relation between the simple affine vertex algebra $L_k(\mathfrak{sl}_2)$ and the simple $N=2$ super Virasoro algebra at central charge $3k/(k+2)$ has been used intensively to study the relation of the representation categories of the two  \cite{Ad2, Ad3, Sato1, KoSa, CLRW}. Similar studies can and should be accomplished for the better understanding of the representation theory of the principal $\mathcal W$-superalgebras. Of course the rational cases are the easiest one. We note that the $\mathcal B^{(p)}$-algebras of \cite{CRW} have recently been proven to be isomorphic to $\mathcal W_{2-p-p^{-1}}(\mathfrak{sl}_{p-1}, f_{\mathrm{sub}})$ \cite{ACGY}. These are at admissible levels and the representation theory has been considered in \cite{ACKR}. One can use these results to also study the representation theory of 
$\mathcal W_{3-p + (p-1)^{-1}}(\mathfrak{sl}(p-1|1))$. Beyond this there have appeared interesting realizations of the subregular $\mathcal W$-algebra at the critical level \cite{CGL1, CGL2, GK} together with some connection to geometry and physics. At the critical level the subregular $\mathcal W$-algebra has a large center.  Recall that in the case of affine vertex algebras at the critical level, the spectrum of the center is isomorphic to the space of opers on the formal disc of the Langlands dual Lie algebra \cite{F2}. This can be seen as Feigin-Frenkel duality at the critical level where the dual level goes to infinity. It is surely interesting to study the behaviour of our duality in the large level/critical level limit. Another related direction that is worth further investigation is to show these types of dualities for the corresponding finite $\mathcal W$-(super)algebras. 
In general we aim to investigate dualities between cosets of seemingly different $\mathcal W$-(super)algebras in best possible generality. Our present work is a  step in this direction. In order to derive further dualities with our methods we need a better understanding of free field realizations and screening charges of 
$\mathcal W$-superalgebras. A first step are those of affine vertex operator superalgebras which is work in progress. There are already works on this direction, see, for example, \cite{IK,R, IMP1, IMP2} for earlier results. Another aim is to prove that the $\mathcal W_{r_1, r_2, r_3}$-algebras and $Y_{r_1, r_2, r_3}$-algebras coincide. Given the geometrically meaning of the $\mathcal W_{r_1, r_2, r_3}$-algebras we consider this to be an important problem.
Free field realizations can also be used to study correspondences between conformal field theories with $\mathcal W$-algebra symmetries and our new free field realizations are used to derive correspondences between principal and subregular $\mathcal W$-algebras of type A and more \cite{CGHL}.

\subsection{Outline}

The proofs of the main theorems go in a few steps. Firstly, we need to characterize the $\mathcal W$-(super)algebras in terms of joint kernels of certain screening charges acting on some free fied algebras. For this, we utilize a free field realization of $V^{\kappa}(\mathfrak{gl}(1|1))$ developed in Section \ref{sec:gl11}. Section \ref{sec:Wfree} is then on free field realizations of the $\mathcal W$-(super)algebras that we study. 
In Section \ref{sec:newduality} we then relate the screening charges of the principal $\mathcal W$-superalgebra to the one of the corresponding subregular $\mathcal W$-algebra. This allows us to prove our main theorems for the univeral $\mathcal W$-(super)algebras at generic level. Next, in Section \ref{sec:Heisenbergcoset} we adapt the theory of Heisenberg cosets developed in \cite{CKLR} to our setting so that we can also determine the cosets of simple quotients and very importantly $C_2$-cofiniteness and rationality. For the latter two the most crucial step is to show that certain lattice vertex operator algebras appear. These results are then applied to our cosets. 

\subsection*{Acknowledgements}
T.C. appreciates many frutiful discussions on related topics with Boris Feigin and Andrew Linshaw. S.N. appreciates discussions on the simplicity of Heisneberg cosets with Yuto Moriwaki. We also thank Tomoyuki Arakawa for useful comments. 
T.C. is supported by NSERC $\#$RES0048511.
N. G is supported by JSPS Overseas Research Fellowships Grant Number 1120358.
S.N. is supported by the Program for Leading Graduate Schools, MEXT, Japan and by JSPS KAKENHI Grant Number 20J10147.
This work was supported by World Premier International Research Center Initiative (WPI Initiative), MEXT, Japan.
A part of this work is done while S.N. was staying at University of Alberta. He is grateful to the institute for the hospitality.

\section{Free field realization of $V^{\kappa}(\mathfrak{gl}(1|1))$}\label{sec:gl11}
We study a free field realization of the universal affine vertex superalgebra $V^\kappa(\mathfrak{gl}(1|1))$. We also realize it as the kernel of a certain screening operator.
\subsection{Heisenberg vertex algebra}\label{Sect: Heisenberg}

We follow the definition of vertex superalgebras in \cite{K2} and also use the language of the $\lambda$-bracket, cf.\ \cite{DK}. For a vertex superalgebra $V$, we denote by $| 0 \rangle$ the vacuum vector, by $\partial$ the translation operator, and by $Y(a,z)=a(z)=\sum_{n \in \Z} a_{(n)} z^{-n-1}$ the field corresponding to $a\in V$.

For a finite dimensional commutative Lie algebra $\mathfrak{h}$ over $\C$ with a symmetric bilinear form $\kappa$, we denote by $\widehat{\mathfrak{h}}_\kappa := \mathfrak{h}[t, t^{-1}] \oplus \C K$  the affine Lie algebra of $\mathfrak{h}$, whose Lie bracket is defined by
\begin{align*}
[h_{(m)}, h'_{(n)}] = m\kappa(h|h')\delta_{m+n, 0}K,\quad
[K, \widehat{\mathfrak{h}}_\kappa] =0,\quad h ,h' \in \mathfrak{h},\ m, n \in \Z,
\end{align*}
where $h_{(m)} = h t^m$. Define a $\widehat{\mathfrak{h}}_\kappa$-module $\pi^\kappa$ by
\begin{align*}
\pi^\kappa_\mathfrak{h} := U\left(\widehat{\mathfrak{h}}_\kappa\right) \otimes_{U\left(\mathfrak{h}[t] \oplus \C K\right)} \C,
\end{align*}
where $\C$ is regarded as a $\mathfrak{h}[t] \oplus \C K$-module, on which $\mathfrak{h}[t]$ acts trivially and $K$ acts as $1$, and $U(\mathfrak{a})$ denotes the universal enveloping algebra of a Lie superalgebra $\mathfrak{a}$. There is a unique vertex algebra structure on $\pi^\kappa_\mathfrak{h}$ with the vacuum vector $| 0 \rangle = 1 \otimes 1$ and
\begin{align*}
Y(h_{(-1)} |0 \rangle, z) = h(z) := \sum_{n \in \Z} h_{(n)} z^{-n-1},\quad
h \in \mathfrak{h}.
\end{align*}
It is called the Heisenberg vertex algebra associated with $\mathfrak{h}$ at level $\kappa$. More generally, for $\mu\in\mathfrak{h}^*$, define a $\widehat{\mathfrak{h}}_\kappa$-module by
\begin{align*}
\pi^\kappa_{\mathfrak{h}, \mu} := U\left(\widehat{\mathfrak{h}}_\kappa\right) \otimes_{U\left(\mathfrak{h}[t] \oplus \C K\right)} \C_{\mu},
\end{align*}
where $\C_\mu:=\C$ is regarded as a $\mathfrak{h}[t] \oplus \C K$-module, on which $h_{(n)}$, ($n \geq 0$), acts as $\mu(h)\delta_{n,0}$ and $K$ acts as $1$. It has a unique $\pi^\kappa_{\mathfrak{h}}$-module structure coming from the $\hat{\mathfrak{h}}_\kappa$-modules structure, called the highest weight $\pi^\kappa_{\mathfrak{h}}$-module with highest weight $\mu$. If $\kappa$ is non-degenerate, $\pi^\kappa_\mathfrak{h}$ is simple, called a non-degenerate Heisenberg vertex algebra. The dimension $\dim\mathfrak{h}$ of $\mathfrak{h}$ is equal to that of subspace of $\pi^\kappa_\mathfrak{h}$ with conformal degree $1$, called the rank of $\pi^\kappa_\mathfrak{h}$.

If we fix a non-degenerate bilinear form $(\cdot|\cdot)$ on $\mathfrak{h}$, then $\mathfrak{h}$ is identified with $\mathfrak{h}^*$ by $h \mapsto (\nu(h) \colon h' \mapsto (h|h'))$. For $\kappa=k(\cdot|\cdot)$, we write $\pi^k_\mathfrak{h}$, (resp.\ $\pi^k_{\mathfrak{h}, \mu}$) instead of $\pi^{\kappa}_\mathfrak{h}$, (resp.\ $\pi^{\kappa}_{\mathfrak{h}, \mu}$), and denote by $\alpha(z) = \sum_{n\in\Z}\alpha_{(n)}z^{-n-1} := \nu^{-1}(\alpha)(z)$ for $\alpha \in \mathfrak{h}^*$. We call $\pi^k_\mathfrak{h}$ the Heisenberg vertex algebra associated with $\mathfrak{h}$ at level $k$.

\subsection{Wakimoto representations of $\widehat{\mathfrak{gl}}(1|1)_\kappa$}\label{sec:Wak-gl11}

Let $\mathfrak{gl}(1|1)$ be the Lie superalgebra $\End(\C^{1|1})$ with Lie superbracket
\begin{align*}
[ x , y ] = x y - (-1)^{\bar{x} \bar{y}} y x,\quad
x, y \in \End(\C^{1|1}),
\end{align*}
where $\bar{x} \in \Z_2 = \{\bar{0}, \bar{1}\}$ denotes the parity of $x \in \End(\C^{1|1})$. Let $\{E_{i,j}\}_{1 \leq i,j \leq 2}$ denote the elementary matrices of $\mathfrak{gl}(1|1) = \End(\C^{1|1})$, $\mathfrak{h} = \C E_{1,1} \oplus \C E_{2,2}$, $\mathfrak{n}_+ = \C E_{1,2}$ and $\mathfrak{n}_- = \C E_{2,1}$. Note that the parity of $E_{i, j}$ is $\overline{i+j}$. For an even supersymmetric invariant bilinear form $\kappa$ on $\mathfrak{gl}(1|1)$, there exist unique $k_1, k_2 \in \C$ such that
\begin{align*}
\kappa = k_1 \kappa_1 + k_2 \kappa_2,
\end{align*}
where $\kappa_1 ( x | y )= \operatorname{str}_{\C^{1|1}}(xy)$ and $\kappa_2 ( x | y ) = -\frac{1}{2}\operatorname{str}_{\mathfrak{gl}(1|1)}\left(\operatorname{ad}(x)\operatorname{ad}(y)\right)$, and $\operatorname{str}_V(?)$ denotes the supertrace over a vector superspace $V$. The non-zero parings of $\kappa$ are:
\begin{align*}
&\kappa ( E_{1,1}|E_{1,1} ) = k_1 + k_2,\quad
\kappa ( E_{2,2}|E_{2,2} ) = -k_1 + k_2,\\
&\kappa ( E_{1,1}|E_{2,2} ) = - k_2,\quad
\kappa ( E_{1,2}|E_{2,1} ) = k_1.
\end{align*}
Define $\chi_i \in \mathfrak{h}^*$ by
\begin{align*}
\chi_i(E_{j, j}) = (-1)^{i+1} \delta_{i, j},\quad
i, j \in \{1, 2\}.
\end{align*}
We identify $\mathfrak{h}$ with $\mathfrak{h}^*$ by $\kappa_1$, under which $E_{i, i}$ is identified with $\chi_i$, ($i = 1,2$), since $\kappa_1(E_{i,i}|E_{j,j}) = \chi_i(E_{j, j})$. Let $\pi^{\kappa-\kappa_2} := \pi^{\kappa-\kappa_2}_\mathfrak{h}$ be the Heisenberg vertex algebra associated with $\mathfrak{h}$ at level $\kappa-\kappa_2$, which is freely generated by fields $\chi_i(z)$, ($i = 1, 2$), with OPEs
\begin{align}\label{eq:chi-OPE}
\chi_i(z)\chi_j(w) \sim \frac{(\kappa-\kappa_2)(E_{i, i}|E_{j, j})}{(z-w)^2},\quad
i, j=1, 2.
\end{align}

Let $\widehat{\mathfrak{gl}}(1|1)_\kappa:=\mathfrak{gl}(1|1) [t, t^{-1}]\oplus \C K$ be the affine Lie superalgebra with Lie superbracket
\begin{align*}
\begin{split}
&[x_{(m)},y_{(n)}]=[x,y]_{(m+n)}+m\delta_{m+n,0}\kappa(x,y)K,\quad x,y\in \mathfrak{gl}(1|1),\ m,n\in\Z,\\
&[K,\widehat{\mathfrak{gl}}(1|1)_\kappa]=0,
\end{split}
\end{align*}
where $x_{(m)}=x t^m$ for $x\in \mathfrak{gl}(1|1)$, $m\in\mathbb{Z}$. Define a $\widehat{\mathfrak{gl}}(1|1)_\kappa$-module $V^\kappa\left(\mathfrak{gl}(1|1)\right)$ by
\begin{align*}
V^\kappa\left(\mathfrak{gl}(1|1)\right) := U\left(\widehat{\mathfrak{gl}}(1|1)_\kappa\right) \otimes_{U\left(\widehat{\mathfrak{gl}}(1|1)_{\kappa,+}\right)} \C,
\end{align*}
where $\C$ is regarded as a $\widehat{\mathfrak{gl}}(1|1)_{\kappa,+}(:=\mathfrak{gl}(1|1) [t] \oplus \C K)$-module, on which $\mathfrak{gl}(1|1) [t]$ acts trivially and $K$ acts as $1$. Then there is a unique vertex superalgebra structure on $V^\kappa\left(\mathfrak{gl}(1|1)\right)$ with the vacuum vector $| 0 \rangle = 1 \otimes 1$ and
\begin{align*}
Y(u_{(-1)} |0 \rangle, z) = u(z) := \sum_{n \in \Z} u_{(n)} z^{-n-1},\quad
u \in \mathfrak{gl}(1|1).
\end{align*}
It is called the universal affine vertex superalgebra of $\mathfrak{gl}(1|1)$ at level $\kappa$. If $k_1\neq 0$, then it admits the Segal-Sugawara conformal field
\begin{align}\label{eq: conformal vector}
\begin{split}
T(z):=\frac{1}{2k_1}\Big(\frac{1-k_2}{k_1}:(&E_{1,1}(z)+E_{2,2}(z))^2:+:E_{1,1}(z)^2:-:E_{2,2}(z)^2:)\\
&-:E_{1,2}(z)E_{2,1}(z):+:E_{2,1}(z)E_{1,2}(z):\Big),
\end{split}
\end{align}
whose central charge is 0, i.e., the superdimension of $\mathfrak{gl}(1|1)$.

Let $M_{\mathfrak{gl}(1|1)}$ be the $bc$-system vertex superalgebra, which is generated by odd fields $b(z), c(z)$ satisfying the OPEs
\begin{align*}
b(z)c(w) \sim \frac{1}{z-w},\quad
b(z)b(w) \sim 0 \sim c(z)c(w).
\end{align*}
The following proposition follows from direct calculations.
\begin{proposition}\label{prop:Wakimoto-gl(1|1)}
There exists a homomorphism of vertex superalgebras\\ $\rho \colon V^\kappa\left(\mathfrak{gl}(1|1)\right) \rightarrow  M_{\mathfrak{gl}(1|1)} \otimes \pi^{\kappa - \kappa_2}$, which satisfies
\begin{align}\label{map}
&\begin{aligned}
&E_{1,2}(z) \mapsto b(z),\quad
E_{2,1}(z) \mapsto :c(z)\left(\chi_1(z) + \chi_2(z)\right): + k_1 \partial c(z),\\
&E_{1,1}(z) \mapsto -:c(z) b(z): + \chi_1(z),\quad
E_{2,2}(z) \mapsto :c(z) b(z): + \chi_2(z).
\end{aligned}
\end{align}
\end{proposition}
\noindent
We denote a $V^\kappa\left(\mathfrak{gl}(1|1)\right)$-module $M_{\mathfrak{gl}(1|1)} \otimes \pi^{\kappa - \kappa_2}$ by $\mathbb{W}^\kappa$. Note that $T(z)$ maps to
\begin{align*}
:\partial c(z)\ b(z):+\frac{1-k_2}{2k_1^2}:(\chi_1+\chi_2)(z)^2:+\frac{1}{2k_1}\left(:\chi_1(z)^2:-:\chi_2(z)^2:+\partial(\chi_1+\chi_2)(z)\right).
\end{align*}
\begin{lemma}\label{lem:rho-inj}
$\rho$ is injective for all $\kappa$.
\begin{proof}
Define conformal gradings $\Delta$ on $V^\kappa\left(\mathfrak{gl}(1|1)\right)$ and $\mathbb{W}^\kappa$ by setting
\begin{align}\label{degreeWakimoto}
&\begin{aligned}
&\conf(|0\rangle)=\conf(c_{(-1)}|0\rangle)=0,\\
&\conf(x_{(-1)}|0\rangle)=\conf(b_{(-1)}|0\rangle)=\conf(\chi_{i (-1)}|0\rangle)=1,\quad
x \in \mathfrak{gl}(1|1),\quad
i=1,2
\end{aligned}
\end{align}
and $\conf(A_{(n)}B) = \conf(A)+\conf(B)-n-1$. We choose the set $\mathcal{A}$ of homogeneous strong generators to be $\{X_{(-1)}|0\rangle\}_{X\in \mathfrak{gl}(1|1)}$ for $V^\kappa(\mathfrak{gl}(1|1))$ and $\{b_{(-1)}|0\rangle$, $c_{(-1)}|0\rangle$, $\chi_{1(-1)}|0\rangle$, $\chi_{2(-1)}|0\rangle\}$ for $\mathbb{W}^\kappa$. Then the associated standard filtrations are \begin{align*}
F_nV:=\mathrm{Span}\left\{a_{(-n_1)}^{i_1}\cdots a_{(-n_r)}^{i_r}|0\rangle \mathrel{} \middle| \mathrel{} a^{i_j}\in \mathcal{A}, \sum_{j}\Delta (a^{i_j})\leq n,r\geq0,n_i\geq 0\right\}
\end{align*}
for $V=V^\kappa\left(\mathfrak{gl}(1|1)\right)$ and $\mathbb{W}^\kappa$ and their associated graded superspaces
\begin{align*}
\mathrm{gr}_F V := \bigoplus_{n = 0}^\infty \frac{F_n V}{F_{n-1} V}
\end{align*}
admit a structure of Poisson vertex superalgebra \cite{Ar1,L3}. Since $\rho$ preserves the gradings $\Delta$ by \eqref{map}, $\rho$ induces a homomorphism of Poisson vertex superalgebras
\begin{align*}
\mathrm{gr}_F \rho \colon \mathrm{gr}_F V^\kappa\left(\mathfrak{gl}(1|1)\right) \rightarrow \mathrm{gr}_F \mathbb{W}^\kappa.
\end{align*}
We have
\begin{align*}
&\mathrm{gr}_F V^\kappa\left(\mathfrak{gl}(1|1)\right) = \C[ \partial^n E_{i,j} \mid 1 \leq i, j \leq 2,\ n \in \Z_{\geq0}],\\
&\mathrm{gr}_F \mathbb{W}^\kappa = \C[ \partial^n b, \partial^n c, \partial^n \chi_i \mid i=1,2,\ n \in \Z_{\geq0}],
\end{align*}
where $\partial^n A$ is the image of $\frac{1}{n!}A_{(-n-1)}|0\rangle \in F_{\Delta(A)+n}V$ in $F_{\Delta(A)+n}V / F_{\conf(A)+n-1}V$. Next, define weight gradings $\wt$ on $\mathrm{gr}_F V^\kappa\left(\mathfrak{gl}(1|1)\right)$ and $\mathrm{gr}_F \mathbb{W}^\kappa$ by setting
\begin{align*}
&\wt(\partial^n E_{1,2}) = \wt(\partial^n b) = \wt(\partial^n c)=0,\\
&\wt(\partial^n E_{i,i}) = \wt(\partial^n E_{2,1}) = \wt(\partial^n \chi_i)=1,\quad
i=1,2
\end{align*}
and $\wt(A B) = \wt(A) + \wt(B)$. They yield filtrations  $G_n \overline{V} = \Span \{ A \in \overline{V} \mid \wt(A) \leq n\}$ on $\mathrm{gr}_F \overline{V}$ for $\overline{V}= \mathrm{gr}_FV^\kappa\left(\mathfrak{gl}(1|1)\right)$, $\mathrm{gr}_F \mathbb{W}^\kappa$. Since $\{ G_m \overline{V}_{\lambda} G_n \overline{V}\} \subset G_{m+n} \overline{V}[\lambda]$, the associated graded superspace
\begin{align*}
\mathrm{gr}_G \overline{V} := \bigoplus_{n = 0}^\infty \frac{G_n \overline{V}}{G_{n-1} \overline{V}}
\end{align*}
also has a structure of Poisson vertex superalgebra. 
Since $\mathrm{gr}_F\rho$ preserves the weight gradings by \eqref{map}, $\mathrm{gr}_F\rho$ induces a homomorphism of Poisson vertex superalgebras
\begin{align*}
\begin{split}
\mathrm{gr}_G \mathrm{gr}_F \rho \colon &\mathrm{gr}_G \mathrm{gr}_F V^\kappa\left(\mathfrak{gl}(1|1)\right) \rightarrow \mathrm{gr}_G \mathrm{gr}_F \mathbb{W}^\kappa,\\
& E_{1,2}\mapsto  b,\quad
 E_{2,1}\mapsto  c (\chi_1+\chi_2),\\
& E_{i,i}\mapsto  \chi_i,\quad
i=1,2,
\end{split}
\end{align*}
where $A$ denotes the image of $A \in G_{\wt(A)}\overline{V}$ in $G_{\wt(A)}\overline{V} / G_{\wt(A)-1}\overline{V}$ by abuse of notation. Since $\mathrm{gr}_G \mathrm{gr}_F \rho$ is injective, so is $\rho$.
\end{proof}
\end{lemma}

By using $\rho$ in Proposition \ref{prop:Wakimoto-gl(1|1)}, the $\mathbb{W}^\kappa$-module
\begin{align*}
\mathbb{W}^\kappa_\mu := M_{\mathfrak{gl}(1|1)} \otimes \pi^{\kappa - \kappa_2}_\mu,\quad
\mu \in \mathfrak{h}^*
\end{align*}
becomes a $V^\kappa\left(\mathfrak{gl}(1|1)\right)$-module, which we call the Wakimoto representation of $\widehat{\mathfrak{gl}}(1|1)_\kappa$ with highest weight $\mu$ (cf. \cite{F05}).
\subsection{Wakimoto representations at generic levels}
Here we study the Wakimoto representations $\mathbb{W}^\kappa_\mu$ for generic level $\kappa$. Set 
\begin{align*}
\mu_i:=\mu(E_{i,i}),\quad i=1,2.
\end{align*}
Consider the highest (resp.\ lowest) Verma module of $\mathfrak{gl}(1|1)$
\begin{align*}
V_{n,e}^\pm:=U\left(\mathfrak{gl}(1|1)\right)\otimes_{U(\mathfrak{b}_\pm)}\C v_{n,e},\quad n,e\in\C,
\end{align*}
where $\mathfrak{b}_+:=\mathrm{Span}\{E_{1,2}, N,E\}$, (resp.\ $\mathfrak{b}_-:=\mathrm{Span}\{E_{2,1}, N,E\}$), with 
\begin{align*}
N:=\frac{1}{2}(E_{1,1}-E_{2,2}),\quad E:=E_{1,1}+E_{2,2}
\end{align*}
and $\C v_{n,e}$ is the one dimensional $\mathfrak{b}_\pm$-module such that $E_{1,2}$ (resp.\ $E_{2,1}$) acts by 0, $N$ by $n$, and $E$ by $e$. If $e\neq 0$, then they are irreducible and we have an isomorphism $V^+_{n,e}\simeq V^-_{n-1,e}$. If $e=0$, then they are only indecomposable and we have the following short exact sequence
\begin{align}\label{eq: Jordan-Holder of Verma}
0\rightarrow A_{n\pm1}\rightarrow V_{n,0}^\pm\rightarrow A_n\rightarrow 0,
\end{align}
where $A_q=\C w_q$ ($q\in\C$) is the one dimensional $\mathfrak{gl}(1|1)$-module such that $E_{1,2}$, $E_{2,1}$, $E$ acts by 0, and $N$ acts by $q$.
\smallskip

Define the $V^\kappa(\mathfrak{gl}(1|1))$-modules
\begin{align*}
\hat{V}^{\pm,\kappa}_{n,e}:=U(\widehat{\mathfrak{gl}}(1|1)_\kappa)\otimes_{U(\mathfrak{gl}(1|1)_{\kappa,+})}V^\pm_{n,e},\quad 
\hat{A}_n^\kappa:=U(\widehat{\mathfrak{gl}}(1|1)_\kappa)\otimes_{U(\mathfrak{gl}(1|1)_{\kappa,+})} A_n.
\end{align*}
\begin{lemma}[\cite{CR2}]\label{Induced modules at generic level}
(1) If $e\neq 0$, then $\hat{V}^{\pm,\kappa}_{n,e}$ is irreducible for 
$\frac{e}{k_1}\notin \mathbb{Q}$, and there exists an isomorpshim 
$$\hat{V}^{+,\kappa}_{n,e}\simeq \hat{V}^{-,\kappa}_{n-1,e}.$$
(2) If $e=0$, then $\hat{V}^{\pm,\kappa}_{n,e}$ admits the following short exact sequence for $k_1\neq 0$
\begin{align}\label{eq: Jordan-Holder of affine Verma}
0\rightarrow \hat{A}^\kappa_{n\pm1}\rightarrow \hat{V}_{n,0}^{\pm,\kappa}\rightarrow \hat{A}_n^\kappa\rightarrow 0.
\end{align}
(3) $\hat{A}_n^\kappa$ is irreducible for $k_1\neq 0$.
\end{lemma}
\proof
We include the proof for the completeness of the paper, following the argument in \cite[Section 3.2]{CR2}. We may assume $k_1\neq 0$.
For (1), the isomorphism $V^+_{n,e}\simeq V^-_{n-1,e}$ induces an isomorphism $\hat{V}^{+,\kappa}_{n,e}\simeq \hat{V}^{-,\kappa}_{n-1,e}$. The conformal dimension of $1\otimes v_{n,e}\in \hat{V}^{\pm,\kappa}_{n,e}$ with respect to \eqref{eq: conformal vector} is 
\begin{align*}
\Delta_{n,e}:=\frac{1}{2k_1}\left(\frac{1-k_2}{k_1}e^2+2en\mp1\right).\end{align*}
Suppose that $\hat{V}^{-,\kappa}_{n,e}$ is reducible. Then we have a nontrivial $V^\kappa(\mathfrak{gl}(1|1))$-module homomorphism 
$$\hat{V}^{-,\kappa}_{n',e}\rightarrow \hat{V}^{-,\kappa}_{n,e},\quad \exists n'\in n+\mathbb{Z},$$
and thus 
$\Delta_{n',e}=\Delta_{n,e}+j$ for some $j\in \mathbb{Z}_{>0}$, 
that is, $\frac{e}{k_1}(n'-n)=j$. Since $n'-n\in \mathbb{Z}$, it is impossible when $\frac{e}{k_1}\notin \mathbb{Q}$, in which case $\hat{V}^{-,\kappa}_{n,e}$ is irreducible. This completes the proof of (1).
For (3), note that the conformal dimension of $1\otimes w_n\in \hat{A}_n^\kappa$ with respect to \eqref{eq: conformal vector} is always 0. Suppose that $\hat{A}^\kappa_n$ is reducible. Then we have a nontrivial $V^\kappa(\mathfrak{gl}(1|1))$-module homomorphism  
$$\hat{V}^{+,\kappa}_{n',0}\rightarrow \hat{A}^\kappa_n\quad \mathrm{or}\quad 
\hat{V}^{-,\kappa}_{n',0}\rightarrow \hat{A}^\kappa_n,\quad \exists n'\in \mathbb{Z}$$
such that the image of $1\otimes v_{n,0}$ is not contained in $\C w_n$. It is impossible since the conformal dimension of $1\otimes v_{n,0}\in \hat{V}^{\pm,\kappa}_{n,0}$ is 0. Thus $\hat{A}^\kappa_n$ is irreducible for $k_1\neq0$. For (2), \eqref{eq: Jordan-Holder of Verma} induces an exact sequence of  $V^\kappa(\mathfrak{gl}(1|1))$-modules
$$\hat{A}^\kappa_{n\pm1}\rightarrow \hat{V}^{\pm,\kappa}_{n,0}\rightarrow \hat{A}^\kappa_{n}\rightarrow 0$$
by the right exactness of the tensor functor $U(\widehat{\mathfrak{gl}}(1|1)_\kappa)\otimes_{U(\mathfrak{gl}(1|1)_{\kappa,+})}(?)$. The map $\hat{A}^\kappa_{n\pm1}\rightarrow \hat{V}^{\pm,\kappa}_{n,0}$ is also injective by (3) for $k_1\neq 0$. This completes the proof of (2).
\endproof

\begin{proposition}\label{generic gl11-Wakimoto}
We have an isomorphism 
\begin{align*}
\mathbb{W}_{\mu}^\kappa\simeq \hat{V}^{-,\kappa}_{n(\mu),e(\mu)}
\end{align*}
of $V^\kappa(\mathfrak{gl}(1|1))$-modules where $n(\mu)=\frac{\mu_1+\mu_2}{2}-1$, $e(\mu)=\mu_1-\mu_2$ for $\frac{e(\mu)}{k_1}\not\in \mathbb{Q}$ if $e(\mu)\neq0$ and for $k_1\neq0$ if $e(\mu)=0$.
\end{proposition}
\proof
Define a conformal grading $\Delta$ on $\mathbb{W}^\kappa_{\mu}$ by $\Delta(|\mu\rangle)=0$ and \eqref{degreeWakimoto}:
\begin{align}\label{eq: conformal grading for Wakimoto representations}
\mathbb{W}^\kappa_{\mu}=\bigoplus_{d\geq0}\mathbb{W}^\kappa_{\mu,d}.
\end{align}
Then the subspace 
$\mathbb{W}^\kappa_{\mu,0}=\C|\mu\rangle \oplus \C c_{(-1)}|\mu\rangle$
is a $\hat{\mathfrak{gl}}(1|1)_{\kappa,+}$-module, which is isomorphic to $V^-_{n(\mu),e(\mu)}$ by
\begin{align*}
&V^-_{n(\mu),e(\mu)}\rightarrow \mathbb{W}^{\kappa}_{\mu,0},\\
&v_{n(\mu),e(\mu)}\mapsto c_{(-1)}|\mu\rangle,\ E_{1,2}v_{n(\mu),e(\mu)}\mapsto |\mu\rangle.
\end{align*}
Thus, by the universality of the induced modules, we have a $V^\kappa(\mathfrak{gl}(1|1))$-module homomorphism 
\begin{align}\label{eq: affine Verma to Wakimoto}
\hat{V}^{-,\kappa}_{n(\mu),e(\mu)}\rightarrow \mathbb{W}_{-n\alpha}^\kappa.
\end{align}
If $\mu_1\neq \mu_2$, then the map \eqref{eq: affine Verma to Wakimoto} is injective by Lemma \ref{Induced modules at generic level} (1). If $\mu_1=\mu_2$, then $(n(\mu),e(\mu))=(-n-1,0)$. It follows from \eqref{eq: Jordan-Holder of affine Verma} that the singular vectors of $\hat{V}^{-,\kappa}_{-n-1,0}$ for $k_1\neq0$ belong to the subspace 
$A_{-n}\subset \hat{A}^\kappa_{-n}$,
which is clearly embedded by \eqref{eq: affine Verma to Wakimoto}.
Thus the map \eqref{eq: affine Verma to Wakimoto} is injective by Lemma \ref{Induced modules at generic level} (2) also in this case. Therefore, the map \eqref{eq: affine Verma to Wakimoto} is injective for generic $\kappa$ for any highest weight $\mu\in\mathfrak{h}^*$.  

Define a conformal degree $\Delta$ of $\hat{V}^{-,\kappa}_{n(\mu),e(\mu)}$ by $\Delta(V^-_{n(\mu),e(\mu)})=0$, $\Delta(X_{(-1)})=1$, ($X\in \mathfrak{gl}(1|1)$), and $\Delta(X_{(-n-1)}Y)=\Delta(X)+\Delta(Y)+n+1$. Then the map \eqref{eq: affine Verma to Wakimoto} preserves the conformal gradings by Proposition \ref{prop:Wakimoto-gl(1|1)}.
Now its surjectivity follows from the equality of the graded characters:
\begin{align*}
\mathrm{ch}\left[\widehat V^-_{n(\mu),e(\mu)}\right] = 2 \prod_{n=1}^\infty (1+q^n)^2(1+q^n)^{-2}=\mathrm{ch}[\mathbb{W}^\kappa_{\mu}].
\end{align*}
Thus \eqref{eq: affine Verma to Wakimoto} is an isomorphism. 
\endproof
\subsection{Resolution}
For 
$\alpha = \chi_1+\chi_2 \in \mathfrak{h}^*$,
define an intertwining operator $S(z): \mathbb{W}^\kappa_{-n\alpha}\xrightarrow{S}\mathbb{W}^\kappa_{-(n+1)\alpha}((z))$, ($k_1\neq0$), by
\begin{align*}
S(z) =\  :b(z) \mathrm{e}^{-\frac{1}{k_1}\int\alpha(z)}:,
\end{align*}
where 
\begin{equation}\label{latticeoperator}
:\mathrm{e}^{-\frac{1}{k_1}\int \alpha(z)}:\ =
T_{-\alpha}\mathrm{exp}\left(\frac{1}{k_1}\sum_{n<0}\frac{\alpha_{(n)}}{n}z^{-n}\right)
\mathrm{exp}\left(\frac{1}{k_1}\sum_{n>0}\frac{\alpha_{(n)}}{n}z^{-n}\right),
\end{equation}
and $T_{-\alpha}$ is the shift operator $\pi^\kappa_{-n\alpha}\rightarrow \pi^\kappa_{-(n+1)\alpha}$ sending the highest weight vector to the highest weight vector and commuting with all $\chi_{i(n)}$, $i=1,2$, $n\neq0$. By direct calculation, one can show that the residue 
$$S:=\int S(z)dz$$
satisfies $S(u_{(-1)}|0\rangle)=0$, $u\in\mathfrak{gl}(1|1)$. It follows that $S$ is a $V^\kappa(\mathfrak{gl}(1|1))$-module homomorphism from $\mathbb{W}^\kappa_{-n\alpha}$ to $\mathbb{W}^\kappa_{-(n+1)\alpha}$ (cf.\ \cite{F05}). 

Using Wakimoto representations, we can extend the injective homomorphism $\rho$ to a long exact sequence as in the following proposition.
\begin{proposition}\label{longexactsequence} The sequence
\begin{equation}\label{resolution}
0\rightarrow V^\kappa(\mathfrak{gl}(1|1))\xrightarrow{\rho} \mathbb{W}^\kappa_0\xrightarrow{S} \mathbb{W}^\kappa_{-\alpha}\rightarrow\cdots \rightarrow\mathbb{W}^\kappa_{-n\alpha}\xrightarrow{S}\mathbb{W}^\kappa_{-(n+1)\alpha}\rightarrow\cdots.
\end{equation}
is a complex of $V^\kappa(\mathfrak{gl}(1|1))$-modules and exact for  $k_1\neq 0$.
\end{proposition}
\begin{proof}
To show that  \eqref{resolution} is a complex, we have to show (1) $\mathrm{Im}(\rho)\subset\mathrm{Ker}(S)$ and (2) $S\circ S=0$. 
(1) follows from $S(x_{(-1)}|0\rangle)=0$, $x\in\mathfrak{gl}(1|1)$. To show (2), notice that we can extend the vertex superalgebra $\mathbb{W}^\kappa$ to  the direct sum of $\mathbb{W}^\kappa$-modules $\bigoplus_{n\geq0}\mathbb{W}^\kappa_{-n\alpha}$ by setting 
$$Y(|-n\alpha\rangle,z)=:\mathrm{e}^{-\frac{n}{k_1}\int \alpha(z)}:,$$
where $|-n\alpha\rangle$ is a fixed non-zero highest weight vector of $\pi^\kappa_{n\alpha}$ (see \eqref{latticeoperator}). This is due to $\kappa(n\alpha,m\alpha)=0$ for $n,m\geq0$ (cf. \cite[Chapter 5]{FBZ}). Then $S$ is the $0$-th mode $Q_{(0)}$ of the field corresponding to $Q=b_{(-1)}|-\alpha\rangle$. By the Jacobi identity of the $\lambda$-bracket, we have
\begin{eqnarray*}
S\circ S(a)
=[Q\hspace{0mm}_\lambda[Q\hspace{0mm}_\mu a]]_{\lambda=\mu=0}
=\frac{1}{2}[[Q\hspace{0mm}_\lambda Q]_{\lambda+\mu} a]]_{\lambda=\mu=0}
=0, \quad a\in\mathbb{W}^\kappa_{-n\alpha}.
\end{eqnarray*}
Thus $S\circ S=0$ follows.

It is clear that the complex \eqref{resolution} preserves the conformal gradings \eqref{eq: conformal grading for Wakimoto representations}. In particular, the subcomplex of conformal grading 0 is isomorphic to 
\begin{align*}
0\rightarrow A_0\rightarrow V^-_{-1,0}\rightarrow V^-_{-2,0}\rightarrow\cdots \rightarrow V^-_{-n,0} \rightarrow V^-_{-n-1,0}\rightarrow \cdots,
\end{align*}
which is a complex of $\mathfrak{gl}(1|1)$-modules.
Since $S(|-n\alpha\rangle)=0$ and $S(c_{(-1)}|-n\alpha\rangle)=|-(n+1)\alpha\rangle$, it is a successive composition of \eqref{eq: Jordan-Holder of Verma}, and thus exact. Now it follows that the complex \eqref{resolution} for $k_1\neq 0$ is a successive composition of \eqref{eq: Jordan-Holder of affine Verma} since $\hat{A}^\kappa_{-n}$, ($n\geq0$), are irreducible for $k_1\neq0$. Thus \eqref{resolution} is exact.
\end{proof}

\section{Free field realizations of $\mathcal{W}$-algebras}\label{sec:Wfree}

\subsection{$\mathcal{W}$-algebras}\label{sec:W-alg}
Let $\mathfrak{g}$ be a finite dimensional simple Lie superalgebra with a non-degenerate even supersymmetric invariant bilinear form $(\cdot|\cdot)$ such that $(\theta|\theta) = 2$ for a long root $\theta$ of $\mathfrak{g}$. Let $f$ be a nilpotent element in the even part of $\mathfrak{g}$ and fix a $\frac{1}{2}\Z$-grading $\Gamma: \mathfrak{g} = \bigoplus_{j \in \frac{1}{2}\Z} \mathfrak{g}_j$ on $\mathfrak{g}$, called a good grading, which satisfies $[\mathfrak{g}_i, \mathfrak{g}_j] \subset \mathfrak{g}_{i+j}$, $f \in \mathfrak{g}_{-1}$, and that $\operatorname{ad}_f \colon \mathfrak{g}_j \rightarrow \mathfrak{g}_{j-1}$ is injective for $j \geq \frac{1}{2}$ and surjective for $j \leq \frac{1}{2}$. Then one can define the $\mathcal{W}$-algebra $\mathcal{W}^k(\mathfrak{g}, f; \Gamma)$ associated with $\mathfrak{g}, f, \Gamma$ and a complex number $k$ via the (generalized) Drinfeld-Sokolov reduction \cite{FF1, KRW}.
\smallskip

Following \cite{KW, G}, we will embed $\mathcal{W}^k(\mathfrak{g}, f; \Gamma)$ into an affine vertex superalgebra under the assumption that
$\Gamma$ is a $\Z$-grading, i.e.,
\begin{align*}
\Gamma: \mathfrak{g} = \bigoplus_{j \in \Z} \mathfrak{g}_j.
\end{align*}
Let $\tau_k$ be the invariant bilinear form on $\mathfrak{g}_0$ defined by
\begin{align}\label{eq:tau-k}
\tau_k(u|v) = k(u|v) + \frac{1}{2} \kappa_{\mathfrak{g}} (u|v) - \frac{1}{2} \kappa_{\mathfrak{g}_0}(u|v),\quad
u, v \in \mathfrak{g}_0,
\end{align}
where $\kappa_{\mathfrak{g}}$, $\kappa_{\mathfrak{g}_0}$ are the Killing forms on $\mathfrak{g}$, $\mathfrak{g}_0$ respectively. Let $\widehat{\mathfrak{g}}_0=\mathfrak{g}_0 [t, t^{-1}]\oplus \C K$  denote the affine Lie superalgebra of $\mathfrak{g}_0$ at level $\tau_k$, which satisfies
\begin{align*}
[u_{(m)}, v_{(n)}] = [u, v]_{(m+n)} + m \tau_k(u|v)\delta_{m+n,0}K,\quad
u, v \in \mathfrak{g}_0,\quad
m, n \in \Z,
\end{align*}
\begin{align*}
[K,\widehat{\mathfrak{g}}_0]=0,
\end{align*}
where $u_{(m)} = u t^m$.
Define a $\widehat{\mathfrak{g}}_0$-module $V^{\tau_k}(\mathfrak{g}_0)$ by 
\begin{align*}
V^{\tau_k}(\mathfrak{g}_0) := U( \widehat{\mathfrak{g}}_0 ) \underset{U(\mathfrak{g}_0 [t] \oplus \C K)}{\otimes}\C,
\end{align*}
where $\C$ is a one dimensional $(\mathfrak{g}_0 [t] \oplus \C K)$-module, on which $K$ acts by $1$ and $\mathfrak{g}_0 [t]$ acts trivially. There is a unique vertex algebra structure on $V^{\tau_k}(\mathfrak{g}_0)$ such that $|0\rangle=1\otimes 1$ is the vacuum vector and $Y(u_{(-1)}\otimes v,z)=J^u(z):=\sum_{n\in\Z}u_{(n)}z^{-n-1}$, ($u \in \mathfrak{g}_0$), with OPEs
\begin{align*}
J^u(z) J^v(w) \sim \frac{J^{[u, v]}(w)}{z-w}+\frac{\tau_k(u|v)}{(z-w)^2},\quad
u, v \in \mathfrak{g}_0.
\end{align*}
It is called the affine vertex superalgebra of $\mathfrak{g}_0$ at level $\tau_k$.

Let $\Delta$ denote the root system of $\mathfrak{g}$ and $\Pi$ the set of its simple roots. The grading $\Gamma$ gives the decomposition $\Delta = \bigsqcup_{j \in \Z}\Delta_j$ where $\Delta_j = \{ \alpha \in \Delta \mid \mathfrak{g}_\alpha \subset \mathfrak{g}_j\}$. Then we can choose $\Delta$ and $\Pi$ such that $\Pi = \Pi_0 \sqcup \Pi_1$, where $\Pi_j = \Pi \cap \Delta_j$. 
For $\alpha\in\Pi_1$, let $\mathfrak{g}_{[\alpha]}\subset \mathfrak{g}$ denote the adjoint $\mathfrak{g}_0$-submodule generated by $\mathfrak{g}_\alpha$. It is simple, and so is its dual module, which we denote by $\C^{[\alpha]}$. More explicitly, take a basis $\{e_\alpha\}_{\alpha\in I\sqcup \Delta}$ of $\mathfrak{g}$ consisting of a basis $\{e_\alpha\}_{\alpha\in I}$ of the Cartan subalgebra $\mathfrak{h}$ and root vectors $e_{\alpha}$ of $\mathfrak{g}_\alpha$. Let $c_{\alpha,\beta}^\gamma$ be the structure constants, i.e., $[e_\alpha,e_\beta]=\sum_{\gamma}c_{\alpha,\beta}^\gamma e_\gamma$, extended to $[u,v]=\sum_{\gamma}c_{u,v}^\gamma e_\gamma$, $u,v,\in \mathfrak{g}$. By setting $v_\beta\in \mathfrak{g}^*$ to be the dual vector of $e_\beta$, we have 
\begin{align*}
\C^{[\alpha]} = \bigoplus_{\beta \in [\alpha]}\C v_\beta,\quad
[\alpha] := \{ \beta \in \Delta_1 \mid \beta - \alpha \in \Z \Delta_0 \},\quad
\alpha \in \Pi_1
\end{align*} 
and
\begin{align*}
u \cdot v_\beta = \sum_{\gamma \in [\alpha]}c_{\gamma, u}^\beta v_\gamma,\quad
u \in \mathfrak{g}_0,\quad
\beta \in [\alpha].
\end{align*}
Define a $\hat{\mathfrak{g}}_0$-module by
\begin{align*}
M_\alpha := U( \widehat{\mathfrak{g}}_0 ) \underset{U(\mathfrak{g}_0 [t] \oplus \C K)}{\otimes}\C^{[\alpha]},\quad
\alpha \in \Pi_1,
\end{align*}
where $\mathfrak{g}_0 [t] \oplus \C K$ acts on $\C^{[\alpha]}$ by $u_{(n)}=\delta_{n,0}u$, ($u\in\mathfrak{g}_0$, $n\geq0$), and $K$ acts by $1$. It has a unique $V^{\tau_k}(\mathfrak{g}_0)$-module structure induced by the $\hat{\mathfrak{g}}_0$-module structure.
Now, by \cite[Section 4.3]{G}, there exists a cochain complex $(C_k', d_{\mathrm{st}(0)})$ consisting of a $\Z_{\geq0}$-graded vertex superalgebra $C_k'$ and an odd vertex operator $d_{\mathrm{st}(0)}$, whose cohomology is a vertex superalgebra such that 
\begin{align*}
H^0 \left(C_k', d_{\mathrm{st}(0)}\right) \simeq V^{\tau_k}(\mathfrak{g}_0),\quad
H^1 \left(C_k', d_{\mathrm{st}(0)}\right) \simeq \bigoplus_{\alpha \in \Pi_1} M_\alpha
\end{align*}
as vertex superalgebras and as $V^{\tau_k}(\mathfrak{g}_0)$-modules, respectively.
Define an intertwining operator $S_\beta (z)$ of type $\binom{M_\alpha}{M_\alpha, V^{\tau_k}(\mathfrak{g}_0)}$ ($\alpha \in \Pi_1$) by
\begin{align}\label{eq;def-of-S}
S_\beta (z) := Y(v_\beta, z), \quad \beta \in [\alpha],
\end{align}
and a screening operator by 
\begin{align}\label{eq;alpha-screening}
\int Q_\alpha(z) dz:= \sum_{\beta \in [\alpha]} (f|e_\beta) \int S_\beta(z) dz\colon V^{\tau_k}(\mathfrak{g}_0) \rightarrow M_\alpha.
\end{align}
Note that $S_\beta(z)$ is uniquely determined by the property
\begin{align}\label{screening}
S_\beta(z) J^u(w) \sim \frac{Y(u\cdot v_\beta,w)}{z-w}=\sum_{\gamma \in [\alpha]} \frac{c_{u, \gamma}^\beta S^\gamma(w)}{z-w},\quad u\in\mathfrak{g}_0.
\end{align}
In \cite{G}, the formula
\begin{align*}
\partial S_\beta (z) = - \frac{(k+h^\vee)^{-1}}{(e_\beta|e_{-\beta})}\Bigl( :J^{h_\beta}(z)S_\beta(z): + \sum_{\begin{subarray}{c} \gamma \in [\alpha], \delta \in \Delta_0 \\ \gamma \neq \beta \end{subarray}} (-1)^{\bar{\gamma}\bar{\delta}}c_{\gamma, -\beta}^{\delta} :J^{e_\delta}(z)S_\gamma(z): \Bigr),
\end{align*}
where $h_\beta = [e_\beta, e_{-\beta}]$, is also used to characterize $S_\beta(z)$. We note that this is nothing but
$$\partial S_\beta (z)=Y(L_0v_\beta,z),$$
where $L(z)=\sum_{n\in\mathbb{Z}}L_nz^{-n-1}$ is the Segal-Sugawara conformal field of $\mathfrak{g}_0$. Thus \eqref{screening} is enough to characterize $S_\beta(z)$.
We have a vertex superalgebra homomorphism
\begin{align}\label{eq: Miura-map}
\mathcal{W}^k(\mathfrak{g},f;\Gamma) \hookrightarrow V^{\tau_k}(\mathfrak{g}_0),
\end{align}
called the Miura map \cite{KW}. The injectivity of this map is proven in \cite{Ar5,F2} if $\g$ is a simple Lie algebra and $f$ is a principal nilpotent element, but the proof there applies for arbitrary $\mathcal{W}$-algebras $\mathcal{W}^k(\mathfrak{g},f;\Gamma)$, see for example \cite{N}.
By \cite{FBZ,G} the image of the Miura map for generic $k$ coincides with 
\begin{align}\label{eq:W-screening}
\mathcal{W}^k(\mathfrak{g},f;\Gamma) \simeq \bigcap_{\alpha \in \Pi_1}\Ker \int Q_\alpha(z)\ dz.
\end{align}

\subsection{Subregular $\mathcal{W}$-algebras}\label{sec:sub-W}
We describe the isomorphism \eqref{eq:W-screening} more explicitly in the case $\mathfrak{g} = \mathfrak{sl}_{n+1}$ or $\mathfrak{so}_{2n+1}$ with a subregular nilpotent element $f_{\text{sub}}$. We have the natural representation $\mathfrak{sl}_{n+1} \hookrightarrow \mathfrak{gl}(V_1)$, $\mathfrak{so}_{2n+1} \hookrightarrow \mathfrak{gl}(V_2)$, where $V_1 = \C^{n+1}$, $V_2 = \C^{2n+1}$, see e.g., \cite{Hum}. Let $\{ e_i \}_{i \in I_s}$ denote the standard bases of $V_s$ with index sets
\begin{align*}
I_1 = \{1, \ldots, n+1\},\quad
I_2 = \{ 0, 1, \ldots, n, -1, \ldots, -n\},
\end{align*}
and $e_{i, j} \in \mathfrak{gl}(V_s)$ the elementary matrix $e_{i, j} \cdot e_m = \delta_{j, m} e_i$, ($i, j, m \in I_s$). Then the Cartan subalgebra $\mathfrak{h}$ of $\mathfrak{g}$ is spanned by
\begin{align*}
&h_i = e_{i, i} - e_{i+1, i+1},\quad
(i = 1, \ldots, n),\quad \text{if}\ \mathfrak{g} = \mathfrak{sl}_{n+1},\\
&h_j =\begin{cases}t_j - t_{j+1},&(j = 1, \ldots, n-1), \\ t_n,&(j =n),\end{cases}
\quad \text{if}\ \mathfrak{g} = \mathfrak{so}_{2n+1},
\end{align*}
where $t_i = e_{i, i} - e_{-i, -i}$. 
The normalized invariant bilinear form $(\cdot | \cdot)$ on $\mathfrak{g}$ is given by $(u|v) = \operatorname{tr}(u \circ v)$ if $\mathfrak{g} = \mathfrak{sl}_{n+1}$, and by $(u|v) = \frac{1}{2}\operatorname{tr}(u \circ v)$ if $\mathfrak{g} = \mathfrak{so}_{2n+1}$, which gives an isomorphism $\nu \colon \mathfrak{h} \xrightarrow{\sim} \mathfrak{h}^*$. Then the set of simple roots of $\mathfrak{g}$ is $\Pi=\{\alpha_i\}_{i=1}^n$ where $\alpha_i = \nu(h_i) \in \mathfrak{h}^*$. Choose a subregular nilpotent element $f_\mathrm{sub}$ of $\mathfrak{g}$ and a semisimple element $x$ of $\mathfrak{g}$ such that $\operatorname{ad}_x$ defines a good $\Z$-grading $\Gamma=\Gamma_{\text{sub}}$ on $\mathfrak{g}$ as follows:
\begin{align*}
&f_\mathrm{sub} = \sum_{i=2}^{n}e_{i+1,i},\quad
x = \frac{1}{2(n+1)} \sum_{i=1}^n (n-i+1)(in+i-2)h_i,\quad \text{if}\ \mathfrak{g} = \mathfrak{sl}_{n+1},\\
&f_\mathrm{sub} = \sum_{i=2}^{n-1}(e_{i+1, i} - e_{-i, -i-1}) + e_{0,n} - e_{-n,0},\quad
x = \sum_{i=1}^n(n-i+1)t_i-t_1,\quad\text{if}\ \mathfrak{g} = \mathfrak{so}_{2n+1}.
\end{align*}
The weighted Dynkin diagrams corresponding to $\Gamma_{\text{sub}}$ are as follows.
\begin{align*}
\quad\\
\setlength{\unitlength}{1mm}
\begin{picture}(0,0)(20,10)
\put(-38,9){$\mathfrak{g} = \mathfrak{sl}_n:$}
\put(0,10){\circle{2}}
\put(-1,13){\footnotesize$0$}
\put(-1,5){\footnotesize$\alpha_1$}
\put(1,10.3){\line(1,0){8}}
\put(10,10){\circle{2}}
\put(9,13){\footnotesize$1$}
\put(9,5){\footnotesize$\alpha_2$}
\put(11,10.3){\line(1,0){6}}
\put(18.5,9.4){$\cdot$}
\put(20,9.4){$\cdot$}
\put(21.5,9.4){$\cdot$}
\put(24.5,10.3){\line(1,0){6}}
\put(32,10){\circle{2}}
\put(31,13){\footnotesize$1$}
\put(31,5){\footnotesize$\alpha_{n-2}$}
\put(33,10.3){\line(1,0){8}}
\put(42,10){\circle{2}}
\put(41,13){\footnotesize$1$}
\put(41,5){\footnotesize$\alpha_{n-1}$}
\put(43,10.3){\line(1,0){8}}
\put(52,10){\circle{2}}
\put(51,13){\footnotesize$1$}
\put(51,5){\footnotesize$\alpha_{n}$}
\put(54,9){.}
\end{picture}\\
\end{align*}
\begin{align*}
\quad\\
\setlength{\unitlength}{1mm}
\begin{picture}(0,0)(20,10)
\put(-38,9){$\mathfrak{g} = \mathfrak{so}_{2n+1}:$}
\put(0,10){\circle{2}}
\put(-1,13){\footnotesize$0$}
\put(-1,5){\footnotesize$\alpha_1$}
\put(1,10.3){\line(1,0){8}}
\put(10,10){\circle{2}}
\put(9,13){\footnotesize$1$}
\put(9,5){\footnotesize$\alpha_2$}
\put(11,10.3){\line(1,0){6}}
\put(18.5,9.4){$\cdot$}
\put(20,9.4){$\cdot$}
\put(21.5,9.4){$\cdot$}
\put(24.5,10.3){\line(1,0){6}}
\put(32,10){\circle{2}}
\put(31,13){\footnotesize$1$}
\put(31,5){\footnotesize$\alpha_{n-2}$}
\put(33,10.3){\line(1,0){8}}
\put(42,10){\circle{2}}
\put(41,13){\footnotesize$1$}
\put(41,5){\footnotesize$\alpha_{n-1}$}
\put(43,10.5){\line(1,0){8}}
\put(43,9.5){\line(1,0){8}}
\put(45.5,9){\Large$>$}
\put(52,10){\circle{2}}
\put(51,13){\footnotesize$1$}
\put(51,5){\footnotesize$\alpha_{n}$}
\put(54,9){.}
\end{picture}\\
\end{align*}

The Lie subalgebra $\mathfrak{g}_0 = \{ u \in \mathfrak{g} \mid [x, u] = 0\}$ decomposes as
\begin{align*}
\mathfrak{g}_0 = \mathfrak{g}_0^\mathrm{red} \oplus \mathfrak{z},\quad
\mathfrak{g}_0^\mathrm{red} := \operatorname{Span}\{e_1, h_1, f_1\},\quad
\mathfrak{z} := \operatorname{Span}\{\widetilde{h}_2, h_3, \ldots, h_n\},
\end{align*}
where $\widetilde{h}_2 = h_2 + \frac{1}{2}h_1$ and
\begin{align*}
\begin{array}{lll}
e_1 = e_{1,2},&f_1 = e_{2,1},& \text{if}\ \mathfrak{g} = \mathfrak{sl}_{n+1},\\
e_1 = e_{1,2}-e_{-2,-1},& f_1 = e_{2,1}-e_{-1,-2},&\text{if}\ \mathfrak{g} = \mathfrak{so}_{2n+1}.
\end{array}
\end{align*}
The Lie subalgebra $\mathfrak{z}$ is commutative and $\mathfrak{g}_0^\mathrm{red}$ is isomorphic to $\mathfrak{sl}_2$ by
\begin{align*}
\begin{array}{ccc}
\mathfrak{sl}_2& \xrightarrow{\sim}&\mathfrak{g}_0^{\mathrm{red}}\\
e,h,f&\mapsto&e_1,h_1,f_1.
\end{array}
\end{align*}
The restriction of $\tau_k$ \eqref{eq:tau-k} on $\mathfrak{g}_0$ is
\begin{align*}
\tau_k(u|v) =
\begin{cases}
(k+h^\vee)(u|v), & u, v \in \mathfrak{z},\\
(k+h^\vee-2)(u|v), & u, v \in \mathfrak{g}_0^{\mathrm{red}},\\
0, & u\in \mathfrak{z},\ v\in \mathfrak{g}_0^{\mathrm{red}},
\end{cases}
\end{align*}
where $h^\vee$ is the dual Coxeter number of $\mathfrak{g}$, which is $n+1$, (resp.\ $2n-1$), if $\mathfrak{g} = \mathfrak{sl}_{n+1}$, (resp.\ $\mathfrak{so}_{2n+1}$). 
Therefore, the affine vertex algebra $V^{\tau_k}(\mathfrak{g}_0)$ decomposes as
\begin{align}\label{g1-decomposition}
V^{\tau_k}(\mathfrak{g}_0) \simeq V^{k+h^\vee-2}(\mathfrak{sl}_2) \otimes \pi^{k+h^\vee}_\mathfrak{z},
\end{align}
where $\pi^{k+h^\vee}_\mathfrak{z}$ is the Heisenberg vertex algebra associated with $\mathfrak{z}$ at level $k+h^\vee$ (Section\ \ref{Sect: Heisenberg}). 

Since
\begin{align}\label{eq; alpha}
[\alpha_2] = \{\alpha_2, \alpha_1+\alpha_2\},\quad
[\alpha_i] = \{\alpha_i\},\quad
i = 3, \ldots, n,
\end{align}\
the screening operators $\int Q_i(z) dz:=\int Q_{\alpha_i}(z) dz$ in \eqref{eq;alpha-screening} are
\begin{align}\label{i-screening operator}
\int Q_i(z)=\int S_{\alpha_i}(z) dz\colon V^{\tau_k}(\mathfrak{g}_0) \rightarrow M_{\alpha_i},\quad i=2,\ldots,n,
\end{align} 

First, consider \eqref{i-screening operator} in the case $i = 3, \ldots, n$. The orthogonal decomposition
\begin{align*}
\mathfrak{z} = \mathfrak{z}_i^\perp \oplus \mathfrak{z}_i,\quad
\mathfrak{z}_i = \C h_i,\quad
\mathfrak{z}_i^\perp = \{ h \in \mathfrak{z} \mid \alpha_i(h) = 0 \}.
\end{align*}
induces the decomposition of the Heisenberg vertex algebra
\begin{align*}
\pi^{k+h^\vee}_\mathfrak{z} \simeq \pi^{k+h^\vee}_{\mathfrak{z}_i^\perp} \otimes \pi^{k+h^\vee}_{\mathfrak{z}_i}.
\end{align*}
It follows from \eqref{g1-decomposition} and \eqref{eq; alpha} that we have an isomorphism of $V^{\tau_\kappa}(\mathfrak{g}_0)$-modules 
\begin{align*}
\begin{array}{ccc}
M_{\alpha_i}& \simeq &\widetilde{M}_{\alpha_i}:=V^{k+h^\vee-2}(\mathfrak{sl}_2) \otimes \pi^{k+h^\vee}_{\mathfrak{z}_i^\perp} \otimes \pi^{k+h^\vee}_{\mathfrak{z}_i, -\alpha_i}\\
v_{\alpha_i}&\mapsto &|0\rangle\otimes|0\rangle\otimes|-\alpha_i\rangle.
\end{array}
\end{align*}
Thus $S_{\alpha_i}(z)$ is identified with
\begin{align*}
\int S_{\alpha_i}(z)\ dz = \int :\mathrm{e}^{-\frac{1}{k+h^\vee} \int \alpha_i(z)}:\ dz\quad
\colon V^{k+h^\vee-2}(\mathfrak{sl}_2) \otimes \pi^{k+h^\vee}_{\mathfrak{z}_i^\perp} \otimes \pi^{k+h^\vee}_{\mathfrak{z}_i} \rightarrow \widetilde{M}_{\alpha_i},
\end{align*}
where $\alpha_i(z) = J^{h_i}(z)$, and, therefore,
\begin{align}\label{eq:sub-qscr-i}
\int Q_i(z)\ dz = \int :\mathrm{e}^{-\frac{1}{k+h^\vee} \int \alpha_i(z)}:\ dz.
\end{align}
Next, consider \eqref{i-screening operator} in the case $\alpha=\alpha_2$. To this end, we use the Wakimoto realization of $V^{k+h^\vee-2}(\mathfrak{sl}_2)$ (cf. \cite{F05}).
\begin{align*}
&\rho_{\mathfrak{sl}_2} \colon V^{k+h^\vee-2}(\mathfrak{sl}_2) \hookrightarrow M_{\mathfrak{sl}_2} \otimes \pi^{k+h^\vee}_a,\\
&J^{e_1}(z) \mapsto \beta(z),\quad
J^{h_1}(z) \mapsto -2:\gamma(z)\beta(z):+a(z),\\
&J^{f_1}(z) \mapsto -:\gamma(z)^2\beta(z):+(k+h^\vee-2)\partial\gamma(z)+\gamma(z)a(z),
\end{align*}
which gives an isomorphism for generic $k$
\begin{align*}
V^{k+h^\vee-2}(\mathfrak{sl}_2) \simeq \operatorname{Ker}\left(\int :\beta(z)\mathrm{e}^{-\frac{1}{k+h^\vee} \int a(z)}:\ dz: M_{\mathfrak{sl}_2} \otimes \pi^{k+h^\vee}_a\rightarrow M_{\mathfrak{sl}_2} \otimes \pi^{k+h^\vee}_{a,-a}\right).
\end{align*}
Here $M_{\mathfrak{sl}_2}$ is the $\beta\gamma$-system vertex algebra generated by the even fields $\beta(z), \gamma(z)$ with OPEs 
\begin{align*}
\beta(z)\gamma(w) \sim \frac{1}{z-w},\quad
\beta(z)\beta(w) \sim 0 \sim \gamma(z)\gamma(w),
\end{align*}
and $\pi^{k+h^\vee}_a$ is the Heisenberg vertex algebra generated by an even field $a(z)$ with an OPE
\begin{align*}
a(z)a(w) \sim \frac{2(k+h^\vee)}{(z-w)^2}.
\end{align*}
Then it follows from \eqref{g1-decomposition} and the isomorphism of vertex algebras 
\begin{align*}
&\pi^{k+h^\vee}_\mathfrak{h} \xrightarrow{\sim} \pi^{k+h^\vee}_a \otimes \pi^{k+h^\vee}_\mathfrak{z}\\
&\alpha_1(z) \mapsto a(z),\quad
\alpha_2(z) \mapsto J^{\widetilde{h}_2}(z)-\frac{1}{2}a(z),\quad
\alpha_i(z) \mapsto J^{h_i}(z),\ (i=3, \ldots, n),
\end{align*}
that we have a vertex algebra embedding
\begin{align}\label{eq: g0-Wakimoto}
\rho_{\mathfrak{g}_0} \colon V^{\tau_k}(\mathfrak{g}_0) \hookrightarrow M_{\mathfrak{sl}_2} \otimes \pi^{k+h^\vee}_\mathfrak{h},
\end{align}
which gives an isomorphism for generic $k$
\begin{align}\label{eq:Vtau-Wakimoto}
V^{\tau_k}(\mathfrak{g}_0) \simeq \operatorname{Ker}\left(\int :\beta(z)\mathrm{e}^{-\frac{1}{k+h^\vee} \int a(z)}:\ dz:\ M_{\mathfrak{sl}_2} \otimes \pi^{k+h^\vee}_\mathfrak{h}\rightarrow M_{\mathfrak{sl}_2} \otimes \pi^{k+h^\vee}_{\mathfrak{h},-\alpha_1}\right).
\end{align}
It gives a $V^{\tau_k}(\mathfrak{g}_0)$-module structure on $M_{\mathfrak{sl}_2} \otimes \pi^{k+h^\vee}_{\mathfrak{h}, -\alpha_2}$. Let $\widetilde{M}_{\alpha_2}$ be a $V^{\tau_k}(\mathfrak{g}_0)$-submodule generated by the subspace 
\begin{align*}
\widetilde{\C}^{[\alpha_2]} = \C |-\alpha_2\rangle \oplus \C\gamma_{(-1)}|-\alpha_2\rangle.
\end{align*}
\begin{lemma}\label{lem:M-isom}
For generic $k$, $M_{\alpha_2} \simeq \widetilde{M}_{\alpha_2}$ as $V^{\tau_k}(\mathfrak{g}_0)$-modules.
\begin{proof}
The linear map
\begin{align*}
\C^{[\alpha_2]} \xrightarrow{\sim} \widetilde{\C}^{[\alpha_2]},\quad
v_{\alpha_2} \mapsto |-\alpha_2\rangle,\quad
v_{\alpha_1+\alpha_2} \mapsto -\gamma_{(-1)}|-\alpha_2\rangle
\end{align*}
gives an isomorphism as $(\mathfrak{g}_0[t] \oplus \C K)$-modules. By the universality of the induced modules, it induces a surjective $V^{\tau_k}(\mathfrak{g}_0)$-module homomorphism
\begin{align}\label{eq:M-to-M}
M_{\alpha_2} \twoheadrightarrow \widetilde{M}_{\alpha_2}.
\end{align}
Since $\C^{[\alpha_2]}$ is simple as a $\mathfrak{g}_0$-module, $M_{\alpha_2}$is simple as a $V^{\tau_k}(\mathfrak{g}_0)$-module for generic $k$. Thus \eqref{eq:M-to-M} is an isomorphism for such $k$.
\end{proof}
\end{lemma}
Under the realization \eqref{eq:Vtau-Wakimoto}, Lemma \ref{lem:M-isom} implies that $S_{\alpha_2}(z)$ is identified with
\begin{align*}
S_{\alpha_2}(z) =\ :\mathrm{e}^{-\frac{1}{k+h^\vee} \int \alpha_2(z)}:\quad\colon \mathrm{Im}(\rho_{\mathfrak{g}_0})\rightarrow \widetilde{M}_{\alpha_2}((z)),
\end{align*}
and thus
\begin{align}\label{eq:sub-qscr-2}
\int Q_2(z)\ dz = \int :\mathrm{e}^{-\frac{1}{k+h^\vee} \int \alpha_2(z)}:\ dz.
\end{align}
We also note that $S_{\alpha_1+\alpha_2}(z) =\ -:\gamma(z)\mathrm{e}^{-\frac{1}{k+h^\vee} \int \alpha_2(z)}:$.
By \eqref{eq:W-screening}, \eqref{eq:sub-qscr-i} and \eqref{eq:sub-qscr-2}, we conclude
\begin{align*}
\mathcal{W}^k (\mathfrak{g}, f_\mathrm{sub})
&\simeq \bigcap_{i=2}^n \operatorname{Ker}\Bigl( \int Q_i(z)\ dz \colon V^{\tau_k}(\mathfrak{g}_0) \rightarrow M_{\alpha_i} \Bigr)\\
&\simeq \bigcap_{i=2}^n \operatorname{Ker}\Bigl( \int :\mathrm{e}^{-\frac{1}{k+h^\vee} \int \alpha_i(z)}:\ dz\quad
\colon \mathrm{Im}(\rho_{\mathfrak{g}_0}) \rightarrow \widetilde{M}_{\alpha_i} \Bigr).
\end{align*}
The composition $\Upsilon_1$ of \eqref{eq: Miura-map} and \eqref{eq: g0-Wakimoto} gives a vertex algebra embedding
of  the subregular $\mathcal{W}$-algebra $\mathcal{W}^k (\mathfrak{g}, f_\mathrm{sub}):=\mathcal{W}^k (\mathfrak{g},f_\mathrm{sub} ;\Gamma_\mathrm{sub})$
\begin{align}\label{eq: Upsilon_1} 
\Upsilon_1: \mathcal{W}^k (\mathfrak{g}, f_\mathrm{sub})\hookrightarrow M_{\mathfrak{sl}_2} \otimes \pi^{k+h^\vee}_\mathfrak{h}.
\end{align}
Now \eqref{eq:Vtau-Wakimoto} implies the following realization of the image $\mathrm{Im}(\Upsilon_1)$.
\begin{theorem}\label{thm:sub-W}
If $k$ is generic, then we have 
\begin{align*}
\mathrm{Im}(\Upsilon_1)=\bigcap_{i = 1}^n \operatorname{Ker} \int Q_i(z) dz
\end{align*}
for $\mathfrak{g} = \mathfrak{sl}_{n+1}$, $\mathfrak{so}_{2n+1}$, where
\begin{align*}
Q_1(z) = :\beta(z)\mathrm{e}^{-\frac{1}{k+h^\vee} \int \alpha_1(z)}:,\quad
Q_i(z) = :\mathrm{e}^{-\frac{1}{k+h^\vee} \int \alpha_i(z)}:,\ 
(i = 2, \ldots, n).
\end{align*}
\end{theorem}

\subsection{Principal $\mathcal{W}$-superalgebras}
We describe the isomorphism \eqref{eq:W-screening} more explicitly in the case  $\mathfrak{g} = \mathfrak{sl}(1|n+1)$ or $\mathfrak{osp}(2|2n)$ with a principal nilpotent element. We have the natural representation $\mathfrak{sl}(1|n+1) \hookrightarrow \mathfrak{gl}(U_1)$, $\mathfrak{osp}(2|2n) \hookrightarrow \mathfrak{gl}(U_2)$, where $U_1 = \C^{1|n+1}$, $U_2 = \C^{2|2n}$, see e.g., \cite[Section 2]{K1}. Let $\{ e_i \}_{i \in J_s}$ denote the standard bases of $U_s$ with index sets
\begin{align*}
&J_1 = J_{1,\bar{0}} \sqcup J_{1,\bar{1}},\quad
J_{1,\bar{0}} = \{0\},\quad
J_{1,\bar{1}} = \{1, \ldots, n+1\},\\
&J_2 = J_{2,\bar{0}} \sqcup J_{2,\bar{1}},\quad
J_{2, \bar{0}} = \{+, -\},\quad
J_{2, \bar{1}} = \{1, \ldots, n, -1, \ldots, -n\},
\end{align*}
where $e_i$ is even, (resp.\ odd), if $i \in J_{s, \bar{0}}$, (resp.\ $i \in J_{s, \bar{1}}$), and $e_{i, j} \in \mathfrak{gl}(U_s)$ the elementary matrix  $e_{i, j} \cdot e_m = \delta_{j, m} e_i$, ($i, j, m \in J_s$). Then the Cartan subalgebra $\mathfrak{h}$ of $\mathfrak{g}$ is spanned by 
\begin{align*}
h_i=
\begin{cases}
-e_{0, 0} - e_{1, 1},& (i=0),\\
e_{i, i} - e_{i+1, i+1},& (i = 1, \ldots, n),
\end{cases}
\quad \text{if}\ \mathfrak{g} = \mathfrak{sl}(1|n+1),
\end{align*}
\begin{align*}
h_j=
\begin{cases}
-\frac{1}{2}(t_0+t_1),& (j=0),\\
\frac{1}{2}(t_j - t_{j+1}),& (j=1, \ldots, n-1),\\
t_n,& (j=n),
\end{cases}
\quad \text{if}\ \mathfrak{g} = \mathfrak{osp}(2|2n),
\end{align*}
where $t_0=e_{+, +} - e_{-, -}$ and $t_j=e_{j,j}-e_{-j,-j}$, ($1\leq j\leq n$).
The normalized invariant bilinear form $(\cdot | \cdot)$ on $\mathfrak{g}$ is given by $(u|v) = -\operatorname{str}(u \circ v)$, where $\operatorname{str}$ denotes the super trace. It induces an isomorphism $\nu \colon \mathfrak{h} \xrightarrow{\sim} \mathfrak{h}^*$.
Then the set of the simple roots is $\Pi=\{\alpha_i\}_{i=0}^n$ where $\alpha_i = \nu(h_i) \in \mathfrak{h}^*$. Choose a principal nilpotent element $f_\mathrm{prin}$ of the even part of $\mathfrak{g}$ and a semisimple element $x$ of $\mathfrak{g}$ such that $\operatorname{ad}_x$ defines a good $\Z$-grading $\Gamma=\Gamma_\mathrm{prin}$ on $\mathfrak{g}$ as follows:
\begin{align*}
&f_\mathrm{prin} = \sum_{i=1}^{n}e_{i+1,i},\quad
x = \sum_{i=0}^{n+1} \left( \frac{n+1}{2}-i+1 \right) e_{i,i} - e_{0,0},\quad \text{if}\ \mathfrak{g} = \mathfrak{sl}(1|n+1),\\
&f_\mathrm{prin} = \sum_{i=1}^{n-1}(e_{i+1, i} - e_{-i, -i-1}) + e_{-n, n},\quad
x = \sum_{i=0}^n \left( n-i+\frac{1}{2} \right) t_i-t_0,\quad \text{if}\ \mathfrak{g} = \mathfrak{osp}(2|2n).
\end{align*}
The weighted Dynkin diagrams corresponding to $\Gamma_{\mathrm{prin}}$ are as follows.
\begin{align*}
\quad\\
\setlength{\unitlength}{1mm}
\begin{picture}(0,0)(20,10)
\put(-38,9){$\mathfrak{g} = \mathfrak{sl}(1|n+1)$:}
\put(0,10){\circle{2}}
\put(-1.1,9.3){\footnotesize$\times$}
\put(-1,13){\footnotesize$0$}
\put(-1,5){\footnotesize$\alpha_0$}
\put(1,10.3){\line(1,0){8}}
\put(10,10){\circle{2}}
\put(9,13){\footnotesize$1$}
\put(9,5){\footnotesize$\alpha_1$}
\put(11,10.3){\line(1,0){6}}
\put(18.5,9.4){$\cdot$}
\put(20,9.4){$\cdot$}
\put(21.5,9.4){$\cdot$}
\put(24.5,10.3){\line(1,0){6}}
\put(32,10){\circle{2}}
\put(31,13){\footnotesize$1$}
\put(31,5){\footnotesize$\alpha_{n-2}$}
\put(33,10.3){\line(1,0){8}}
\put(42,10){\circle{2}}
\put(41,13){\footnotesize$1$}
\put(41,5){\footnotesize$\alpha_{n-1}$}
\put(43,10.3){\line(1,0){8}}
\put(52,10){\circle{2}}
\put(51,13){\footnotesize$1$}
\put(51,5){\footnotesize$\alpha_{n}$}
\put(54,9){.}
\end{picture}\\
\end{align*}
\begin{align*}
\quad\\
\setlength{\unitlength}{1mm}
\begin{picture}(0,0)(20,10)
\put(-38,9){$\mathfrak{g} = \mathfrak{osp}(2|2n)$:}
\put(0,10){\circle{2}}
\put(-1.1,9.3){\footnotesize$\times$}
\put(-1,13){\footnotesize$0$}
\put(-1,5){\footnotesize$\alpha_0$}
\put(1,10.3){\line(1,0){8}}
\put(10,10){\circle{2}}
\put(9,13){\footnotesize$1$}
\put(9,5){\footnotesize$\alpha_1$}
\put(11,10.3){\line(1,0){6}}
\put(18.5,9.4){$\cdot$}
\put(20,9.4){$\cdot$}
\put(21.5,9.4){$\cdot$}
\put(24.5,10.3){\line(1,0){6}}
\put(32,10){\circle{2}}
\put(31,13){\footnotesize$1$}
\put(31,5){\footnotesize$\alpha_{n-2}$}
\put(33,10.3){\line(1,0){8}}
\put(42,10){\circle{2}}
\put(41,13){\footnotesize$1$}
\put(41,5){\footnotesize$\alpha_{n-1}$}
\put(43,10.5){\line(1,0){8}}
\put(43,9.5){\line(1,0){8}}
\put(45.5,9){\Large$<$}
\put(52,10){\circle{2}}
\put(51,13){\footnotesize$1$}
\put(51,5){\footnotesize$\alpha_{n}$}
\put(54,9){.}
\end{picture}\\
\end{align*}
The Lie subalgebra $\mathfrak{g}_0=\{u\in\mathfrak{g}\mid [x,u]=0\}$ decomposes as
\begin{align*}
\mathfrak{g}_0 = \mathfrak{g}_0^\mathrm{red} \oplus \mathfrak{z},\quad 
\mathfrak{g}_0^\mathrm{red}:=\operatorname{Span}\{ e_0, h_0, h_1, f_0\},\  \mathfrak{z} :=  \operatorname{Span}\{ \widetilde{h}_2, h_3, \ldots, h_n\},
\end{align*}
where
\begin{align*}
&e_0 = e_{0,1},\quad
f_0 = e_{1,0},\quad 
\tilde{h}_2=h_2-h_0,\quad 
\text{if}\ \mathfrak{g} = \mathfrak{sl}(1|n+1),\\
&e_0 = e_{+,1} + e_{-1,-},\quad
f_0 = e_{1,+} - e_{-,-1},\quad
\tilde{h}_2=h_2-(1+\delta_{n,2})h_0,\quad 
\text{if}\ \mathfrak{g} = \mathfrak{osp}(2|2n).
\end{align*}
The Lie subalgebra $\mathfrak{z}$ is commutative and $\mathfrak{g}_0^{\mathrm{red}}$ is isomorphic to $\mathfrak{gl}(1|1)$ by
\begin{align*}
\begin{array}{cccc}
\iota:&\mathfrak{gl}(1|1)& \xrightarrow{\sim} &\mathfrak{g}_0^\mathrm{red}\\
&E_{1,1}, E_{2,2}, E_{1,2}, E_{2,1} &\mapsto&
-r(h_0 + h_1), rh_1, e_0, f_0, 
\end{array}
\end{align*}
where  $r=1$, (resp.\ 2), is the lacity of $\mathfrak{g}=\mathfrak{sl}(1|n+1)$, (resp.\ $\mathfrak{osp}(2|2n)$).
(See also Section\ \ref{sec:Wak-gl11}.) The restriction of $\tau_k$ on $\mathfrak{g}_0$ is
\begin{align*}
\tau_k(u|v) =
\begin{cases}
(k+h^\vee)(u|v), & u, v \in \mathfrak{z},\\
(k_1\kappa_1 + k_2\kappa_2)(\iota^{-1}(u)|\iota^{-1}(v)), & u, v \in \mathfrak{g}_0^\mathrm{red},\\
0, & u\in \mathfrak{z},\ v\in \mathfrak{g}_0^\mathrm{red},
\end{cases}
\end{align*}
where
\begin{align*}
k_1 = -r(k+h^\vee),\quad k_2 =r(k+h^\vee)+1,
\end{align*}
and $h^\vee=n$ is the dual Coxeter number of $\mathfrak{g}=\mathfrak{sl}(1|n+1)$, $\mathfrak{osp}(2|2n)$. (See also Section\ \ref{sec:Wak-gl11} for $\kappa_i$.)
Then the affine vertex superalgebra $V^{\tau_k}(\mathfrak{g}_0)$ decomposes as
\begin{align}\label{g1-decomposition super}
V^{\tau_k}(\mathfrak{g}_0) \simeq V^\kappa(\mathfrak{gl}(1|1)) \otimes \pi^{k+h^\vee}_\mathfrak{z},
\end{align}
where $\kappa = k_1\kappa_1 + k_2\kappa_2$ and $\pi^{k+h^\vee}_\mathfrak{z}$ is the Heisenberg vertex algebra associated with $\mathfrak{z}$ at level $k+h^\vee$, (Section\ \ref{Sect: Heisenberg}). Since
\begin{align}\label{eq; alpha super}
[\alpha_1] = \{\alpha_1, \alpha_0+\alpha_1\},\quad
[\alpha_i] = \{\alpha_i\},\quad
i = 2, \ldots, n,
\end{align}
the screening operators $\int Q_i(z)dz:=\int Q_{\alpha_i}(z)dz$ ($i=1,\ldots,n$) are 
\begin{align}\label{eq: i-screening operator for super}
\int Q_i(z)=\int S_{\alpha_i}(z) dz\colon V^{\tau_k}(\mathfrak{g}_0) \rightarrow M_{\alpha_i}.
\end{align} 

First, consider \eqref{eq: i-screening operator for super} in the case $i = 2, \ldots, n$. The orthogonal decomposition 
\begin{align*}
\mathfrak{g}_0 = \mathfrak{h}_i^\perp \oplus \mathfrak{h}_i,\quad
\mathfrak{h}_i = \C h_i,\quad
\mathfrak{h}_i^\perp = \operatorname{Span}\{ e_0, f_0, h \mid h \in \mathfrak{h}, \alpha_i(h) = 0 \}
\end{align*}
induces the decomposition of the affine vertex superalgebra $V^{\tau_k}(\mathfrak{g}_0)$
\begin{align*}
V^{\tau_k}(\mathfrak{g}_0) \simeq V^{\tau_k}(\mathfrak{h}_i^\perp) \otimes \pi^{k+h^\vee}_{\mathfrak{h}_i}.
\end{align*}
It follows from \eqref{g1-decomposition super} and \eqref{eq; alpha super} that we have an isomorphism of $V^{\tau_\kappa}(\mathfrak{g}_0)$-modules 
\begin{align*}
\begin{array}{ccc}
M_{\alpha_i}& \simeq& \widetilde{M}_{\alpha_i}:=V^{\tau_k}(\mathfrak{h}_i^\perp) \otimes \pi^{k+h^\vee}_{\mathfrak{h}_i, -\alpha_i}\\
v_{\alpha_i}&\mapsto & |0\rangle\otimes |0\rangle\otimes |-\alpha_i\rangle.
\end{array}
\end{align*}
Thus $S_{\alpha_i}(z)$ is identified with
\begin{align*}
\int S_{\alpha_i}(z)\ dz = \int :\mathrm{e}^{-\frac{1}{k+h^\vee} \int \alpha_i(z)}:\ dz\quad
\colon V^{\tau_k}(\mathfrak{h}_i^\perp) \otimes \pi^{k+h^\vee}_{\mathfrak{h}_i} \rightarrow \widetilde{M}_{\alpha_i},
\end{align*}
where $\alpha_i(z) = J^{h_i}(z)$, and, therefore,
\begin{align}\label{eq:sW-qscr-i}
\int Q_i(z)z=\int:\mathrm{e}^{-\frac{1}{k+h^\vee} \int \alpha_i(z)}:\ dz.
\end{align}
Next, consider \eqref{eq: i-screening operator for super} in the case $n=1$. By Proposition \ref{prop:Wakimoto-gl(1|1)}, we have a vertex superalgebra embedding
\begin{align*}
\rho_{\mathfrak{gl}(1|1)} \colon V^\kappa(\mathfrak{gl}(1|1)) \hookrightarrow M_{\mathfrak{gl}(1|1)} \otimes \pi^{\kappa-\kappa_2}_\chi,
\end{align*}
which is injective by Lemma \ref{lem:rho-inj} and gives an isomorphism 
\begin{align}\label{eq:gl11-scr}
\begin{split}
&V^\kappa(\mathfrak{gl}(1|1))\\
&\simeq \operatorname{Ker}\left(\int:b(z)\mathrm{e}^{\frac{1}{r(k+h^\vee)} \int (\chi_1+\chi_2)(z)}:\ dz\ : M_{\mathfrak{gl}(1|1)} \otimes \pi^{\kappa-\kappa_2}_\chi\rightarrow M_{\mathfrak{gl}(1|1)} \otimes \pi^{\kappa-\kappa_2}_{\chi,-(\chi_1+\chi_2)}\right).
\end{split}
\end{align}
for $k\neq-h^\vee$ by Proposition \ref{longexactsequence}.
Here $M_{\mathfrak{gl}(1|1)}$ is the $bc$-system vertex superalgebra and $\pi^{\kappa-\kappa_2}_\chi$ is the Heisenberg vertex algebra generated by even fields $\chi_1(z)$, $\chi_2(z)$ with OPEs \eqref{eq:chi-OPE}. 
Then it follows from  \eqref{g1-decomposition super} and the isomorphism of vertex algebras 
\begin{align*}
&\pi^{k+h^\vee}_\mathfrak{h} \xrightarrow{\sim} \pi^{\kappa-\kappa_2}_\chi \otimes \pi^{k+h^\vee}_\mathfrak{z}\\
&\alpha_0(z) \mapsto -\frac{1}{r} (\chi_1+\chi_2)(z),\
\alpha_1(z) \mapsto \frac{1}{r} \chi_2(z),\
\alpha_i(z) \mapsto J^{h_i}(z),\quad (i = 3, \ldots, n),\\
&\alpha_2(z) \mapsto 
\begin{cases}
J^{\widetilde{h}_2}(z)+(\chi_1+\chi_2)(z),& \mathrm{if}\ \mathfrak{g}=\mathfrak{sl}(1|n+1),\\
J^{\widetilde{h}_2}(z)+\frac{1+\delta_{n,2}}{2}(\chi_1+\chi_2)(z),&\mathrm{if}\ \mathfrak{g}=\mathfrak{osp}(2|2n),\\
\end{cases}
\end{align*}
that we have a vertex superalgebra embedding
\begin{align}\label{eq: g0-Wakimoto super}
\rho_{\mathfrak{g}_0} \colon V^{\tau_k}(\mathfrak{g}_0) \hookrightarrow M_{\mathfrak{gl}(1|1)} \otimes \pi^{k+h^\vee}_\mathfrak{h},
\end{align}
which gives an isomorphism for $k\neq-h^\vee$:
\begin{align}\label{eq:super-Vtau-Wakimoto}
\begin{split}
&V^{\tau_k}(\mathfrak{g}_0)\\
& \simeq \operatorname{Ker}\left( \int :b(z)\mathrm{e}^{-\frac{1}{k+h^\vee} \int \alpha_0(z)}:\ dz\quad
\colon M_{\mathfrak{gl}(1|1)} \otimes \pi^{k+h^\vee}_\mathfrak{h} \rightarrow M_{\mathfrak{gl}(1|1)} \otimes \pi^{k+h^\vee}_{\mathfrak{h}, -\alpha_0} \right).
\end{split}
\end{align}
Then it gives a $V^{\tau_k}(\mathfrak{g}_0)$-module structure on $\widetilde{M}_{\alpha_1}:=M_{\mathfrak{gl}(1|1)} \otimes \pi^{k+h^\vee}_{\mathfrak{h}, -\alpha_1}$. We have an isomorphism
\begin{align}\label{eq: isom for Wakimoto}
\widetilde{M}_{\alpha_1} \simeq \widehat{V}^{+,\kappa}_{\frac{3}{2}, -1}\otimes \pi_{\mathfrak{z}}^{k+h^\vee}
\end{align}
as $V^{\tau_k}(\mathfrak{g}_0)\simeq V^\kappa(\mathfrak{gl}(1|1))\otimes \pi_{\mathfrak{z}}^{k+h^\vee}$-module for $k\notin \mathbb{Q}$ by Lemma \ref{Induced modules at generic level} (1) and Proposition \ref{generic gl11-Wakimoto}.
Note that as a $V^{\tau_k}(\mathfrak{g}_0)$-module, $\widetilde{M}_{\alpha_1}$ is generated by a $(\mathfrak{g}_0[t] \oplus \C K)$-submodule
\begin{align*}
\widetilde{\C}^{[\alpha_1]} = \C |-\alpha_1\rangle \oplus \C c_{(-1)}|-\alpha_1\rangle.
\end{align*}
\begin{lemma}\label{lem:super-M-isom}
For $k\notin\mathbb{Q}$,
we have $M_{\alpha_1} \simeq \widetilde{M}_{\alpha_1}$ as $V^{\tau_k}(\mathfrak{g}_0)$-modules.
\begin{proof}
The assertion follows from \eqref{eq: isom for Wakimoto} and the $(\mathfrak{g}_0[t] \oplus \C K)$-module isomorphism
\begin{align*}
\C^{[\alpha_1]} \xrightarrow{\sim} \widetilde{\C}^{[\alpha_1]},\quad
v_{\alpha_1} \mapsto |-\alpha_1\rangle,\quad
v_{\alpha_0+\alpha_1} \mapsto -c_{(-1)}|-\alpha_1\rangle.
\end{align*}
\end{proof}
\end{lemma}
Under the realization \eqref{eq:super-Vtau-Wakimoto}, Lemma \ref{lem:super-M-isom} implies that the intertwining operator $S_{\alpha_1}(z)$ is identified with
\begin{align*}
S_{\alpha_1}(z)=\ :\mathrm{e}^{-\frac{1}{k+h^\vee} \int \alpha_1(z)}:\quad:
\mathrm{Im}(\rho_{\mathfrak{g}_0})\rightarrow \widetilde{M}_{\alpha_1}((z)),
\end{align*}
and thus
\begin{align}\label{eq:sW-scr-1}
\int Q_1(z) dz=\int:\mathrm{e}^{-\frac{1}{k+h^\vee} \int \alpha_1(z)}:dz.
\end{align}
We also note that $S_{\alpha_1+\alpha_2}(z) = - :c(z)\mathrm{e}^{-\frac{1}{k+h^\vee} \int \alpha_1(z)}:$.

By \eqref{eq:W-screening}, \eqref{eq:sW-qscr-i} and \eqref{eq:sW-scr-1}, we conclude
\begin{align*}
\mathcal{W}^k (\mathfrak{g})
&\simeq \bigcap_{i=1}^n \operatorname{Ker}\left( \int Q_i(z)\ dz \colon V^{\tau_k}(\mathfrak{g}_0) \rightarrow M_{\alpha_i} \right)\\
&\simeq \bigcap_{i=1}^n \operatorname{Ker}\left( \int :\mathrm{e}^{-\frac{1}{k+h^\vee} \int \alpha_i(z)}:\ dz\quad
\colon \mathrm{Im}(\rho_{\mathfrak{g}_0}) \rightarrow \widetilde{M}_{\alpha_i} \right).
\end{align*}
The composition $\Psi_1$ of \eqref{eq: Miura-map} and \eqref{eq: g0-Wakimoto super} gives an vertex superalgebra embedding
of the principal super $\mathcal{W}$-algebra 
$\mathcal{W}^k (\mathfrak{g}) = \mathcal{W}^k (\mathfrak{g},f_\mathrm{prin} ;\Gamma_\mathrm{prin})$
\begin{align}\label{Psi_1}
\Psi_1: \mathcal{W}^k (\mathfrak{g})\hookrightarrow M_{\mathfrak{gl}(1|1)} \otimes \pi^{k+h^\vee}_\mathfrak{h}.
\end{align}
Now \eqref{eq:super-Vtau-Wakimoto} implies the following realization of the image $\mathrm{Im}(\Psi_1)$.
\begin{theorem}\label{thm:super-W}
If $k$ is generic, then we have
\begin{align*}
\mathrm{Im}(\Psi_1)= \bigcap_{i = 0}^n \operatorname{Ker} \int Q_i(z) dz
\end{align*}
for $\mathfrak{g} = \mathfrak{sl}(1|n+1)$, $\mathfrak{osp}(2|2n)$, where
\begin{align*}
Q_0(z) = :b(z)\mathrm{e}^{-\frac{1}{k+h^\vee} \int \alpha_0(z)}:,\quad
Q_i(z) = :\mathrm{e}^{-\frac{1}{k+h^\vee} \int \alpha_i(z)}:\ 
(i = 1, \ldots, n).
\end{align*}
\end{theorem}

\section{Dualities in coset vertex algebras}\label{sec:newduality}

\subsection{Coset vertex algebras}
Given a vertex superalgebra $V$ and a subalgebra $W\subset V$, the subspace
\begin{align*}
\mathrm{Com}(W,V):=\{a\in V\mid \forall b\in W,\ Y(b,z)Y(a,w)\sim 0\},
\end{align*} 
is a vertex subalgebra, called the coset vertex (super)algebra of the pair $(V,W)$.

For $\mathfrak{g}=\mathfrak{sl}_{n+1},\ \mathfrak{so}_{2n+1}$, define a field $H_1(z)$ on $M_{\mathfrak{sl}_2} \otimes \pi^{k+h^\vee}_\mathfrak{h}$ by
\begin{align}\label{eq: H_1}
\begin{split}
&H_1(z) = \omega^\vee_1(z)-:\beta(z)\gamma(z):,\\
&\omega^\vee_1=
\begin{cases}
\displaystyle \frac{1}{n+1}\sum_{i=1}^n (n-i+1)\alpha_i,& \mathrm{if}\ \mathfrak{g} = \mathfrak{sl}_{n+1},\\
\displaystyle \sum_{i=1}^n \alpha_i, & \mathrm{if}\ \mathfrak{g} = \mathfrak{so}_{2n+1}.
\end{cases}
\end{split}
\end{align}
Note that $\omega^\vee_1$ represents the first fundamental coweight of $\mathfrak{g}$. By direct computations, one can show that $H_1(z)$ lies in the kernels of the screening operators $\int Q_i(z) dz$, ($1\leq i\leq n$), and thus it defines a field on $\mathcal{W}^k (\mathfrak{g}, f_\mathrm{sub})$ by Theorem \ref{thm:sub-W}. In fact, $H_1(z)$ corresponds to the field $J^{\{ \omega_1^\vee\}}(z)$ on $\mathcal{W}^k (\mathfrak{g}, f_\mathrm{sub})$ in \cite[Theorem 2.1 (a)]{KW} via the Miura map. We shall write $H_1(z)$ for $J^{\{ \omega_1^\vee\}}(z)$ by abuse of notation.
Let $\pi_{H_1}$ be the Heisenberg vertex subalgebra of $\mathcal{W}^k (\mathfrak{g}, f_\mathrm{sub})$ generated by $H_1(z)$ and consider the coset vertex algebra
$\mathrm{Com}\left( \pi_{H_1}, \mathcal{W}^k (\mathfrak{g}, f_\mathrm{sub}) \right)$.
We will describe it in terms of screening operators for generic $k$.

Let $L_1=\Z x\oplus \Z y$ a $\Z$-lattice equipped with a bilinear form $(\cdot|\cdot)$ given by $(ax+by|cx+dy) = ac-bd$, $\pi_{L_1}$ the Heisenberg vertex algebra associate with the commutative Lie algebra $L_1 \otimes_\Z \C$, (Section\ \ref{Sect: Heisenberg}), and 
$$V_{L_1}:=\bigoplus_{(m,n)\in\Z^2}\pi_{L_1,mx+ny}$$
the lattice vertex superalgebra associated with $L_1$ and the vertex subalgera
\begin{align*}
V_{x+y} := \bigoplus_{n \in \Z} \pi_{L_1, n(x+y)}.
\end{align*}
The Friedan-Martinec-Shenker bosonization \cite[Chapter 7]{F2} gives a vertex algebra embedding
\begin{align*}
\begin{split}
\Upsilon_2 \colon &M_{\mathfrak{sl}_2} \hookrightarrow V_{x+y}\\
&\beta(z) \mapsto :\mathrm{e}^{\int (x+y)(z)}:,\\
&\gamma(z) \mapsto -:x(z)\mathrm{e}^{-\int (x+y)(z)}:,
\end{split}
\end{align*}
whose image is equal to the kernel of the screening operator $\int :\mathrm{e}^{\int x(z)}:dz$
\begin{align*}
\operatorname{Im}(\Upsilon_2) = \Ker \left(\int :\mathrm{e}^{\int x(z)}:\ dz\quad: V_{x+y}\rightarrow \bigoplus_{n \in \Z} \pi_{L_1, (n+1)x+ny}\right).
\end{align*}
By composing it with \eqref{eq: Upsilon_1}, we obtain a vertex algebra embedding
\begin{align}\label{eq:sub-W-free}
\Upsilon:=\Upsilon_2 \circ \Upsilon_1 \colon \mathcal{W}^k (\mathfrak{g}, f_\mathrm{sub}) \hookrightarrow V_{x+y} \otimes \pi^{k+h^\vee}_\mathfrak{h},
\end{align}
whose image for generic $k$ coincides with
\begin{align}\label{eq:sub-W-image}
\begin{split}
\operatorname{Im}(\Upsilon)
= &\Ker \int :\mathrm{e}^{\int x(z)}:\ dz\ \cap\ \Ker\int:\mathrm{e}^{-\frac{1}{k+h^\vee} \int (\alpha_1 - (k+h^\vee)(x+y))(z)}:dz\\
& \cap\ \bigcap_{i=2}^n \Ker\int :\mathrm{e}^{-\frac{1}{k+h^\vee} \int \alpha_i(z)}:dz.
\end{split}
\end{align}
Let $\pi_{\widetilde{\alpha}}^{k+h^\vee}\subset \pi_{L_1} \otimes \pi^{k+h^\vee}_\mathfrak{h}$ be the Heisenberg vertex subalgebra generated by 
\begin{equation*}
\begin{split}
&\widetilde{\alpha}_0(z) = x(z),\quad \widetilde{\alpha}_1(z) = (\alpha_1 - (k+h^\vee)(x+y))(z),\\
&\widetilde{\alpha}_i(z) = \alpha_i(z),\quad (i=2, \ldots, n-1),\\
&\widetilde{\alpha}_n(z) = r \alpha_n(z),
\end{split}
\end{equation*} 
where $r=1$, (resp.\ 2), is the lacity of $\mathfrak{g}=\mathfrak{sl}_{n+1}$, (resp.\ $\mathfrak{so}_{2n+1}$).
Since $\Upsilon_2$ induces an isomorphism
\begin{align*}
\begin{array}{ccl}
\pi_{H_1} \otimes \pi_{\widetilde{\alpha}}^{k+h^\vee}& \xrightarrow{\sim} &\pi_{L_1} \otimes \pi^{k+h^\vee}_\mathfrak{h}\\
1\otimes \widetilde{\alpha}_i(z)&\mapsto &1\otimes \widetilde{\alpha}_i(z),\quad (i=0,\ldots,n)\\
H_1(z)\otimes 1 &\mapsto &\omega^\vee_1(z)-y(z),
\end{array}
\end{align*}
we have $\mathrm{Com}\left( \pi_{H_1}, V_{x+y} \otimes \pi^{k+h^\vee}_\mathfrak{h} \right) = \pi_{\widetilde{\alpha}}^{k+h^\vee}$.
Therefore, $\Upsilon$ restricts to 
\begin{align}\label{eq: Upsilon}
\Upsilon \colon \mathrm{Com}\left( \pi_{H_1}, \mathcal{W}^k (\mathfrak{g}, f_\mathrm{sub}) \right) \hookrightarrow \pi_{\widetilde{\alpha}}^{k+h^\vee},
\end{align}
and thus we obtain the following proposition.
\begin{proposition}\label{prop:sub-W-coset}
For generic $k$, 
\begin{align*}
\mathrm{Com}\left( \pi_{H_1}, \mathcal{W}^k (\mathfrak{g}, f_\mathrm{sub}) \right) \simeq
&\Ker_{\pi_{\widetilde{\alpha}}^{k+h^\vee}} \int :\mathrm{e}^{\int \widetilde{\alpha}_0(z)}:\ dz\ \cap\ \bigcap_{i=1}^{n-1} \Ker_{\pi_{\widetilde{\alpha}}^{k+h^\vee}}\int :\mathrm{e}^{-\frac{1}{k+h^\vee} \int \widetilde{\alpha}_i(z)}:dz\\
&\cap\ \Ker_{\pi_{\widetilde{\alpha}}^{k+h^\vee}}\int:\mathrm{e}^{-\frac{1}{r(k+h^\vee)} \int \widetilde{\alpha}_n(z)}:dz.
\end{align*}
\end{proposition}

Similarly, for $\mathfrak{g}=\mathfrak{sl}(1|n+1),\ \mathfrak{osp}(2|2n)$, consider the field $H_2(z)$ on $M_{\mathfrak{gl}(1|1)} \otimes \pi^{k+h^\vee}_\mathfrak{h}$ defined by
\begin{align}\label{eq: H_2}
\begin{split}
&H_2(z) =\omega^\vee_0(z)+ :b(z)c(z):,\\
&\omega^\vee_0=
\begin{cases}
\displaystyle-\frac{1}{n}\sum_{i=0}^n (n-i+1) \alpha_i,& \mathrm{if}\ \mathfrak{g} = \mathfrak{sl}(1|n+1),\\
\displaystyle -2\sum_{i=0}^{n-1} \alpha_i - \alpha_n, & \mathrm{if}\ \mathfrak{g} = \mathfrak{osp}(2|2n).
\end{cases}
\end{split}
\end{align}
Note that $\omega^\vee_0$ represents the $0$-th fundamental coweight of $\mathfrak{g}$.
By direct computations, one can show that $H_2(z)$ lies in the kernels of the screening operators $\int Q_i(z) dz$, ($0\leq i\leq n$), and thus $H_2(z)$ defines a field on $\mathcal{W}^k (\mathfrak{g})$ by Theorem \ref{thm:super-W}. In fact, $H_2(z)$ corresponds to the field $J^{\{ \omega_0^\vee\}}(z)$ on $\mathcal{W}^k (\mathfrak{g})$ in \cite[Theorem 2.1 (a)]{KW} via the Miura map. We shall write $H_2(z)$ for $J^{\{ \omega_0^\vee\}}(z)$ by abuse of notation.
Let $\pi_{H_2}$ be the Heisenberg vertex subalgebra of $\mathcal{W}^k (\mathfrak{g})$ generated by $H_2(z)$ and consider the coset vertex algebra
$\mathrm{Com}\left( \pi_{H_2}, \mathcal{W}^k (\mathfrak{g}) \right)$.
We will describe it in terms of screening operators for generic $k$.

Let $\Z=\Z\phi$ a $\Z$-lattice equipped with a bilinear form $(m\phi|n\phi)=mn$, $\pi_{\Z}$ the Heisenberg vertex algebra associate with the commutative Lie algebra $\Z\otimes_\Z\C$, 
(Section\ \ref{Sect: Heisenberg}), and 
$$V_{\Z}:=\bigoplus_{m\in\Z}\pi_{\Z,m\phi}$$
the lattice vertex superalgebra associated with $\Z$. 
By the boson-fermion correspondence, e.g., \cite[Chapter 5]{FBZ}, we have an isomorphism 
\begin{align*}
\begin{array}{ccc}
\Psi_2 \colon M_{\mathfrak{gl}(1|1)}& \xrightarrow{\sim} &V_\Z\\
b(z) &\mapsto &:\mathrm{e}^{\int \phi(z)}:,\\
c(z) &\mapsto &:\mathrm{e}^{\int -\phi(z)}:.
\end{array}
\end{align*}
Composing it with \eqref{Psi_1}, we obtain a vertex algebra embedding 
\begin{align}\label{eq:super-W-free}
\Psi:=\Psi_2 \circ \Psi_1 \colon  \mathcal{W}^k (\mathfrak{g}) \hookrightarrow V_\Z \otimes \pi^{k+h^\vee}_\mathfrak{h},
\end{align}
whose image for generic $k$ coincides with
\begin{align}\label{eq:super-W-image}
\operatorname{Im}(\Psi) = \Ker\int:\mathrm{e}^{-\frac{1}{k+h^\vee} \int (\alpha_0-(k+h^\vee)\phi)(z)}:dz\ \cap\ \bigcap_{i=1}^n \Ker\int :\mathrm{e}^{-\frac{1}{k+h^\vee} \int \alpha_i(z)}:dz.
\end{align}
Let $\pi_{\widetilde{\beta}}^{k+h^\vee}\subset V_\Z \otimes \pi^{k+h^\vee}_\mathfrak{h}$ be the Heisenberg vertex subalgebra generated by 
\begin{equation*}
\begin{split}
&\widetilde{\beta}_0(z) = -\frac{1}{k+h^\vee}\alpha_0(z)+\phi(z),\\
&\widetilde{\beta}_i(z) = -\frac{1}{k+h^\vee}\alpha_i(z),\quad (i=1,\ldots, n).
\end{split}
\end{equation*} 
Since $\Psi_2$ induces an isomorphism
\begin{align*}
\begin{array}{ccl}
\pi_{H_2} \otimes \pi_{\widetilde{\beta}}^{k+h^\vee}& \xrightarrow{\sim} &\pi_{\Z} \otimes \pi^{k+h^\vee}_\mathfrak{h}\\
1\otimes \widetilde{\beta}_i(z)&\mapsto &\widetilde{\beta}_i(z),\quad (i=0,\ldots, n)\\
H_2(z)\otimes 1 &\mapsto &\omega^\vee_0(z)+\phi(z),
\end{array}
\end{align*}
we have $\mathrm{Com}\left( \pi_{H_2}, V_\Z \otimes \pi^{k+h^\vee}_\mathfrak{h} \right) = \pi_{\widetilde{\beta}}^{k+h^\vee}$. Therefore, $\Psi$ restricts to 
\begin{align}\label{eq: Psi}
\Psi \colon \mathrm{Com}\left( \pi_{H_2}, \mathcal{W}^k (\mathfrak{g}) \right) \hookrightarrow \pi_{\widetilde{\beta}}^{k+h^\vee},
\end{align}
and thus we obtain the following proposition.
\begin{proposition}\label{prop:super-W-coset}
For generic $k$, we have an isomorphism
\begin{align*}
\mathrm{Com}\left( \pi_{H_2}, \mathcal{W}^k (\mathfrak{g}) \right) \simeq \bigcap_{i=0}^n \Ker_{\pi_{\widetilde{\beta}}^{k+h^\vee}}\int :\mathrm{e}^{\int \widetilde{\beta}_i(z)}:dz
\end{align*}
where $\mathfrak{g} = \mathfrak{sl}(1|n+1)$ or $ \mathfrak{osp}(2|2n)$.
\end{proposition}
\subsection{Feigin-Frenkel type duality}\label{Duality}
Let $(\mathfrak{g}_1,\mathfrak{g}_2) = (\mathfrak{sl}_{n+1},\mathfrak{sl}(1|n+1))$ or $(\mathfrak{so}_{2n+1}, \mathfrak{osp}(2|2n))$. Denote by $r$ the lacity of $\mathfrak{g}_1$, (which is equal to that of $\mathfrak{g}_2$), and by $h^\vee_i$ the dual Coxeter number of $\mathfrak{g}_i$. Then 
\begin{align*}
(r,h^\vee_1,h^\vee_2)=
\begin{cases}
(1,n+1,n),& \mathrm{if}\ (\mathfrak{g}_1,\mathfrak{g}_2)=(\mathfrak{sl}_{n+1}, \mathfrak{sl}(1|n+1)),\\
(2,2n-1,n),& \mathrm{if}\ (\mathfrak{g}_1,\mathfrak{g}_2)=(\mathfrak{so}_{2n+1}, \mathfrak{osp}(2|2n)).
\end{cases}
\end{align*}
For $k_1\in \C\backslash \{-h^\vee_1\}$, define $k_2 \in \C$ by the formula
\begin{align}\label{eq: ell}
r(k_1+h^\vee_1)(k_2+h^\vee_2)=1.
\end{align}
\begin{theorem}\label{thm:coset}
Let $k_1, k_2\in \C$ be generic satisfying \eqref{eq: ell}. Then we have an isomorphism
\begin{align*}
&\mathrm{Com}\left( \pi_{H_1}, \mathcal{W}^{k_1} (\mathfrak{g}_1, f_\mathrm{sub}) \right) \simeq \mathrm{Com}\left( \pi_{H_2}, \mathcal{W}^{k_2} (\mathfrak{g}_2) \right),
\end{align*}
for $(\mathfrak{g}_1,\mathfrak{g}_2)=(\mathfrak{sl}_{n+1}, \mathfrak{sl}(1|n+1))$, or $(\mathfrak{so}_{2n+1}, \mathfrak{osp}(2|2n))$. 
\end{theorem}
\begin{proof}
For $k_1, k_2\in\C$ satisfying \eqref{eq: ell}, we have an isomorphism
\begin{align*}
\begin{array}{cccc}
\pi_{\widetilde{\alpha}}^{k_1+h^\vee_1}& \rightarrow &\pi_{\widetilde{\beta}}^{k_2+h^\vee_2}&\\
\widetilde{\alpha}_i(z)& \mapsto &\widetilde{\beta}_i(z),&(
i = 0, \ldots, n),
\end{array}
\end{align*}
since both of the Gram matrices for $\{\widetilde{\alpha}_i\}_{i=0}^n$ and $\{\widetilde{\beta}_i\}_{i=0}^n$ are 
\begin{align*}
\bordermatrix{
             & 0 & 1 & 2 & \cdots & n-2    & n-1    & n       \cr
\hspace{3.5mm} 0 & 1  & -\mathcal{K}    & 0         & \cdots & 0        & 0        &  0        \cr
\hspace{3.5mm} 1 & -\mathcal{K}    & 2\mathcal{K}  & -\mathcal{K}     & \cdots & 0         & 0        & 0         \cr
\hspace{3.5mm} 2 & 0    & -\mathcal{K}  & 2\mathcal{K}     & \cdots & 0         & 0        & 0         \cr
\hspace{3.5mm} \vdots & \vdots & \vdots & \vdots  & \ddots & \vdots & \vdots & \ \vdots \cr
n-2   & 0        & 0         & 0          & \cdots & 2\mathcal{K}  & -\mathcal{K}    & 0         \cr
n-1     &  0      & 0          & 0         & \cdots & -\mathcal{K}    & 2\mathcal{K}         & -r\mathcal{K}         \cr
\hspace{3.5mm} n      & 0       & 0          & 0          & \cdots & 0         & -r\mathcal{K}         & 2r\mathcal{K}          \cr
}_{,}
\end{align*}
where $\mathcal{K} = k_1+h^\vee_1$. By applying the Feigin-Frenkel duality for the Virasoro vertex algebras (cf.\ \cite[Chapter 15]{FBZ}), we have
\begin{align*}
&\Ker_{\pi_{\widetilde{\alpha}}^{k_1+h_1^\vee}}\int :\mathrm{e}^{-\frac{1}{k_1+h^\vee_1} \int \widetilde{\alpha}_i(z)}:dz\ =\ \Ker_{\pi_{\widetilde{\alpha}}^{k_1+h^\vee_1}}\int :\mathrm{e}^{\int \widetilde{\alpha}_i(z)}:dz,\quad (i=1, \ldots, n-1),\\
&\Ker_{\pi_{\widetilde{\alpha}}^{k_1+h^\vee_1}}\int:\mathrm{e}^{-\frac{1}{r(k_1+h^\vee_1)} \int \widetilde{\alpha}_n(z)}:dz\ =\ \Ker_{\pi_{\widetilde{\alpha}}^{k_1+h^\vee_1}}\int:\mathrm{e}^{\int \widetilde{\alpha}_n(z)}:dz,
\end{align*}
for generic $k_1$.
Hence, for generic $k_1,k_2\in\C$ satisfying \eqref{eq: ell},
\begin{align*}
\mathrm{Com}\left( \pi_{H_1}, \mathcal{W}^{k_1} (\mathfrak{g}_1, f_\mathrm{sub}) \right) &\simeq \bigcap_{i=0}^n \Ker_{\pi_{\widetilde{\alpha}}^{k_1+h^\vee_1}}\int:\mathrm{e}^{\int \widetilde{\alpha}_i(z)}:dz
\simeq \bigcap_{i=0}^n \Ker_{\pi_{\widetilde{\beta}}^{k_2+h^\vee_2}}\int:\mathrm{e}^{\int \widetilde{\beta}_i(z)}:dz\\
&\simeq \mathrm{Com}\left( \pi_{H_2}, \mathcal{W}^{k_2} (\mathfrak{g}_2) \right)
\end{align*}
by Proposition \ref{prop:sub-W-coset} and Proposition \ref{prop:super-W-coset}.
\end{proof}

\subsection{Kazama-Suzuki Coset}
Let $(\mathfrak{g}_1, \mathfrak{g}_2)$ be as in Section\ \ref{Duality}.
Let $\pi_{\widetilde{H}_1}\subset \mathcal{W}^k (\mathfrak{g}_1, f_\mathrm{sub})$ $\otimes V_\Z$ be the Heisenberg vertex subalgebra generated by the field
\begin{align}\label{eq: tilde{H}_1}
\widetilde{H}_1(z) :=\phi(z) - H_1(z)=-\omega^\vee_1(z)+y(z)+\phi(z)
\end{align}
(see \eqref{eq: H_1}).
Next, consider the lattice $\Z\sqrt{-1}\subset \mathbb{C}$, i.e., the lattice $\Z\psi$, spanned by $\psi$, equipped with a bilinear form $(\cdot|\cdot)$, which satisfies $(m\psi|n\psi)=-mn$. Let $\pi_{\Z\sqrt{-1}}$ be the Heisenberg vertex algebra associated with the abelian Lie algebra $\C\otimes_{\Z}\Z\sqrt{-1}$ and
\begin{align*}
V_{\Z\sqrt{-1}} := \bigoplus_{n \in \Z}\pi_{\Z\sqrt{-1}, n\psi}.
\end{align*}
the lattice vertex superalgebra associated with $\Z\sqrt{-1}$.
Let $\pi_{\widetilde{H}_2}\subset \mathcal{W}^k (\mathfrak{g}_2) \otimes V_{\Z\sqrt{-1}}$, be the Heisenberg vertex subalgebra generated by the field
\begin{align}\label{eq: tilde{H}_2}
\widetilde{H}_2(z) :=\psi(z) + H_2(z)=\omega^\vee_0(z)+\phi(z)+\psi(z)
\end{align}
(see \eqref{eq: H_2}).
\begin{theorem}\label{thm:KS}
Let $k_1, k_2\in\C$ be generic satisfying \eqref{eq: ell}. Then we have isomorphisms
\begin{enumerate}
\item $\mathcal{W}^{k_1} (\mathfrak{g}_1, f_\mathrm{sub}) \simeq \mathrm{Com}\left( \pi_{\widetilde{H}_2}, \mathcal{W}^{k_2} (\mathfrak{g}_2) \otimes V_{\Z\sqrt{-1}} \right)$,
\item $\mathcal{W}^{k_2} (\mathfrak{g}_2) \simeq \mathrm{Com}\left( \pi_{\widetilde{H}_1},  \mathcal{W}^{k_1} (\mathfrak{g}_1, f_\mathrm{sub}) \otimes V_\Z \right)$,
\end{enumerate}for $(\mathfrak{g}_1,\mathfrak{g}_2)=(\mathfrak{sl}_{n+1}, \mathfrak{sl}(1|n+1))$, or $(\mathfrak{so}_{2n+1}, \mathfrak{osp}(2|2n))$.
\begin{proof}
Recall that $r$ is the lacity of $\mathfrak{g}_1$ and $h^\vee_i$ is the dual Coxeter number of $\mathfrak{g}_i$ as in the beginning of Section \ref{Duality}. Denote by $\mathfrak{h}_i$ a Cartan subalgebra of $\mathfrak{g}_i$. Let $\{\alpha_i\}_{i=1}^n$, (resp.\ $\{\beta_i\}_{i=0}^n$), be the set of simple roots of $\mathfrak{g}_1$, (resp.\ $\mathfrak{g}_2$), and $\alpha_i(z)=Y(\alpha_{i(-1)}|0\rangle,z)$ the corresponding fields on $\pi^{k_1+h^\vee_1}_{\mathfrak{h}_1}$, (resp.\ $\beta_i(z)=Y(\beta_{i(-1)}|0\rangle,z)$ on $\pi^{k_2+h^\vee_2}_{\mathfrak{h}_2}$).

First, we will show (1). By \eqref{eq:super-W-free}, we have a vertex superalgebra embedding
\begin{align}\label{eq: Kazama-Suzuki emb}
\Psi \otimes \operatorname{id} \colon\quad \mathcal{W}^{k_2} (\mathfrak{g}_2) \otimes V_{\Z\sqrt{-1}} \hookrightarrow V_\Z \otimes \pi^{k_2+h^\vee_2}_{\mathfrak{h}_2} \otimes V_{\Z\sqrt{-1}}.
\end{align}
Let $V_{\Z(\phi+\psi)}\subset V_\Z \otimes V_{\Z\sqrt{-1}}$ be the lattice vertex subalgebra corresponding to the sublattice $\Z(\phi+\psi)\subset \Z+\Z\sqrt{-1}$ and 
\begin{align*}
V_{X+Y}\subset V_\Z \otimes \pi^{k_2+h^\vee_2}_{\mathfrak{h}_2} \otimes V_{\Z\sqrt{-1}}
\end{align*}
the vertex subalgebra generated by $V_{\Z(\phi+\psi)}$ and the Heisenberg vertex subalgebra $\pi_{X,Y}$ generated by the fields
\begin{align*}
X(z) = -\frac{1}{k_2+h^\vee_2}\beta_0(z) + \phi(z),\quad
Y(z) = \frac{1}{k_2+h^\vee_2}\beta_0(z) + \psi(z).
\end{align*}
Let $\pi_A\subset V_\Z \otimes \pi^{k_2+h^\vee_2}_{\mathfrak{h}_2} \otimes V_{\Z\sqrt{-1}}$ be the Heisenberg vertex subalgebra generated by the fields
\begin{align*}
&A_i(z) = 
\begin{cases}
r \beta_1(z) - \phi(z) - \psi(z),& i=1,\\
r \beta_i(z),&i = 2, \ldots, n-1,\\
\beta_n(z),& i=n.
\end{cases}
\end{align*}
It follows from 
\begin{align*}
X(z)A_i(w) \sim 0 \sim Y(z)A_i(w),\quad
i=1, \ldots, n,
\end{align*}
that $V_{X+Y}\otimes \pi_A\subset  V_\Z \otimes \pi^{k_2+h^\vee_2}_{\mathfrak{h}_2} \otimes V_{\Z\sqrt{-1}}$.
By \eqref{eq: tilde{H}_2}, we have
\begin{align*}
\mathrm{Com}\left( \pi_{\widetilde{H}_2}, V_\Z \otimes \pi^{k_2+h^\vee_2}_{\mathfrak{h}_2} \otimes V_{\Z\sqrt{-1}} \right) = V_{X+Y} \otimes \pi_A,
\end{align*}
and thus \eqref{eq: Kazama-Suzuki emb} implies
\begin{align*}
\mathrm{Com}\left( \pi_{\widetilde{H}_2}, \mathcal{W}^{k_2} (\mathfrak{g}_2) \otimes V_{\Z\sqrt{-1}} \right) \hookrightarrow V_{X+Y} \otimes \pi_A,
\end{align*}
whose image for generic $k_2$ coincides with
\begin{align*}
&\Psi\otimes \mathrm{id}\left(\mathrm{Com}\left( \pi_{\widetilde{H}_2}, \mathcal{W}^{k_2} (\mathfrak{g}_2) \otimes V_{\Z\sqrt{-1}} \right)\right)\\
=&\Ker\int:\mathrm{e}^{\int X(z)}:dz\ \cap\ \Ker\int:\mathrm{e}^{-\frac{1}{r(k_2+h^\vee_2)} \int (A_1+X+Y)(z)}:dz\\
&\cap\ \bigcap_{i=2}^{n-1} \Ker\int :\mathrm{e}^{-\frac{1}{r(k_2+h^\vee_2)} \int A_i(z)}:dz\ \cap\ \Ker\int :\mathrm{e}^{-\frac{1}{k_2+h^\vee_2} \int A_n(z)}:dz
\end{align*}
by \eqref{eq:super-W-image}. Since $(A_1+X+Y)_{(0)}|n(\phi+\psi)\rangle=0$, we have
\begin{align*}
\begin{split}
&\Ker\int :\mathrm{e}^{-\frac{1}{r(k_2+h^\vee_2)} \int (A_1+X+Y)(z)}:dz\\
&=\bigoplus_{n\in\Z}\left(\Ker_{\pi_{X,Y}}\int :\mathrm{e}^{-\frac{1}{r(k_2+h^\vee_2)} \int (A_1+X+Y)(z)}:dz\right)_{(-1)}|n(\phi+\psi)\rangle\\
&=\bigoplus_{n\in\Z}\left(\Ker_{\pi_{X,Y}}\int :\mathrm{e}^{\int (A_1+X+Y)(z)}:dz\right)_{(-1)}|n(\phi+\psi)\rangle\\
&=\Ker\int :\mathrm{e}^{\int (A_1+X+Y)(z)}:dz.
\end{split}
\end{align*}
by the Feigin-Frenkel duality for the Virasoro vertex algebras, (cf.\ \cite[Chapter 15]{FBZ}). 
Similarly, we have
\begin{align*}
\begin{split}
&\Ker\int :\mathrm{e}^{-\frac{1}{r(k_2+h^\vee_2)} \int A_i(z)}:dz=\Ker\int :\mathrm{e}^{\int A_i(z)}:dz,\quad i =1,\ldots,n-1,\\
&\Ker\int :\mathrm{e}^{-\frac{1}{k_2+h^\vee_2} \int A_n(z)}:dz=\Ker\int :\mathrm{e}^{\int A_n(z)}:dz.
\end{split}
\end{align*}
Therefore, we have
\begin{align*}
&(\Psi\otimes \mathrm{id})\left(\mathrm{Com}\left( \pi_{\widetilde{H}_2}, \mathcal{W}^{k_2} (\mathfrak{g}_2) \otimes V_{\Z\sqrt{-1}} \right)\right)\\
= &\Ker\int:\mathrm{e}^{\int X(z)}:dz\ \cap\ \Ker\int:\mathrm{e}^{\int (A_1+X+Y)(z)}:dz\ \cap\ \bigcap_{i=2}^n \Ker\int :\mathrm{e}^{\int A_i(z)}:dz.
\end{align*}
Now (1) follows from the above equation with \eqref{eq:sub-W-image} and the isomorphism
\begin{align*}
&V_{X+Y} \otimes \pi_A \xrightarrow{\sim} V_{x+y} \otimes \pi^{k_1+h^\vee_1}_{\mathfrak{h}_1}\\
&X(z) \mapsto x(z),\quad
Y(z) \mapsto y(z),\quad
A_i(z) \mapsto -\frac{1}{k_1+h^\vee_1}\alpha_i(z),\quad
i = 1, \ldots, n.
\end{align*}

Next, we will show (2) in the same way as the proof of (1). By \eqref{eq:sub-W-free}, we have a vertex superalgebra embedding
\begin{align*}
\Upsilon\otimes \mathrm{id}: \mathcal{W}^{k_1} (\mathfrak{g}_1, f_\mathrm{sub}) \otimes V_\Z \hookrightarrow V_{x+y} \otimes \pi^{k_1+h^\vee_1}_{\mathfrak{h}_1} \otimes V_\Z.
\end{align*}
Let $V_{\widetilde{\Z}}$ be the vertex superalgebra generated by the fields $:\mathrm{e}^{\int \pm\widetilde{\phi}(z)}:$ where
\begin{align*}
\widetilde{\phi}(z) = x(z) + y(z) + \phi(z),
\end{align*}
and $\pi_B$ the Heisenberg vertex subalgebra generated by the fields
\begin{align*}
&B_0(z) = -y(z)-\phi(z),\quad B_1(z) = \alpha_1(z) - (k_1+h^\vee_1)(x+y)(z),\\
&B_i(z) = \alpha_i(z),\ (i = 2, \ldots, n-1),\quad B_n(z) = r \alpha_n(z).
\end{align*}
Then we have $\mathrm{Com}\left( \pi_{\widetilde{H}_1}, V_{x+y} \otimes \pi^{k_1+h^\vee_1}_{\mathfrak{h}_1} \otimes V_\Z \right) = V_{\widetilde{\Z}} \otimes \pi_B$ and thus
\begin{align*}
\Upsilon\otimes \mathrm{id}: \mathrm{Com}\left( \pi_{\widetilde{H}_1}, \mathcal{W}^{k_1} (\mathfrak{g}_1, f_\mathrm{sub}) \otimes V_\Z \right) \hookrightarrow V_{\widetilde{\Z}} \otimes \pi_B,
\end{align*}
whose image for generic $k_1$ coincides with
\begin{align*}
&\Ker \int :\mathrm{e}^{\int (B_0+\widetilde{\phi})(z)}:\ dz\ \cap\ \bigcap_{i=1}^{n-1} \Ker\int :\mathrm{e}^{-\frac{1}{k_1+h^\vee_1} \int B_i(z)}:dz\\
&\cap\ \Ker\int :\mathrm{e}^{-\frac{1}{r(k_1+h^\vee_1)} \int B_n(z)}:dz\\
&=\Ker \int :\mathrm{e}^{\int (B_0+\widetilde{\phi})(z)}:\ dz\ \cap\ \bigcap_{i=1}^n \Ker\int :\mathrm{e}^{\int B_i(z)}:dz
\end{align*}
by \eqref{eq:sub-W-image} and the Feigin-Frenkel duality for the Virasoro vertex algebras. Now (2) follows from the above equation with \eqref{eq:super-W-image} and the isomorphism
\begin{align*}
&V_{\widetilde{\Z}} \otimes \pi_B \xrightarrow{\sim} V_\Z \otimes \pi^{k_2+h^\vee_2}_{\mathfrak{h}_2}\\
&:\mathrm{e}^{\int \pm\widetilde{\phi}(z)}:\ \mapsto\ :\mathrm{e}^{\int \pm\phi(z)}:,\quad
B_i(z) \mapsto -\frac{1}{k_2+h^\vee_2}\beta_i(z),\quad
i= 0,\ldots, n.
\end{align*}
\end{proof}
\end{theorem}

\section{Applications}\label{sec:Heisenbergcoset}

Here we discuss some important consequences of the theory of Heisenberg cosets \cite{CKLR}, see \cite{L2} for earlier literature.

\subsection{Simplicity of Heisenberg cosets}

We begin with simple observations for vertex superalgebras. Let $B$ be a simple vertex superalgebra and $\{B_\lambda\}_{\lambda\in I}$ a set of countably many inequivalent simple $B$-modules which contains $B$ as $B_0$. Let $C$ be a (not necessarily simple) vertex superalgebra and $\{C_\lambda\}_{\lambda\in I}$ a set of $C$-modules, which contains $C$ as $C_0$, with the same index set $I$. Assume that the $B\otimes C$-module 
\begin{align*}
A:=\bigoplus_{\lambda\in I}B_\lambda\otimes C_\lambda
\end{align*}
has the structure of a vertex superalgebra extending the $B\otimes C$-module structure.
Assume the following conditions:
\begin{enumerate}
\item
For each $\lambda\in I$, $B_\lambda$ and $C_\lambda$ are of countable dimension.	
\item
The $B$-modules $\{B_\lambda\}_{\lambda\in I}$ are objects in a semi-simple abelian full subcategory $\mathcal{B}$ of the category of weak $B$-modules. $\mathcal{B}$ is closed under direct sum over any countable index set.
\end{enumerate}
Then we have the following:
\begin{proposition}[{\cite[Theorem 8.1, Remark 8.3]{CL3}}]\label{decomposition under quotient}
Any quotient $A_s$ of the vertex superalgebra $A$ is an object of $\mathcal{B}$. The coset vertex superalgebra $D:=\mathrm{Com}(B,A_s)$ is a quotient of $C$. For each $\lambda\in I$, there exists some quotient $D_\lambda$ of the $C$-module $C_\lambda$, which admits a $D$-module structure such that 
\begin{align*}
A_s\simeq \bigoplus_{\lambda\in I}B_\lambda\otimes D_\lambda
\end{align*}
as $B\otimes D$-modules.
\end{proposition} 
\begin{proof} The assertions are obvious, but we include the proof for the completeness of the paper. Since $A$ is of countable dimension, $A$ is an object of $\mathcal{B}$ by the assumptions on $\mathcal{B}$. Since $\mathcal{B}$ is abelian, any quotient of $A$ as a vertex superalgebra is an object of $\mathcal{B}$. Since $\mathcal{B}$ is semi-simple, we have
\begin{align*}
A_s\simeq \bigoplus_{\lambda\in I} B_\lambda \otimes D_\lambda, \quad D_\lambda:=\mathrm{Hom}_{B}(B_\lambda, A^s)
\end{align*}
as $B$-modules. Then we have a natural surjection $C_\lambda\twoheadrightarrow D_\lambda$ for each $\lambda\in I$ as vector spaces. Since $D = \mathrm{Com}(B,A_s)$, it is a vertex superalgebra obtained as a quotient of $C$. Since $A_s$ is a $D$-module, the same is true for $D_\lambda=\mathrm{Hom}_{B}(B_\lambda, A_s)$.
\end{proof}
Next, we consider a criterion for the simplicity of the coset vertex superalgebra of a pair of simple vertex superalgebras. 
\begin{lemma}\label{vacuum-like vectors}
Let $V$ be a vertex superalgebra and $M$ a $V$-module. Let $m\in M$ be a vacuum-like vector, i.e., an element satisfying that $a_{(n)}m=0$ for all $a\in V$, $n\geq0$. Then the $\C$-linear map
$$F_m:\ V\rightarrow M,\quad a\mapsto a_{(-1)}m$$
is a $V$-module homomorphism.
\end{lemma}
\proof
We need to show $(a_{(n)}b)_{(-1)}m=a_{(n)}(b_{(-1)}m)$ for all $a,b\in V$ and $n\in\Z$, which are special cases of \cite[Proposition 4.5.6]{LL}.
\endproof
Now, let $W\subset V$ be simple vertex superalgebras of countable dimension. Suppose that $V$ is semi-simple as a $W$-module:
$$V=\bigoplus_{\lambda\in \Lambda} \widehat{W}_\lambda,$$
where $\Lambda$ is an index set that labels the inequivalent simple $W$-modules $W_\lambda$ that appear as submodules of $V$, and $\widehat{W}_\lambda$ is the $W$-submodule spanned by all the simple $W$-submodules isomorphic to $W_\lambda$. We assume $0\in \Lambda$ and $W_0=W$. 
\begin{lemma}
The subspace $\widehat{W}_0$ is isomorphic to $W\otimes \mathrm{Com}(W,V)$ as a $W$-module. In particular, $\widehat{W}_0$ has the structure of a vertex superalgebra.
\end{lemma}
\proof
Since $\widehat{W}_0$ is semi-simple as a $W$-module, we have simple submodules $\{M_\alpha\}_{\alpha\in I}$ such that 
$\widehat{W}_0=\bigoplus_{\alpha\in I}M_\alpha$. Note that $M_\alpha$ is isomorphic to $W$ as a $W$-module. Since $W$ is simple and of countable dimension over $\C$, we have
$\mathrm{End}_W(W)=\C$ by Schur's lemma. Thus $\mathrm{Com}(W,W)=\C |0\rangle$ and so the space of vacuum-like vectors in $M_\alpha$ is one dimensional.
Note that the vacuum-like vectors in $V$ with respect to $W$ is nothing but $\mathrm{Com}(W,V)$. Therefore, the linear map 
$$W\otimes \mathrm{Com}(W,V)\rightarrow \widehat{W}_0,\quad w\otimes u\mapsto w_{(-1)}u$$
is an isomorphism of $W$-modules by Lemma \ref{vacuum-like vectors}. 
\endproof
Now we have a criterion for the simplicity of $\mathrm{Com}(W,V)$. 
\begin{proposition}\label{simplicity}
Let $V$ be a simple vertex superalgebra of countable dimension and $W$ a simple vertex (super)subalgebra. Suppose that $V$ is semi-simple as a $W$-module and $W\otimes \mathrm{Com}(W,V)$ has a conformal vector $\omega$ which is also a conformal vector of $V$. Then $\mathrm{Com}(W,V)$ is also simple.
\end{proposition}
\proof
Take a non-zero ideal $\mathcal{I}\subset \mathrm{Com}(W,V)$ and let $\widehat{\mathcal{I}}\subset V$ denote the ideal generated by $\mathcal{I}$ in $V$. We show 
\begin{align}
\widehat{\mathcal{I}}\cap (W\otimes \mathrm{Com}(W,V))=W\otimes \mathcal{I}.
\end{align}
Indeed, by \cite[Proposition 4.5.6]{LL}, we have
\begin{align}
\widehat{\mathcal{I}}=\mathrm{Span}_{\C} \{a_{(n)}u\mid a\in V,\ u\in \mathcal{I},\ n\in\Z\}.
\end{align}
Then by using the conformal field $Y(\omega,z)=\sum_{n\in \Z} L_nz^{-n-2}$ and the skew-symmetry:
$$Y(a,z)u=-(-1)^{\bar{a}\bar{u}}\mathrm{e}^{z L_{-1}}Y(u,-z)a,$$
we have
\begin{equation*}
\widehat{\mathcal{I}}
=\mathrm{Span}_{\C} \{u_{(n)}a\mid u\in \mathcal{I},\ a\in V,\ n\in\Z\}
=\sum_{\lambda\in \Lambda}\sum_{n\in\Z}\mathcal{I}_{(n)}\widehat{W}_{\lambda}.
\end{equation*}
Since vertex operators $u_{(n)}$ commute with $w_{(m)}$ for all $u \in \mathrm{Com}(W,V)$ and $w \in W$, the actions $u_{(n)}$ on $W_\lambda$ are in $\mathrm{Hom}_W(W_\lambda,V)$. By Schur's lemma, we have $\mathrm{Hom}_W(W_\lambda,W_\mu)=0$ if $\lambda\neq \mu$. Thus, 
$$\mathcal{I}_{(n)}\widehat{W}_{\lambda}\subset (W\otimes \mathrm{Com}(W,V))_{(n)}\widehat{W}_\lambda\subset \widehat{W}_\lambda$$
so that we have
\begin{align*}
\widehat{\mathcal{I}}\cap (W\otimes \mathrm{Com}(W,V))
&=\mathrm{Span}_{\C} \{u_{(n)}a\mid u\in \mathcal{I},\ a\in W\otimes \mathrm{Com}(W,V),\ n\in\Z\}\\
&=W\otimes \mathcal{I}.
\end{align*}
Since $V$ is simple, $\widehat{\mathcal{I}}=V$ and thus $\mathcal{I}=\mathrm{Com}(W,V)$. This shows that $\mathrm{Com}(W,V)$ is simple.
\endproof
Consider the case when $V$ is a simple vertex operator superalgebra and $W$ is a Heisenberg vertex subalgebra generated by primary fields of conformal degree 1. Let $\mathcal{H}\subset W$ be a subspace of primary fields of conformal degree 1 with OPEs
\begin{equation}\label{eq:bilinear}
h_1(z)h_2(w)\sim \frac{\ell(h_1,h_2)}{(z-w)^2},\quad  h_1,h_2\in \mathcal{H}
\end{equation}
for some symmetric bilinear form $\ell: \mathcal{H}\times \mathcal{H}\rightarrow \C$, generating $W$. Suppose that $\ell$ is non-degenerate (i.e. $W$ is simple) and the conditions (1)(2) in Corollary \ref{Heisenberg coset}. Then the assumption in Proposition \ref{simplicity} for the conformal vector is automatically satisfied, and we have
$$\widehat{W}_0=\bigcap_{h\in \mathcal{H}}\mathrm{Ker}(h_{(0)}).$$
The conformal vector $\omega$ lies in this subspace since $h_{(0)}\omega=0$ for $h\in \mathcal{H}$. These imply Proposition 3.2 and Theorem 2.9 in \cite{CKLR}:
\begin{corollary}[{\cite[Proposition 3.2]{CKLR}}]\label{Heisenberg coset}
Let $V$ be a vertex operator superalgebra of countable dimension and $\pi$ a non-degenerate Heisenberg vertex subalgebra generated by a subspace $\mathcal{H}$ spanned by primary fields of conformal degree 1. Suppose the following conditions:
\begin{itemize}
\item[(1)]
The action $\mathcal{H}\times V\rightarrow V$, $((h,a)\mapsto h_{(0)}a)$, is semi-simple,
\item[(2)]
Each $\pi$-submodule spanned by all $a\in V$ satisfying that $h_{(0)}a=\lambda(h)a$ for $h \in \mathcal{H}$ with fixed $\lambda\in \mathcal{H}^*$ is bounded from below by the conformal degree.
\end{itemize}
Then a simple quotient $V\twoheadrightarrow V_s$ induces a simple quotient $\mathrm{Com}(\pi,V)\twoheadrightarrow \mathrm{Com}(\pi,V_s)$. 
\end{corollary}
\proof
By the assumption, $V$ is semi-simple as a $\pi$-module, cf \cite[Theorem 1.7.3]{FLM}. Thus we have
\begin{align}\label{Heisenberg decomposition}
V\simeq \bigoplus_{\lambda\in \mathcal{H}^*}\pi_\lambda\otimes C_\lambda,\quad C_\lambda:=\left\{a\in V\mid h_{(n)}a=\delta_{n,0}\lambda(h)a,\ \forall h\in \mathcal{H}, n\geq0\right\}
\end{align}
as $\pi\otimes \mathrm{Com}(\pi, V)$-modules, where $\pi_\lambda$ is the highest weight $\pi$-module with the highest weight $\lambda$. Then Proposition \ref{decomposition under quotient} applies and we obtain a quotient 
$$\mathrm{Com}(\pi,V)\twoheadrightarrow \mathrm{Com}(\pi,V_s),$$
induced by $V\rightarrow V_s$. By the argument just after Proposition \ref{simplicity}, Proposition \ref{simplicity} applies to conclude that $\mathrm{Com}(\pi,V_s)$ is simple.
\endproof

\begin{corollary}[{\cite[Theorem 2.9]{CKLR}}]
Suppose that $V$ is simple in addition to the assumptions in Corollary \ref{Heisenberg coset}. Then each $C_\lambda$ appearing in the decomposition \eqref{Heisenberg decomposition} is a simple $\mathrm{Com}(\pi,V)$-module. \end{corollary}
\begin{proof}
First, observe that operator product algebra of $V$ respects Heisenberg weight, since every vertex operator of $V$ is especially an intertwiner for the Heisenberg subalgebra. This means that $v_{(n)}w \in \widehat W_{\lambda+\mu} =\pi_{\lambda+\mu}\otimes C_{\lambda+\mu}$ for $v \in  \widehat W_\lambda = \pi_\lambda\otimes C_\lambda$ and $w \in \widehat W_\mu = \pi_\mu\otimes C_\mu$. Now, fix a $\widehat W_\nu$ appearing in the decomposition of $V$ and let $ \pi_\nu \otimes \mathcal I \subset \widehat W_\nu$ be a $\pi \otimes \mathrm{Com}(\pi,V)$-submodule of $\widehat W_\nu$. Denote by $\widehat{\mathcal I}$ the ideal generated by $\pi_\nu \otimes \mathcal I$ in $V$. Then by the same argument and same notation as in the proof of Proposition \ref{simplicity}, we have
\begin{equation*}
\widehat{\mathcal{I}}
=\mathrm{Span}_{\C} \{u_{(n)}a\mid u\in \pi_\nu \otimes \mathcal{I},\ a\in V,\ n\in\Z\}
=\sum_{\lambda\in \mathcal H^*}\sum_{n\in\Z} (\pi_\nu \otimes \mathcal{I})_{(n)}\widehat{W}_{\lambda}.
\end{equation*}
But we just explained that 
$$(\pi_\nu \otimes\mathcal{I})_{(n)}\widehat{W}_{\lambda}\subset (\widehat W_\nu)_{(n)}\widehat{W}_\lambda\subset \widehat{W}_{\nu+\lambda}.$$
Therefore,
\begin{align*}
\widehat{\mathcal{I}}\cap \widehat W_{\lambda}
&=\mathrm{Span}_{\C} \{u_{(n)}a\mid u\in \pi_\nu \otimes \mathcal{I},\ a\in \pi \otimes C,\ n\in\Z\}
=  \pi_\nu \otimes \mathcal{I}.
\end{align*}
Since $V$ is simple, we have $\widehat{ \mathcal I} =V$ and thus $\mathcal I = C_\lambda$, that is, $C_\lambda$ is simple as a $\mathrm{Com}(\pi,V)$-module. 
\end{proof}

\subsection{$C_2$-cofiniteness of Kazama-Suzuki cosets}\label{sec:C2}

The aim of this subsection is to study if the $C_2$-cofiniteness is inherited via a Kazama-Suzuki type coset construction. The proof follows essentially from results in \cite{CKLR}. The main issue is to understand whether double commutants of Heisenberg cosets are lattice vertex operator algebras. 

Recall the decomposition \eqref{Heisenberg decomposition} as a $\pi\otimes \mathrm{Com}(\pi,V)$-module. We remark that the set $L = \{ \lambda \in \mathcal H^* | C_\lambda \neq  0 \}$ is a lattice if $V$ is simple, see \cite[Proposition 3.6]{LX} (statements in \cite{LX} are stated for vertex operator algebras, but exactly the same proofs apply for vertex operator superalgebras).

\begin{assumption}\label{ass:real}
Under the set-up of Corollary \ref{Heisenberg coset} for a simple vertex operator superalgebra $V$, let $L = \{ \lambda \in \mathcal H^* | C_\lambda \neq  0 \}$ be the lattice of weights appearing in $V$. Then we assume the following conditions for a basis $B= \{ h_1, \dots, h_n\}$ of $\mathcal{H}$:
\begin{enumerate}
\item The bilinear form $\ell$ in \eqref{eq:bilinear} satisfies $\ell(h_i, h_j)\in\mathbb R$ for all $i,j$.
\item $\lambda(h_i)\in\mathbb R$ for all $i$ and $\lambda\in L$.
\item $L \otimes_{\mathbb Z}\mathbb C = \mathcal H^*$. 
\end{enumerate}
\end{assumption}
This reality assumption insures the existence of vertex tensor category of highest weight $\pi$-modules. Invoking the proof of \cite[Theorem 3.5]{CKLR}, we obtain the following theorem:
\begin{lemma}\label{lemma:C2}
Let $V$ be a simple vertex operator superalgebra of countable dimension with a rank $n$ Heisenberg subalgebra $\pi$ acting semi-simply on $V$ and denote by $C=\mathrm{Com}(\pi, V)$ the coset. If $C$ is a $C_2$-cofinite vertex operator algebra of CFT-type and  Assumption \ref{ass:real} holds, then $\mathrm{Com}(C, V)$ is a rank $n$ lattice vertex operator superalgebra $V_N$. Moreover, if the lattice $N$ is positive definite, $V$ is $C_2$-cofinite. 
\end{lemma}
\begin{proof}
Since $C$ is a $C_2$-cofinite vertex operator algebra of CFT-type, \cite[Proposition 4.1, Theorem 4.13]{Hu3} insure the assumption just after \cite[Proposition 3.2]{CKLR} for the existence of the vertex tensor category of (generalized) $C$-modules. By \cite[Theorem 2.3]{CKLR}, we also have a vertex tensor category generated by $\{\pi_\lambda \}_{\lambda\in L}$. By \cite[Theorem 5.5]{CKM2}, the vertex tensor category of the tensor product $\pi \otimes C$ is the Deligne product of these two categories. Thus the same arguments in Section 3 in \cite{CKLR} for the proof of \cite[Theorem 3.5]{CKLR} may apply. This guarantees that
\begin{align*}
N = \{ \lambda \in L \mid C_\lambda \simeq C\ \text{as} \ C\text{-modules} \}
\end{align*}
is a sublattice of $L$ and $C_\mu \simeq C_{\mu+\lambda}$ for all $\lambda\in N$ and $\mu \in L$. By setting $C_{[\mu]}:= C_\mu$ for $[\mu]\in L/N$, we have the following decomposition  \cite[(3.7)]{CKLR}:
\begin{equation}\label{eq:extensionV}
V \simeq \bigoplus_{[\mu] \in L/N}  V_{[\mu]} \otimes C_{[\mu]}, \qquad V_{[\mu]} = \bigoplus_{\lambda \in N} \pi_{\mu+\lambda}.
\end{equation}
Especially, $\mathrm{Com}(C, V)=V_{[0]}=V_N$ is the lattice vertex operator superalgebra of the lattice $N$. Again by the same arguments for the proof of \cite[Theorem 3.5]{CKLR}, the $C$-modules $C_{[\mu]}$, ($[\mu]\in L/N$), are mutually inequivalent and simple. Since $C$ is $C_2$-cofinite, $C$ has only finitely many inequivalent simple ordinary modules by \cite[Theorem 3.13]{L4}. Thus $L/N$ is a finite set. Therefore $N$ is a full rank sublattice of $L$.

Now we assume that $N$ is positive definite so that $V_N$ is especially $C_2$-cofinite and of CFT-type. Then $V$ is an extension \eqref{eq:extensionV} of the $C_2$-cofinite vertex operator superalgebra $C \otimes V_N$ of CFT-type by finitely many simple $C\otimes V_N$-modules. By \cite[Proposition 5.2]{ABD}, each $V_{[\mu]} \otimes C_{[\mu]}$ is already $C_2$-cofinite as a simple $C\otimes V_N$-module and thus $V$ itself is $C_2$-cofinite.
\end{proof}
To apply this lemma to our setting, let us introduce a few Heisenberg vertex operator algebras. 
Let $\mathfrak{c}=\C\alpha\oplus \C\beta$ be a $\C$-vector space equipped with a symmetric bilinear form $(\text{-}|\text{-}):\mathfrak{c} \otimes \mathfrak{c} \rightarrow \C$ determined by $(\alpha|\alpha)=q\in \mathbb{Q}_{>0}$, $(\beta|\beta)=r\in \mathbb{Z}\setminus\{0\}$ and $(\alpha|\beta)=0$. We assume $q+r\neq0$.
Then the elements $\gamma=(\alpha+\beta)/(q+r)$ and $\mu=(r \alpha-q\beta)/(q+r)$ span $\mathfrak{c}$ and satisfy 
$$(\gamma|\gamma)=\frac{1}{q+r},\quad (\mu|\mu)=\frac{q r}{q+r},\quad (\gamma|\mu)=0.$$
For $\eta\in\{\alpha, \beta, \gamma,\mu\}$, let $\pi_\eta$ denote the  Heisenberg vertex subalgebra of $\pi_\mathfrak{c}$ generated by $\eta$, and $\pi_{\eta,n}$ the $\pi_\eta$-module $\pi_{\eta,n\eta}$ for simplicity.
Then we have
\begin{align}\label{change of basis}
\pi_{\alpha,n}\otimes \pi_{\beta,m}\simeq \pi_{\gamma,qn+rm}\otimes \pi_{\mu,n-m}
\end{align}
as $\pi_{\gamma}\otimes \pi_{\mu}$-modules.
Now, let $W$ be a simple vertex operator superalgebra containing $\pi_\alpha$ such that
\[
W \simeq  \bigoplus_{n \in \mathbb Z} \pi_{\alpha,n} \otimes D_n, 
\] 
as $\pi_\alpha \otimes D$-modules. Here $D=D_0= \mathrm{Com}(\pi_\alpha, W)$ and $D_n$ is a $D$-module. Let $V_{\Z\beta}$ denote the lattice vertex operator superalgebra associated with the lattice $\Z \beta$ and consider $W \otimes V_{\mathbb Z\beta }$ and the coset
\begin{align*}
C_0 = \mathrm{Com}(\pi_\gamma, W \otimes V_{\mathbb Z \beta}),
\end{align*}
where we consider an embedding $\pi_\gamma \hookrightarrow \pi_\alpha \otimes \pi_\beta \subset W \otimes V_{\mathbb Z \beta}$. Then
\begin{equation}\nonumber
\begin{split}
D &= \mathrm{Com}(\pi_\alpha \otimes \pi_\beta, W \otimes V_{\mathbb Z \beta })=  \mathrm{Com}(\pi_\mu \otimes \pi_\gamma, W \otimes V_{\mathbb Z \beta })\\ &=    
  \mathrm{Com}(\pi_\mu,  \mathrm{Com}(\pi_\gamma, W \otimes V_{\mathbb Z \beta }))  =  \mathrm{Com}(\pi_\mu, C_0).
\end{split}
\end{equation}
If $C_0$ is a $C_2$-cofinite vertex operator algebra of CFT-type, the assumptions of Lemma \ref{lemma:C2} for $V = W \otimes V_{\mathbb Z \beta }$ and $C = C_0$ are satisfied. Thus, we have
\begin{align}\label{eq:WV-dec1}
W \otimes V_{\Z \beta}\simeq \bigoplus_{[\lambda] \in L/N} V_{[\lambda]} \otimes C_{[\lambda]}
\end{align}
with $L=\Z\,q\gamma + \Z\,r\gamma$ and a certain sublattice $N$ of $L$. By setting $q=q_2/q_1$ with $q_1,\in \Z_{>0}$, $q_2\in \Z\setminus \{0\}$ such that $(q_1,q_2)=1$, we have $\Z\,q + \Z\,r=\Z\,(g/q_1) $ with $g = (q_2, r)$. Thus $L=\Z\,(g/q_1)\gamma$ and there exists $a\in \Z\setminus\{0\}$ such that $N=\Z\,(a g/q_1)\gamma$. 
\begin{lemma}\label{lem:C2lat}
Let $C_0$, $D$ be vertex algebras given in the above and suppose that $C_0$ is a $C_2$-cofinite vertex operator algebra of CFT-type. Then $\mathrm{Com}(D, C_0)$ is a lattice vertex algebra of some positive-definite even lattice with rank one.
\end{lemma}
\begin{proof}
Since $C_0$ is a vertex operator algebra and $D = \mathrm{Com}(\pi_\mu, C_0)$, the double commutant $\operatorname{Com}(D, C_0)$ is either $\pi_\mu$ or a larger vertex operator algebra that extends $\pi_\mu$. We have
\[
W \otimes V_{\mathbb Z\beta } \simeq  \bigoplus_{n, m \in \mathbb Z} \pi_{\alpha,n} \otimes \pi_{\beta,m} \otimes D_n 
\simeq  \bigoplus_{n, m \in \mathbb Z} \pi_{\gamma,qn+rm} \otimes \pi_{\mu,n-m} \otimes D_n.
\] 
On the other hand, we have the decomposition \eqref{eq:WV-dec1} with
\[
\operatorname{Com}(C_0, W \otimes V_{\mathbb Z\beta }) = V_N = \bigoplus_{s \in \Z}\pi_{\gamma, (a g/q_1)s}.
\]
For any $s \in \Z$, $\pi_{\gamma,qn+rm} \simeq \pi_{\gamma, (a g/q_1)s}$ if and only if $qn+rm = (a g/q_1)s$. Thus, it follows that
\[
C_0 \simeq \bigoplus_{\substack{n, m \in \mathbb Z\\ qn + rm = (a g/q_1)s}}  \pi_{\mu, n-m} \otimes D_n \simeq \bigoplus_{\substack{n, m \in \mathbb Z\\ qn + rm =0}}  \pi_{\mu, n-m} \otimes D_{n}
\]
and hence $D_n \simeq D_{n+a_0s}$, ($s\in \mathbb Z$) for some $a_0\in\mathbb{Z}\backslash \{0\}$. This implies that 
\begin{align*}
\operatorname{Com}(D, C_0)
\simeq \bigoplus_{\substack{n \in  \Z a_0, m \in \mathbb Z\\ qn + rm =0}}  \pi_{\mu, n-m} 
= \bigoplus_{\substack{s, m \in \mathbb Z\\ a_0qs + rm =0}}  \pi_{\mu, a_0s-m}   
\simeq \bigoplus_{s \in \mathbb Z}  \pi_{\mu,bs}
\end{align*}
for some $b\in \mathbb{Q}\backslash \{0\}$. This must be an even lattice vertex operator algebra, since it is a vertex algebra extension of a Heisenberg algebra. From the assumption, we have $(\mu|\mu)\in \mathbb{Q}\backslash\{0\}$. If $(\mu|\mu)<0$, then the lattice has negative signature and thus conformal weights of the lattice vertex operator algebra are not bounded from below. This is impossible for the coset of a vertex operator algebra of CFT-type. Hence, $(\mu|\mu)$ must be positive.
\end{proof}
\begin{corollary} \label{cor:C2}
With the same set-up as in Lemma \ref{lem:C2lat}, $D$ is $C_2$-cofinite.
\begin{proof}
The assertion is immediate from \cite[Corollary 2]{Mi}.
\end{proof}
\end{corollary}

\subsection{Rationality of Heisenberg cosets}\label{sec:rational}

In general it is an open and difficult problem to prove that the coset of a rational and $C_2$-cofinite vertex operator algebra is itself rational and $C_2$-cofinite. If one is considering the coset by a Heisenberg vertex algebra, then this problem can be solved. This follows from  two Theorems. Firstly a result of Geoffrey Mason
that ensures that the Heisenberg subalgebras inside rational and $C_2$-cofinite vertex operator algebras extend to lattice vertex algebras \cite{Mason}. Secondly,  Masahiko Miyamoto proved a $C_2$-cofiniteness result for finite abelian orbifolds \cite{Mi}. These two results are sufficient to ensure a rationality and $C_2$-cofiniteness result of Heisenberg cosets \cite[Section 4.3]{CKLR} (an alternative argument is \cite{CaM}). We will now explain these results and then generalize them to Heisenberg cosets of vertex operator superalgebras. 

Heisenberg cosets of strongly rational vertex operator algebras are strongly rational in the following sense:
\begin{theorem}\label{thm:rationality} Let $V$ be a simple, rational, $C_2$-cofinite vertex operator algebra of CFT-type. Suppose that $U \subset H \subset V$ where $H$ is a Cartan subalgebra of $V$ and $U$ is a nondegenerate subspace, so that $\pi_U$ is a simple Heisenberg vertex operator algebra of rank the dimension of $U$. Set $C = \mathrm{Com}(\pi_U, V)$. Then
\begin{enumerate}
\item $C$ is simple and of CFT-type.
\item $C$ is $C_2$-cofinite.
\item Every grading-restricted generalized $C$-module is completely reducible.
\end{enumerate}
\begin{proof}
First, we show (1). Let $\omega_U, \omega_V$ be the conformal vectors of $\pi_U, V$ respectively. Then $\omega_V-\omega_U$ is the conformal vector of $C$. Since $(\omega_V)_{(n)} = (\omega_V-\omega_U)_{(n)}$ on $C$ for $n \in \Z_{\geq 0}$, each homogeneous subspace of $C$ with conformal degree $m$ is in that of $V$ with the same degree. This proves that $C$ is of CFT-type. The simplicity of $C$ is immediate from Corollary \ref{Heisenberg coset}. Then (1) is proven.

Next, we show (2) and (3). By \cite[Theorem 1]{Mason}, there is a unique maximal vertex subalgebra $W \subset V$ with the conformal vector $\omega_U$, which is indeed isomorphic to a lattice vertex operator algebra $V_\Lambda$ for a positive-definite even lattice $\Lambda \subset U$ with $\dim U = \operatorname{rank} \Lambda$. This implies that $\operatorname{Com}(C, V) = V_\Lambda$. Then (2), (3) follow from \cite[Corollary 2]{Mi}, \cite[Theorem 4.12]{CKLR} respectively. This completes the proof.
\end{proof} 
\end{theorem}
\begin{corollary}\label{remark:super}
Let $V$ be a simple vertex superalgebra of CFT-type such that the even subalgebra $V_0$ is simple, rational, $C_2$-cofinite and of CFT-type. Then Theorem \ref{thm:rationality} also holds for $V$.
\begin{proof}
First, (1) is clear by using the same proof of (1) in Theorem \ref{thm:rationality}. We need to show (2) and (3). Let $C = C_0 \oplus C_1$ be the decomposition of $C$ into the even and odd parts. Then by Theorem \ref{thm:rationality}, $C_0 = \operatorname{Com}(\pi_U, V_0)$ is simple, $C_2$-cofinite and of CFT-type. Consider $\Z_2$-action on $C$ with respect to the decomposition $C = C_0 \oplus C_1$. Then $C_0 = C^{\Z_2}$. By \cite[Theorem 3]{DM} (where the assertion is proved for simple vertex operator algebras, but the same proof applies for simple vertex operator superalgebras), it follows that $C_1$ is a simple $C_0$-module. Then $C$ is already $C_2$-cofininte as a $C_0$-module by \cite[Proposition 5.2]{ABD}. This proves (2). Finally, we show (3), but this follows from Theorem \ref{thm:super-semisimple}.
\end{proof}
\end{corollary}

To state Theorem \ref{thm:super-semisimple}, let us introduce several categories consisting of $V_0$-modules. Let $V_0$ be a vertex algebra and $\mathcal{C}_0$ be the vertex tensor category of $V_0$-module, that is, assume that $\mathcal{C}_0$ exists. Let $V$ be a simple vertex superalgebra with $V= V_0 \oplus V_1$ the decomposition into even and odd part. Suppose that $\mathcal C_0$ is semisimple and $V$ is an object in $\mathcal{C}_0$. The vertex superalgebra $V$ in Corollary \ref{remark:super} satisfies these assumptions: The category of $V_0$-module forms a vertex tensor category by \cite{Hu1, Hu2} and is in fact a modular tensor category $\mathcal{C}_0$. The proof of \cite[Theorem 3.1]{CKLR} also applies to the vertex superalgebra $V$ and so $V_1$ is a simple current for $V_0$.

By \cite{CKL}, $V$ is a commutative associative superalgebra object in $\mathcal{C}_0$. This algebra is an order two simple current extensions and these are always seperable \cite[Example 2.8]{DMNO} (special Frobenius algebras in the language of \cite{FRS}). Let $\mathcal{C}_V$ be the supercategory of modules for the commutative associative superalgebra object $V$ and $\mathcal C^\mathrm{loc}_V$ be the full subcategory of $\mathcal{C}_V$ consisting of local modules for $V$. Then objects in $\mathcal{C}_V$ consist of pairs $(M, \mu_M)$ of objects $M = M_0 \oplus M_1$ with $M_0, M_1 \in \mathcal{C}_0$ and even morphisms $\mu_M \colon V \boxtimes M \rightarrow M$ satisfying some conditions. An object $(M, \mu_M)$ in $\mathcal{C}_V$ is called local if $\mu_W \circ \mathcal{R}_{M, V} \circ \mathcal{R}_{V, M} = \mu_M$ for the braiding isomorphism $\mathcal{R}_{N, N'} \colon N \boxtimes N' \rightarrow N' \boxtimes N$. See \cite{CKM} for the definitions and properties of $\mathcal{C}_V$ and $\mathcal{C}_V^\mathrm{loc}$. By \cite[Theorem 1.3]{CKM}, $\mathcal C^\mathrm{loc}_V$ is equivalent to the supercategory of $V$-modules in $\mathcal{C}_0$ as braided tensor supercategories, and there is an induction functor $\mathcal{F}$ whose right-adjoint is the restriction functor $\mathcal{G}$:
\begin{equation*}
\begin{split}
&\mathcal{F} \colon \mathcal{C}_0 \rightarrow \mathcal{C}_V,\quad
M \mapsto (V \boxtimes M, \mu_{V \boxtimes M});\\
&\mathcal{G} \colon \mathcal{C}_V \rightarrow \mathcal{C}_0, \quad
(M, \mu_M) \mapsto M.
\end{split}
\end{equation*}
See \cite{CKM} for details on these functors.

\begin{theorem}\label{thm:super-semisimple} 
Let $V_0$ be a simple vertex algebra, $V$ be a simple vertex superalgebra with $V= V_0 \oplus V_1$ the decomposition into even and odd parts and $\mathcal{C}_0$ be a vertex tensor category of modules of the vertex algebra $V_0$. Suppose that $V$ is an object in $\mathcal{C}_0$, and $\mathcal{C}_0$ is semisimple and rigid. Then the category $\mathcal C^\mathrm{loc}_V$ of local $V$-modules that lie in $\mathcal C_0$ is semisimple as well.
\end{theorem}
\begin{proof}
Recall that the objects and the Hom spaces in $\mathcal{C}_V$ has $\Z_2$-grading, which gives the respective super-structure. Denote by $\underline{\mathcal{C}}_V$ the category obtained from $\mathcal{C}_V$ by forgetting the above $\Z_2$-grading and by $\mathcal{H} = \mathcal{H}_{\mathcal{C}_V} \colon \mathcal{C}_V \rightarrow \underline{\mathcal{C}}_V$ the forgetful functor.
Note that $\mathcal{H}$ restricts to a functor on the subcategories of local modules $\mathcal{H} = \mathcal{H}_{\mathcal{C}_V^{\mathrm{loc}}} \colon \mathcal{C}_V^{\mathrm{loc}} \rightarrow \underline{\mathcal{C}}_V^{\mathrm{loc}}$. Since $V$ is a simple vertex superalgebra that lies in a semisimple vertex tensor category $\mathcal{C}_0$ of the simple vertex algebra $V_0$, the proof of \cite[Theorem 3.1]{CKLR} applies to $V$ and so $V_1$ is a simple current for $V_0$. Then $V$ is separable by \cite[Example 2.8]{DMNO}. Then the category $\underline{\mathcal{C}}_V$ of the separable algebra object $V$ in the fusion category $\mathcal{C}_0$ must be semisimple by \cite[Proposition 2.7]{DMNO}. Moreover, by \cite[Proposition 4.5, Remark 2.64]{CKM}, every simple object in $\underline{\mathcal{C}}_V^{\rm{loc}}$ is of the form $\mathcal{H} \circ \mathcal{F}(X)$ for some simple $X$ in $\mathcal{C}_0$. 

Let $\underline{\mathcal{F}} = \mathcal{H} \circ \mathcal{F}$ and $\underline{\mathcal{G}} \colon \underline{\mathcal{C}}_V \rightarrow \mathcal{C}_0$ be the restriction functor. Then by definition, $\mathcal{G} = \underline{\mathcal{G}} \circ \mathcal{H}$. Let $(X, \mu_X)$ be any object in $\mathcal{C}^\mathrm{loc}_V$. Then 
\[
\mathcal{H}(X, \mu_X) \simeq \bigoplus_{i \in I} (Y^i, \mu_{Y^i})
\]
as an object in $\underline{\mathcal{C}}_V^{\mathrm{loc}}$ with simple objects $(Y^i, \mu_{Y^i})$ in $\underline{\mathcal{C}}_V$ for some finite index set $I$. Let $Y^i_0 \oplus Y^i_1$ be the decompositon of $Y^i$ into even and odd parts. By \cite[Corollary 4.22]{CKM} we then have that $Y^i_0$ is simple in $\mathcal{C}_0$ and $(Y^i, \mu_{Y^i}) \simeq \underline{\mathcal{F}}(Y_0^i)$ in $\mathcal{C}_V$. Thus,
\[
\mathcal{H}(X, \mu_X) \simeq \underline{\mathcal{F}}(X_0),\quad
X_0 = \bigoplus_{i \in I} Y^i_0 \in \mathcal{C}_0.
\]
Let $(Z, \mu_Z)$ be an arbitrary object in $\mathcal{C}^\mathrm{loc}_V$ that is simple in $\mathcal{C}_V$. Then by \cite[Proposition 4.5]{CKM}, $(Z, \mu_Z) \simeq \mathcal{F}(Z_0)$ for some simple $Z_0$ in $\mathcal{C}_0$. Since $\mathcal{G}$ is the right-adjoint of $\mathcal{F}$, we have
\begin{align*}
\Hom_{\mathcal{C}^\mathrm{loc}_V}\left( (Z, \mu_Z), (X, \mu_X) \right)
&\simeq \Hom_{\mathcal{C}_V}\left( \mathcal{F}(Z_0), (X, \mu_X) \right)
\simeq \Hom_{\mathcal{C}_0}\left( Z_0, \mathcal{G}(X, \mu_X) \right)\\
&\simeq \Hom_{\mathcal{C}_0}\left( Z_0, \mathcal{G} \circ \mathcal{F}(X_0) \right)
\simeq \Hom_{\mathcal{C}_V^\mathrm{loc}}\left( (Z, \mu_Z), \mathcal{F}(X_0) \right),
\end{align*}
where we use the equations
\[
\mathcal{G}(X, \mu_X)
= \underline{\mathcal{G}} \circ \mathcal{H}(X, \mu_X)
\simeq \underline{\mathcal{G}} \circ \underline{\mathcal{F}}(X_0)
= \underline{\mathcal{G}} \circ \mathcal{H} \circ \mathcal{F} (X_0)
= \mathcal{G} \circ \mathcal{F}(X_0).
\]
But since $\mathcal F(X_0)$ is completely reducible in $\mathcal{C}^\mathrm{loc}_V$, this can only be true for all objects $(Z, \mu_Z)$ in $\mathcal{C}^\mathrm{loc}_V$ that are simple in $\mathcal{C}_V$ if $\mathcal{F}(X_0)\hookrightarrow (X, \mu_X)$ in $\mathcal{C}^\mathrm{loc}_V$, . Since $\mathcal{G}(\mathcal{F}(X_0)) \cong \mathcal{G}(X, \mu_X)$, these two modules have the same graded dimensions and so the embedding must be an isomorphism, i.e. $\mathcal{F}(X_0)$ is isomorphic to $(X, \mu_X)$ in $\mathcal{C}^\mathrm{loc}_V$. 
\end{proof}

\subsection{Application to the main Theorem}
We strengthen Theorem \ref{Duality} and Theorem \ref{thm:KS} from generic levels to all the ``non-critical'' levels, descend them to the simple quotients, and obtain a rationality result.

Let $V$ be a finite dimensional vector space over $\C$. A family of vector subspaces $\{W^\alpha\}_{\alpha\in\C}$ is called continuous if they are of the same dimension $d\in\Z_{\geq0}$ and the induced map $\C\rightarrow \operatorname{Gr}(d,V)$ to the Grassmannian manifold is continuous \cite{T}. For a $\mathbb{Z}$-graded vector space $V=\bigoplus_{n\in\Z}V_n$ such that $\mathrm{dim}V_n<\infty$, ($n\in\mathbb{Z}$),  a family of graded vector subspaces $\{W^\alpha\}_{\alpha\in\C}$, ($W^\alpha=\bigoplus_{n\in\Z}W^\alpha_n$), is called continuous if the homogeneous subspaces $\{W^\alpha_n\}_{\alpha\in \C}$, ($n\in\Z$), are continuous families.
\begin{lemma}\label{continuity}
Let $V=\bigoplus_{n\in\mathbb{Z}}V_n$ be a $\mathbb{Z}$-graded vertex superalgebra with $\mathrm{dim} V_n<\infty$, ($n\in\mathbb{Z}$). Let $\{W^1_\alpha\}_{\alpha\in\C}$, $\{W^2_\alpha\}_{\alpha\in\C}$ be $\mathbb{Z}$-graded vertex (super)subalgebras which form continuous families as vector spaces. If $W^1_\alpha= W^2_\alpha$ on some open dense subset $U\subset \C$, then $W^1_\alpha= W^2_\alpha$ for all $\alpha\in \C$ as vertex superalgebras.
\end{lemma}
\proof
Since $W^i_\alpha$ is a vertex subalgebra of $V$, we have embeddings $\iota^i_\alpha \colon W^i_\alpha \hookrightarrow V$ for all $\alpha \in \C$ with $i = 1, 2$. Let $W^i_\alpha = \bigoplus_{n \in \Z}W^i_{\alpha, n}$ be the decomposition with respect to the $\Z$-grading. By the assumption that $\{ W^i_{\alpha, n}\}_{\alpha \in \C}$ forms a continuous family for all $n \in \Z$ with $i =1, 2$ and $W^1_\alpha= W^2_\alpha$ for $\alpha \in U$, it follows that $\dim W^1_{\alpha, n} = \dim W^2_{\alpha, n} =: d_n$ for all $n \in \Z$, and the maps
\begin{align*}
p_i \colon \C \ni \alpha \mapsto \iota^i_\alpha(W^i_{\alpha, n}) \in \operatorname{Gr}(d_n, V_n),\quad
i = 1,2,
\end{align*}
define continuous curves in $\operatorname{Gr}(d_n, V_n)$. Then, by the assumption that $U$ is open dense in $\C$, the curve $p_1$ densely coincides with the curve $p_2$, and thus is equal to $p_2$ by the continuity of $p_1$ and $p_2$. This implies that $\iota^1_\alpha(W^1_{\alpha, n}) = \iota^2_\alpha(W^2_{\alpha, n})$ for all $n\in\Z$ and $\alpha \in \C$. Summing up all these equations, we have $\iota^1_\alpha(W^1_{\alpha}) = \iota^2_\alpha(W^2_{\alpha}) =: W_\alpha$. Since $\iota^i_\alpha$ is the embedding of vertex (super)algebras into $V$ for all $\alpha\in\C$ and $i=1, 2$, $W_\alpha$ is a vertex subalgebra of $V$ for all $\alpha\in\C$. Therefore $W^1_\alpha = W^2_\alpha$ for all $\alpha\in \C$ as vertex (super)algebras. This completes the proof.
\endproof
\noindent
We note that it was used in \cite{AFO} to prove the Feigin-Frenkel duality for principal $\mathcal{W}$-algebras at all levels.

Set the rational numbers $x_i$, ($i=1,2$), by 
\begin{align}\label{eq: second critical level}
&(x_1,x_2)=
\begin{cases}
\left(-n+\frac{1}{n},-\frac{n^2}{n+1}\right),&\mathrm{if}\ (\mathfrak{g}_1,\mathfrak{g}_2)=(\mathfrak{sl}_{n+1}, \mathfrak{sl}(1|n+1)),\\
\left(-2n+2,-n+\frac{1}{2}\right),&\mathrm{if}\ (\mathfrak{g}_1,\mathfrak{g}_2)=(\mathfrak{so}_{2n+1}, \mathfrak{osp}(2|2n)),
\end{cases}
\end{align}
and, set $K_i:=\{-h_i^\vee\}$, $S_i:= \{-h_i^\vee, x_i\}$. Then we have the following.
\begin{corollary}  \label{cor:coset}
Let $(\mathfrak{g}_1,\mathfrak{g}_2)=(\mathfrak{sl}_{n+1}, \mathfrak{sl}(1|n+1))$, or $(\mathfrak{so}_{2n+1}, \mathfrak{osp}(2|2n))$, and $(k_1,k_2)$ satisfy \eqref{eq: ell}. Then the following isomorphisms hold:
\begin{enumerate}
\item For $k_1 \in \mathbb C\setminus S_1$ and $k_2 \in \mathbb C\setminus S_2$,
$$\mathrm{Com}\left( \pi_{H_1}, \mathcal{W}^{k_1} (\mathfrak{g}_1, f_\mathrm{sub}) \right) \simeq \mathrm{Com}\left( \pi_{H_2}, \mathcal{W}^{k_2} (\mathfrak{g}_2) \right),$$
\item  For $k_1 \in \mathbb C\setminus K_1$ and $k_2 \in \mathbb C\setminus K_2$,
$$\mathcal{W}^{k_1} (\mathfrak{g}_1, f_\mathrm{sub}) \simeq \mathrm{Com}\left( \pi_{\widetilde{H}_2}, \mathcal{W}^{k_2} (\mathfrak{g}_2) \otimes V_{\Z\sqrt{-1}} \right),$$
\item  For $k_1 \in \C\setminus K_1$ and $k_2 \in \mathbb C\setminus K_2$,
$$\mathcal{W}^{k_2} (\mathfrak{g}_2) \simeq \mathrm{Com}\left( \pi_{\widetilde{H}_1},  \mathcal{W}^{k_1} (\mathfrak{g}_1, f_\mathrm{sub}) \otimes V_\Z \right).$$
\end{enumerate}
\end{corollary}
\proof
Notice that $\mathcal{W}^{k_1} (\mathfrak{g}_1, f_\mathrm{sub})$ is a good vertex algebra in the sense of \cite[Definition 6.1]{CL3} with the structure map $\pi_{H_1} \hookrightarrow \mathcal{W}^{k_1} (\mathfrak{g}_1, f_\mathrm{sub})$. Since the Fock modules of $\pi_{H_1}$ are simple unless $\pi_{H_1}$ is degenerate, \cite[Lemma 6.3]{CL3} holds, for all $k_1$ such that $k_1 \in \mathbb C\setminus S_1$. Using the same proof as in \cite[Section 6]{CL3}, it follows that the graded character of $\mathrm{Com}\left( \pi_{H_1}, \mathcal{W}^{k_1} (\mathfrak{g}_1, f_\mathrm{sub}) \right)$ is independent of $k_1$, for all $k_1$ such that $k_1 \in \mathbb C\setminus S_1$. In the same way, one can show that the graded characters of the Heisenberg cosets of vertex (super)algebras appearing in the assertions are independent of $k_1$ (resp. $k_2$) in the range given in each assertions. 

(1) First, note that the pair $(x_1,x_2)$ in \eqref{eq: second critical level} satisfies the relation \eqref{eq: ell}. Recall that non-degenerate Heisenberg vertex algebras are all isomorphic if and only if their ranks are equal and that they have the conformal gradings by the Segal-Sugawara vectors with each homogeneous subspace is of finite dimension. 
Next, note that the excluded level $(k_1,k_2)=(x_1,x_2)$ is exactly when the Heisenberg vertex algebras $\pi_{H_1}$ $\pi_{H_2}$ degenerate.
Then $\{\mathrm{Com}(\pi_{H_1},\mathcal{W}^{k_1}(\mathfrak{g}_1,f_{\mathrm{sub}}))\}_{k_1\in \C\setminus S_1}$ is a continuous family of vertex algebras inside a non-degenerate Heisenberg vertex algebra of rank $n+1$ by \eqref{eq: Upsilon} and so is $\{\mathrm{Com}(\pi_{H_2},\mathcal{W}^{k_2}(\mathfrak{g}_2)\}_{k_2\in \C\setminus S_2}$ by \eqref{eq: Psi}. They are isomorphic if a generic level $(k_1,k_2)$ with \eqref{eq: ell}, i.e., all values for $k_1\in \C\setminus S_1$, (equivalently $k_2\in \C\setminus S_2$ by \eqref{eq: ell}), except for countably many values. Thus we may apply Lemma \ref{continuity} with $\C$ replaced by $\C\setminus S_1$. This completes the proof. (2) and (3) are proved in the same way. 
\endproof
Since the $\mathcal{W}$-superalgebras $\mathcal{W}^{k_1}(\mathfrak{g_1},f_{\mathrm{sub}})$ and $\mathcal{W}^{k_2}(\mathfrak{g}_2)$ as above are of CFT type, their simple quotients exist uniquely, which we denote by $\mathcal{W}_{k_1}(\mathfrak{g_1},f_{\mathrm{sub}})$ and $\mathcal{W}_{k_2}(\mathfrak{g}_2)$ respectively. It follows from the construction \cite{KRW,KW} of $\mathcal{W}$-superalgebras that the Heisenberg subalgebras $\pi_{H_1}\subset \mathcal{W}^{k_1}(\mathfrak{g_1},f_{\mathrm{sub}})$ and $\pi_{H_2}\subset \mathcal{W}^{k_2}(\mathfrak{g_2})$ act semi-simply. Then we see that  the Heisenberg vertex subalgebras $ \pi_{\widetilde{H}_1}\subset \mathcal{W}_{k_1} (\mathfrak{g}_1, f_\mathrm{sub}) \otimes V_\Z$ and $\pi_{\widetilde{H}_2}\subset \mathcal{W}_{k_2} (\mathfrak{g}_2) \otimes V_{\Z\sqrt{-1}}$ act semi-simply.
By applying Corollary \ref{Heisenberg coset} to the above Corollary, we obtain the following. 
\begin{corollary}\label{isom: simple quotients}
Let $(\mathfrak{g}_1,\mathfrak{g}_2)=(\mathfrak{sl}_{n+1}, \mathfrak{sl}(1|n+1))$, or $(\mathfrak{so}_{2n+1}, \mathfrak{osp}(2|2n))$, and $(k_1,k_2)$ satisfy \eqref{eq: ell}. Then the following isomorphisms hold:
\begin{enumerate}
\item For $k_1 \in \mathbb C\setminus S_1$ and $k_2 \in \mathbb C\setminus S_2$,
$$\mathrm{Com}\left( \pi_{H_1}, \mathcal{W}_{k_1} (\mathfrak{g}_1, f_\mathrm{sub}) \right) \simeq \mathrm{Com}\left( \pi_{H_2}, \mathcal{W}_{k_2} (\mathfrak{g}_2) \right)$$
 \item  For $k_1 \in \mathbb C\setminus K_1$ and $k_2 \in \mathbb C\setminus K_2$,
$$\mathcal{W}_{k_1} (\mathfrak{g}_1, f_\mathrm{sub}) \simeq \mathrm{Com}\left( \pi_{\widetilde{H}_2}, \mathcal{W}_{k_2} (\mathfrak{g}_2) \otimes V_{\Z\sqrt{-1}} \right),$$
\item  For $k_1 \in \mathbb C\setminus K_1$ and $k_2 \in \mathbb C\setminus K_2$,
$$\mathcal{W}_{k_2} (\mathfrak{g}_2) \simeq \mathrm{Com}\left( \pi_{\widetilde{H}_1},  \mathcal{W}_{k_1} (\mathfrak{g}_1, f_\mathrm{sub}) \otimes V_\Z \right).$$
\end{enumerate}
\end{corollary}
The following is a special case of a $C_2$-cofiniteness result \cite[Theorem 5.9.1]{Ar3} and a rationality result \cite[Theorem 9.4]{AvE} of  simple $\mathcal{W}$-algebras at the admissible levels. 
\begin{theorem}[\cite{Ar3,AvE}]\label{thm: c2 and rationality for W}${}$
\begin{enumerate}
\item 
$\mathcal{W}_k (\mathfrak{sl}_{n+1}, f_\mathrm{sub})$ is $C_2$-cofinite  and rational for 
\begin{align}\label{admissible level 1}
k=-(n+1) + \frac{u}{n},\quad u\in \mathbb Z_{> n},\ (u, n)=1.
\end{align} 
\item 
$\mathcal{W}_k (\mathfrak{so}_{2n+1}, f_\mathrm{sub})$ is $C_2$-cofinite for 
\begin{align}\label{admissible level 2}
k=-(2n-1) + \frac{u}{v},\quad
u\in \mathbb Z_{> v},\ 
v\in\{2n-1, 2n\},\ 
(u, v)=1.
\end{align}
\end{enumerate}
\end{theorem}
\begin{corollary}\label{cor:C2cofiniteness} ${}$
\begin{enumerate}
\item 
$\mathrm{Com}(\pi_{H_1},  \mathcal{W}_k (\mathfrak{sl}_{n+1}, f_\mathrm{sub}))$ is $C_2$-cofinite and rational for \eqref{admissible level 1}.
\item
$\mathrm{Com}(\pi_{H_1},\mathcal{W}_k (\mathfrak{so}_{2n+1}, f_\mathrm{sub}))$ is $C_2$-cofinite for \eqref{admissible level 2}.
\end{enumerate}
\end{corollary}
\proof
(1) follows from Theorem \ref{thm:rationality} and Theorem \ref{thm: c2 and rationality for W} (1). We show (2). First, by Corollary \ref{isom: simple quotients} (2), $\mathcal{W}_k (\mathfrak{so}_{2n+1}, f_\mathrm{sub})=\operatorname{Com}\left( \pi_{\widetilde{H}_2}, \mathcal{W}_{\ell} (\mathfrak{osp}(2|2n)) \otimes V_{\Z\sqrt{-1}} \right)$ with $\ell = -n + \frac{v}{2u}$. Recall that $\widetilde{H}_2 = H_2 + \psi$, where $H_2 \in \mathcal{W}_{\ell}(\mathfrak{osp}(2|2n))$ and $\psi \in  V_{\Z\sqrt{-1}}$, see \eqref{eq: tilde{H}_2}. Then we may apply Corollary \ref{cor:C2} for $W=\mathcal{W}_{\ell} (\mathfrak{osp}(2|2n))$, $\alpha = H_2$ and $\beta = \psi$ since $C_0 = \mathcal{W}_k (\mathfrak{so}_{2n+1}, f_\mathrm{sub})$ is a $C_2$-cofinite vertex operator algebra of CFT-type by Theorem \ref{thm: c2 and rationality for W} (2). Here $q = 1-\frac{v}{u}$ and $r = -1$. As a consequence, $D = \operatorname{Com}(\pi_{H_1},\mathcal{W}_k (\mathfrak{so}_{2n+1}, f_\mathrm{sub}))$ is $C_2$-cofinite. This proves (2).
\endproof
\begin{corollary}\label{cor:rational} ${}$
\begin{enumerate}
\item
Let $\ell=-n +\frac{n}{u}$  with  $u\in \mathbb Z_{> n}$ and $(u, n)=1$. Then $\mathcal{W}_{\ell} (\mathfrak{sl}(1|n+1))$ is $C_2$-cofinite and rational.
\item
Let $\ell=-n +\frac{v}{2u}$  with  $u\in \mathbb Z_{>v}$, $v\in\{2n-1, 2n\}$ and $(u, v)=1$. Then $\mathcal{W}_{\ell} (\mathfrak{osp}(2|2n))$ is $C_2$-cofinite.
\end{enumerate}
\end{corollary}
\proof
First, we show (1). Since $V_\Z$ is simple, rational, $C_2$-cofinite and of CFT type, so are $\mathcal{W}_k(\mathfrak{sl}_{n+1},f_{\mathrm{sub}})\otimes V_\Z$ with $k$ as \eqref{admissible level 1} and its even part $\left( \mathcal{W}_k ( \mathfrak{sl}_{n+1},f_{\mathrm{sub}} ) \otimes V_\Z \right)_0 = \mathcal{W}_k( \mathfrak{sl}_{n+1},f_{\mathrm{sub}} ) \otimes V_{2\Z}$ by Theorem \ref{thm: c2 and rationality for W} (1). Since the Heisenberg subalgebra $\pi_{\widetilde{H}_1}$ acts semisimply, (1) follows from Corollary \ref{remark:super} and Corollary \ref{isom: simple quotients} (3).

Next, we show (2). Let $C=\operatorname{Com}\bigl( \pi_{H_2},\mathcal{W}_\ell(\mathfrak{osp}(2|2n))\bigr)$. Then $C$ is isomorphic to $\operatorname{Com}\bigl( \pi_{H_1}, \mathcal{W}_{k} (\mathfrak{so}_{2n+1}, f_\mathrm{sub}) \bigr)$ with $k$ as \eqref{admissible level 2} by Corollary \ref{isom: simple quotients} (1), and is $C_2$-cofinite by Corollary \ref{cor:C2cofiniteness} (2). By Lemma \ref{lemma:C2}, the double coset $\operatorname{Com}\bigl( C,\mathcal{W}_\ell(\mathfrak{osp}(2|2n)) \bigr)$ is a rank one lattice vertex superalgebra $V_N$ with some lattice $N$ that is positive definite since $\operatorname{Com}\bigl( C,\mathcal{W}_\ell(\mathfrak{osp}(2|2n)) \bigr)$ is of CFT-type. Therefore $\mathcal{W}_\ell(\mathfrak{osp}(2|2n))$ is $C_2$-cofinite by Lemma \ref{lemma:C2}.

\endproof

In order to prove rationality one needs some new technology and rationality proofs in the ortho-symplectic case will appear in \cite{CL5}.

\end{document}